\documentclass[12pt]{amsart}

\usepackage{graphicx}
\usepackage[utf8]{inputenc}
\usepackage[english]{babel}
\usepackage{csquotes}
\usepackage{amssymb}
\usepackage{tikz-cd}
\usepackage{csquotes}

\usepackage{amsmath,amsfonts,pictex,graphicx,fullpage,etex,tikz,hyperref}

\usepackage{todonotes}

\usepackage{enumerate}
\usepackage{enumitem}

\newtheorem{theorem}{Theorem}[section]
\newtheorem{lemma}[theorem]{Lemma}
\newtheorem{proposition}[theorem]{Proposition}
\newtheorem{notation}[theorem]{Notation}
\newtheorem{corollary}[theorem]{Corollary}

\theoremstyle{definition}
\newtheorem{definition}[theorem]{Definition}

\theoremstyle{remark}
\newtheorem{remark}[theorem]{Remark}

\newtheorem*{claim}{Claim}

\numberwithin{equation}{section}

\newcommand{\T}{\mathrm{T}}
\newcommand{\dd}{\mathrm{d}}

\newcommand{\Aut}{\mathrm{Aut}}
\newcommand{\End}{\mathrm{End}}
\newcommand{\id}{\mathrm{Id}}
\newcommand{\Z}{\mathbb{Z}}
\newcommand{\s}{\mathcal{S}}
\newcommand{\n}{\mathbf{n}}
\newcommand{\dist}{\mathrm{dist}}
\newcommand{\Q}{\mathbb{Q}}
\newcommand{\R}{\mathbb{R}}
\newcommand{\K}{\mathbb{K}}

\newcommand{\N}{\mathbb{N}}
\newcommand{\Vol}{\mathrm{Vol}}
\newcommand{\GL}{\mathrm{GL}}
\newcommand{\acton}{\curvearrowright}
\newcommand{\Fix}{\mathrm{Fix}}

\newcommand{\holder}{H\"{o}lder }
\newcommand{\smooth}{ C^{\infty}}

\begin{document}

\title{EXPONENTIAL MIXING AND SMOOTH CLASSIFICATION OF COMMUTING EXPANDING MAPS}

\author{RALF SPATZIER $^\ast$}
\address{DEPARTMENT OF MATHEMATICS, UNIVERSITY OF MICHIGAN,
    ANN ARBOR, MI, 48109. U.S.A.}
\email{spatzier@umich.edu}

\author{LEI YANG $^{\ast \ast}$}
\address{EINSTEIN INSTITUTE OF MATHEMATICS, THE HEBREW UNIVERSITY OF JERUSALEM, JERUSALEM, 9190401, ISRAEL.}
\email{yang.lei@mail.huji.ac.il}

\thanks{$^{\ast }$ Supported in part by the NSF grants  DMS-1307164 and an Eisenbud Professorship at MSRI}

\thanks{$^{\ast \ast}$ Supported in part by Postdoctoral Fellowships at MSRI and The Hebrew University of Jerusalem}


\date{}




\begin{abstract}   We show that genuinely higher rank expanding actions of  abelian semi-groups on compact manifolds are 
$C^{\infty}$-conjugate to affine actions on infra-nilmanifolds.   This is based on the classification of expanding diffeomorphisms up to \holder conjugacy by Gromov and Shub, and is similar to recent work on smooth classification of higher rank Anosov actions on tori and nilmanifolds. To prove regularity of the conjugacy in the higher rank setting, we establish exponential mixing of  solenoid actions induced from  semi-group actions by nilmanifold endomorphisms, a result of independent interest.  We then proceed similar to the case of  higher rank Anosov actions.

\end{abstract}

\maketitle


\section{Introduction}


\subsection{Smooth classification of higher rank expanding actions}
\par Let $G$ be a connected and simply connected nilpotent Lie group. Let $\End(G)$ and $\Aut(G)$ denote the semi-group of endomorphisms of $G$ and group of automorphisms of $G$ respectively. Let $\Gamma \subset G$ be a discrete subgroup such that the quotient space $\Gamma\setminus G$ is compact. Then we call the compact manifold $\Gamma\setminus G$ a {\em nilmanifold}. A compact manifold $M$ is called an  {\em infra-nilmanifold} if it admits a finite nilmanifold covering $\Gamma \setminus G$. If $A \in G \rtimes \End(G)$ satisfies that $A (\Gamma ) \subset \Gamma$, then $A$ induces a smooth map on $\Gamma\setminus G$, a so-called  
 {\em affine nilendomorphism} of $\Gamma\setminus G$. The $\Aut(G)$ component of 
$A$ is called the {\em linear part} of $A$. A map on $M$ is called an {\em affine infra-nilendomorphism} if it  lifts to an {\em affine nilendomorphism} of a finite nilmanifold covering $\Gamma \setminus G$.

\par Let $M$ be a compact smooth manifold endowed with a Riemannian metric $\| \cdot \|$.  We say a smooth map $$\tau: M \rightarrow M$$
is {\em expanding} if there exists a constant $c>1$
such that for any $v$ in the tangent bundle $ \T M$, we have 
$$\| D\tau (v)\| \geq c\|v\|.$$
It is natural to 
ask if one can classify all expanding maps,
up to conjugacy. Based on the work of Shub \cite{shub1970}, Gromov \cite{gromov1981} found the best possible answer to this question: 

\noindent he proved that every expanding map on
a compact manifold is topologically conjugate to an {\em affine infra-nil endomorphism.} In other words, for a 
finite cover $\overline{M}$ of $M$, there exists a compact nilmanifold $\Gamma \setminus G$, a homeomorphism $\phi: \overline{M}  \rightarrow \Gamma \setminus G$ and an affine nilendomorphism $A$ on $\Gamma \setminus G$ such that the following diagram
commutes:
$$
\begin{tikzcd}
\overline{M}\arrow{r}{\phi} \arrow{d}{\tau} & \Gamma\setminus G \arrow{d}{A} \\
\overline{M} \arrow{r}{\phi} & \Gamma\setminus G.
\end{tikzcd}$$

Moreover, $\phi$ is  bi-H\"older. The linear part of $A= \phi \circ \tau \circ \phi^{-1}$ is given by 
the  induced action   of $\tau$ on the fundamental group $\Gamma$ of $M$. We remark that any  expanding map has at least one  fixed point $p \in M$ \cite[Theorem 1]{Shub1969} so that the induced map on $\pi _1 (M,p)$ is well defined. Finally, if $M =  \Gamma\setminus G $ then $\phi$ can be chosen to be homotopic to the identity $\id$. We remark that Dekimpe clarified some of the algebraic issues with the notions of affine nilendomorphisms and the proof of the Gromov and Shub classification result  \cite{dekimpe2012}.

We remark that in dimensions at least 5, passing to a finite cover $\overline{M}$ of $M$ if necessary, we can further assume that $\overline{M}$ is actually diffeomorphic to a nilmanifold $\Gamma \setminus G$ .   Indeed, 
  exotic differentiable structures on nilmanifolds  always become standard on a finite cover - cf. the Appendix by J. Davis in \cite{kalinin_fisher_spatzier2013}.


\par In general, an expanding map is not $C^1$-conjugate 
to an affine infra-nilendomorphism. One can construct simple examples by perturbing a suitable affine example locally at a fixed point, changing the derivative at the fixed point.   Furthermore, Farrell and Jones constructed expanding maps on  tori with exotic differentiable structures \cite{farrelljones1978}. For higher rank $\Z^k_{+}$ actions ($k\geq 2$),  the situation changes dramatically.  Note that higher rank is needed as one can always take product actions of individual non-algebraic expanding maps. We also have to avoid finite symmetries that disguise a product of rank one actions. Hence we make the following definition.

\begin{definition}
\label{def_factor_rank}
Let $\rho$ be a $C^{\infty}$ $\Z^{k}_{+}$ ($k\geq 2$) action on a manifold $M$. We call $\rho$ 
 {\em genuinely higher rank} if for all finite index sub-semigroups $Z$ of
$\Z^k$, no continuous quotient of any finite extension of the $Z$-action factors through a finite extension of 
a $\Z _+$ action.
\end{definition}
\par It is easy to show that after passing to a finite index sub-semigroup and a finite cover, a $\Z^k _+$ action $\rho$ with an expanding map is $C^0$ conjugate to an affine action $\rho_l$ on a nilmanifold via the conjugacy $\phi$ we get from a single expanding map, cf. Lemma \ref{lemma-finite-cover-reduction}. We will use genuine higher rank in this paper to show that the \holder conjugacy $\phi$ is actually $C^{\infty}$.

\par We summarize the main result of this paper as follows:
\begin{theorem}
\label{goal_thm}
Let $\rho$ be a $C^{\infty}$ $\Z^k_{+}$ ($k\geq 2$) action on a compact manifold $M$. 
Suppose that $\rho$ is {\em genuinely higher rank} and contains an expanding map $\rho(a)$, for some $a \in \Z^{k}_{+}$. Then $M$ is  diffeomorphic to an infra-nilmanifold $\mathcal{M}$ and $\rho$ is $C^{\infty}$ conjugate to a $\Z^k_+$ action $\rho_l$ on $\mathcal{M}$ by affine nil-endomorphisms.
\end{theorem}


\subsection{Related results}
\par Rigidity of higher rank actions on compact manifolds has been studied in different contexts. For Anosov actions, Rodriguez Hertz \cite{hertz2007} classified $\Z^k$, $k \geq 2$, actions on tori containing one 
Anosov element and satisfying certain additional conditions. His work required that the rank $k$ of the action is comparable to the dimension of the torus. Kalinin, Fisher and Spatzier \cite{kalinin_fisher_spatzier2013} proved that if a 
$\Z^k$ action $\alpha$ on a torus or nilmanifold is genuinely higher rank and contains ``many'' Anosov elements, then it is $C^{\infty}$ conjugate to affine actions. Later Rodriguez Hertz and Wang 
\cite{hertz_wang2014} obtained the optimal result for Anosov actions on nilmanifolds  by showing that existence of a single Anosov element implies  existence of  ``many'' Anosov elements. Here by saying ``many'' Anosov elements we mean that there is an Anosov element in each so-called  Weyl chamber of Lyapunov exponents of $\alpha$. We refer the reader to the introduction of \cite{kalinin_fisher_spatzier2011} and to   \cite{brinprize2015} for  brief surveys of  results and methods in the classification of higher rank Anosov actions.


\subsection{Exponential mixing of nilendomorphisms, expanding maps and their  solenoids}

\par We will apply techniques similar to those in \cite{kalinin_fisher_spatzier2013} and \cite{hertz_wang2014}
to show the conjugacy $\phi$ is $\smooth$. The first difficulty here is that these actions are not invertible, so the Weyl chambers of the Lyapunov exponents of $\rho$ are only indirectly defined, and not in terms of the dynamical behavior (slow exponential growth) of actual elements close to the respective Weyl chamber walls. To overcome this difficulty, we shall extend the action $\rho$ to the solenoid $\s(M)$ of $M$, defined below in detail in \S \ref{solenoid}. Basically one wants to invert a covering map from a  space to itself by considering the space of all possible orbits on which one has a tautological inverse.  As the future orbits are well defined this becomes the space of pasts.  One can easily generalize this construction to semigroups generated by commuting covering maps.  As discussed in  \cite{katok_spatzier1996}, the solenoid has a completely algebraic description in terms of a p-adification of the space.  As it turns out, we will actually never need the original notion of the space of pasts and will work directly with the algebraic definition.

The notion of solenoid was used by Williams \cite{williams1974} to study expanding attractors, and also by Katok and Spatzier in \cite{katok_spatzier1996} to prove measure rigidity statements. This makes $\rho$ (and $\rho_l$) a partially hyperbolic $\Z^k$ action on $\s(M)$.  According to \cite{katok_spatzier1996}, the Lyapunov exponents and Weyl chambers of both $\rho$ and $\rho_l$ are well defined. This allows us to proceed as in the first step of the proof of global rigidity for Anosov actions.    

\par The crucial next step in the proof of  Theorem \ref{goal_thm} is to prove an exponential mixing result for the action of $\rho$ on the solenoid $\s(M)$.  This is the main novelty in this paper, and has independent interest.  We actually prove such mixing quite generally for semigroups of endomorphisms of nilmanifolds.  Indeed,  passing to a finite cover, we may identify $M$ as a nilmanifold $\Gamma \setminus G$. 

By the theory of nilpotent Lie groups, there exists a nilpotent algebraic group $N$ over $\Q$ such that the nilpotent Lie group $G = N(\R)$ and $\Gamma = N(\Z)$. We will see later that the solenoid $\s(M)$ of $M$ can be identified with an $S$-adic nilmanifold $N(\Z) \setminus N(\R) \times \prod_{p \in S} N(\Z_p)$ for a finite subset $S$ of primes. Let $\mu$ denote the probability measure on $\s(M)$ induced by the Haar measure on $N(\R) \times \prod_{p \in S} N(\Z_p)$. Note that $\rho_l$ is a $\Z^k$ action on $\s(M)$ by nilendomorphisms.

\par Our main exponential mixing result is the following:
\begin{theorem}
\label{thm:exponential-mixing-intro}
 Let $M = N(\Z) \setminus N(\R)$ and $\s(M) = N(\Z) \setminus N(\R) \times \prod_{p \in S} N(\Z_p)$. Let $C^{\theta}(M)$ and $C^{\theta}(\s(M))$ denote the space of $\theta$-\holder functions defined on $M$ and $\s(M)$ respectively. Let $\|\cdot\|_{\theta}$ denote the $\theta$-\holder norm.
  Let $\rho_l$ denote a $\Z^k$ action on $\s(M)$ by nilendormorphisms. Suppose that $\rho_l(a)$ acts ergodically on $\s(M)$ for every $a \in \Z^k$. Then there exist constants $a_1 >0$ and $\eta'>0$ depending on $\theta$, such that for any $a \in \Z^k$, any $f \in C^{\theta}(M)$, regarded as a function on $\s(M)$, and any $g \in C^{\theta}(\s(M))$, we have 
\begin{equation} \label{equ:exponential-mixing}\int_{\s(M)} f(\rho_l(a)\overline{z}) g(\overline{z}) \dd \mu(\overline{z}) = \int_{\s(M)} f \dd \mu  \int_{\s(M)} g \dd \mu  +O(e^{-\eta'\|a\| } \|f\|_{\theta} \|g\|_{\theta}),\end{equation}
where $\|\cdot\|$ denotes the supremum norm on $\Z^k$.
\end{theorem}

\par In the case of $\Z^k$ actions by ergodic automorphisms on nilmanifolds, exponential mixing was established by Gorodnik and Spatzier \cite{gorodnik_spatzier2014} and \cite{gorodnik_spatzier_acta}, based on the work of Green and Tao \cite{green_tao2012,GreenTao-erratum}. In our case, the structure of $\s(M)$ is essentially different from a real nilmanifold. Since $\s(M)$ is an $S$-adic nilmanifold, $p$-adic analysis will play an important role in the proof. We will prove the theorem in \S \ref{exponential_mixing}.
\par In \S \ref{exponential_mixing}, we will show that if the action $\rho_l$ is {\em genuinely higher rank}, then there exists a subgroup $\Sigma \subset \Z^k$ isomorphic to $\Z^2$ such that every element in $\Sigma$ acts ergodically on $\s(M)$. Therefore, under the {\em genuinely higher rank} hypothesis, we can choose a subgroup $\Sigma \cong \Z^2$ of $\Z^k$ such that the above exponential mixing result holds for $\Sigma$. Note that $\rho$ is conjugate to $\rho_l$ via a \holder homeomorphism $\phi$, if Theorem \ref{thm:exponential-mixing-intro} holds for $\rho_l$ and $\mu$, then it will also hold for $\rho$ and $\tilde{\mu}  := \phi^{-1}_{\ast} \mu$. Therefore, Theorem \ref{thm:exponential-mixing-intro} implies the following corollary which is crucial to establish the smoothness of $\phi$:

\begin{corollary}
\label{cor:exponential-mixing}
Let $\rho$, $M$, $\s(M)$, $\tilde{\mu}$, $C^{\theta}(M)$, $C^{\theta}(\s(M))$ and $\|\cdot\|_{\theta}$ be as above. Suppose the action $\rho$ is {\em genuinely higher rank,} then there exists a subgroup $\Sigma$ of $\Z^k$ isomorphic to $\Z^2$ and constants $a_1 >0$ and $\eta' >0$, such that the following holds: for any $f \in C^{\theta}(M)$, regarded as functions on $\s(M)$, any $g \in C^{\theta}(\s(M))$ and any $a \in \Sigma$,
$$\left|\int_{\s(M)} f(\rho(a)\overline{z}) g(\overline{z}) \dd \tilde{\mu} (\overline{z}) - \int_{\s(M)} f \dd \tilde{\mu} \int_{\s(M)} g \dd \tilde{\mu} \right| \leq a_1 e^{-\eta' \|a\|} \|f\|_{\theta} \|g\|_{\theta}.$$
  
\end{corollary}

\par Let us emphasize that our exponential mixing results are  different from exponential mixing for just the semi-group.  Indeed we can go to infinity in the solenoid in a variety of ways, e.g.  by going back far in the past and returning to the present.   Exponential mixing for just the future of an expanding map follows from the  standard techniques of Markov sections and transfer operators. These techniques however are not  able to handle our case.   In addition, we allow for quite general semi-groups of endomorphisms of nilmanifolds, not just expanding and hyperbolic ones. Quite generally, there are now several techniques available to prove exponential mixing: Fourier analysis, representation theory,  Markov systems and transfer operators especially in combination with contact structures.  However, for one reason or another, none of these work generally for semi-groups of nilmanifold endomorphisms.



\par Finally, let us  note three more   corollaries of exponential mixing, similar to results in \cite{gorodnik_spatzier2014}.  We  refer there for a more extensive discussion of ideas and background.  The proofs are identical, and  we will not discuss them here in detail. 

First  consider a single ergodic nilendomorphism   $\alpha$ on a nilmanifold $X$.
For a function $f: X \to \R$, we  set 
$$
S_n (f,x)=\sum _{i=0} ^{n-1} f (\alpha ^i (x)),
$$
and for simplicity assume that $\int_X f\,d\mu=0$.

One says that  the sequence $f\circ \alpha^n$ satisfies the {\it central limit theorem}
if for some $\sigma>0$, $n^{-1/2}S_n(f,\cdot)$ converges in distribution
to the normal law with mean $0$ and variance $\sigma^2$.
More generally, the sequence $f\circ \alpha^n$ satisfies the
{\it central limit theorem for subsequences}
if there exists $\sigma >0$ such that for every increasing sequence of measurable functions $k_n (x) $
taking values in $\N$ such that for almost all $x$, $\lim_{n\to\infty} \frac{k_n(x)}{n} =c$ for some
fixed constant $0<c < \infty$,    the sequence  $n^{-1/2} S_{k_n(\cdot)} (f,\cdot)$ converges
  in distribution to the normal law  with mean $0$ and variance $\sigma^2/c$.
We define $S_t(f,x)$ for all $t\ge 0$ by linear interpolation of its values at integral points.
The sequence $f\circ \alpha^n$ satisfies the {\it Donsker invariance principle}
if there exists $\sigma>0$ such that
the sequence of random functions $(n\sigma^2)^{-1/2}S_{nt}(f,\cdot)\in C([0,1])$ converges in distribution
to the standard Brownian motion in $C([0,1])$.
The sequence $f\circ \alpha^n$ satisfies the {\it Strassen invariance principle}
if there exists $\sigma>0$ such that for almost every $x$, the sequence of functions
$(2n\sigma^2\log\log n)^{-1/2}S_{nt}(f,x)$ 
is relatively compact in $C([0,1])$ and its limit set is precisely 
the set of absolutely continuous functions $g$ on $[0, 1]$ such that
$g (0)  = 	0$	and $\int_0^1 g'(t)^2\,  dt \leq 1$.
This is a strong version of the law of the iterated logarithm.

\begin{corollary}  \label{affine central limit}
Let $\alpha $ be an  ergodic endomorphism of a compact nilmanifold $X$,
and let $f$ be a H\"older function on $X$ which has zero integral.
\begin{enumerate}
\item If $f$ is not a measurable coboundary.
then the sequence $\{f\circ\alpha^n\}$ satisfies 
the central limit theorem, the central limit theorem of subsequences,
and  the Donsker and Strassen invariance principles.
\item If $f$ is a   measurable coboundary then $f$ is an $L^2$-coboundary. Equivalently, the variance $\sigma =0$.  
\end{enumerate}
\end{corollary}

Livsic proved for Anosov diffeomorphisms that a measurable coboundary for a smooth function is automatically smooth. Veech discussed this issue for ergodic  toral automorphisms in \cite{Veech} using sophisticated Fourier analysis.  He also gave  counter examples in the $C^1$-category.  Gorodnik and Spatzier proved the nilautomorphism version of this result in \cite{gorodnik_spatzier2014}.

\begin{corollary}
\label{cor:measurable livsic}
Let $M$ be as above, and let $ \alpha : M \to M$ be a nilendomorphism, ergodic w.r.t. Haar measure.  Suppose $f: M \to \R$ is a $\smooth$ function, $g: M \to \R$  a measurable function such that $f = g - g \circ \alpha$.  Then $g$ is $\smooth$.  
\end{corollary}

For the proof we just extend f and g to functions on the solenoid.  Note that they are independent of the $p$-adic direction.  Hence we can  use the central limit theorem, Corollary \ref{affine central limit},  and exponential mixing as in \cite{gorodnik_spatzier2014}.  

In our last corollary, we consider genuinely higher rank actions.  Again the proof is identical to \cite{gorodnik_spatzier2014}.

\begin{corollary}
\label{cor:cocycle rigidity}
Let $\rho$, $M$, $\s(M)$, $\tilde{\mu}$, $C^{\theta}(M)$, $C^{\theta}(\s(M))$ and $\|\cdot\|_{\theta}$ be as above. Suppose the action $\rho$ is {\em genuinely higher rank}. Then any $\smooth$ cocycle $\alpha: \Z _+ ^k \times M \to \R$ is $\smooth$ cohomologous to a constant cocycle.   \end{corollary}

\subsection{Organization of the paper}
\par 
Before \S \ref{lowdimension}, we always assume that $\dim M \geq 5$. In \S \ref{preliminaries}, we recall the result of Gromov and Shub on bi-H\"older conjugacies between expanding maps and their linearizations, reduce  the  main  theorem  to  the  case  of  actions  H\"older conjugate to nilmanifold endomorphisms, recall the structures of solenoids, and discuss the Lyapunov exponents of the extended $\Z^k$-actions of $\rho$ and $\rho_l$ on the solenoid $\s(M)$. In \S \ref{exponential_mixing}, we will prove Theorem \ref{thm:exponential-mixing-intro} and Corollary \ref{cor:exponential-mixing}.
 In \S \ref{smoothness of leaves}, we apply Corollary \ref{cor:exponential-mixing} and the techniques developed in \cite{hertz_wang2014} to show that every coarse Lyapunov distribution of $\rho$ admits a \holder foliation with $\smooth$ leaves. This result is crucial for applying the techniques developed in \cite{kalinin_fisher_spatzier2013}. In \S \ref{regularity},
we combine the exponential mixing result, the result proved in \S \ref{smoothness of leaves} and the techniques developed in 
\cite{kalinin_fisher_spatzier2013} to prove Theorem \ref{goal_thm} when $\dim M \geq 5$. In \S \ref{lowdimension}, we discuss the case $\dim M \leq 4$.


\begin{notation}
\label{notation}
In this paper, we will use the following conventions. For two quantities $A$ and $B$, $A\ll B$ means that there exists an absolute constant $C >0$ such that $A \leq CB$. $A \gg B$ means that $B \ll A$. $A \asymp B$ means that $A \ll B$ and $B \ll A$. $O(A)$ denotes a quantity $\asymp A$. Given a sequence of quantities $\{A_i>0: i \in \N \}$, another sequence $\{B_i : i \in \N\}$ is said to be of order $o(A_i)$ if $|B_i|/ A_i \rightarrow 0$ as $i \rightarrow \infty$.
\end{notation}


\noindent {\bf Acknowledgements.}  The first author thanks David Fisher and Boris Kalinin for early discussions on these matters.  Both authors thank MSRI where they collaborated on  this work in Spring 2015.

\section{Preliminaries on expanding maps and solenoids}
\label{preliminaries}

 In this section we review and refine some background material on expanding maps and their conjugacies and centralizers.  This allows us to reduce the main theorem  to the case  of actions \holder conjugate to nilmanifold endomorphisms. We then define solenoid extensions and also  Lyapunov exponents.


\subsection{Gromov's conjugacy theorem on expanding maps}
\par We recall the result of Shub \cite{shub1970} and Gromov \cite{gromov1981}:

\begin{theorem}[see~{\cite[Theorem 1]{shub1970}} and~{\cite[\S 1]{gromov1981}}]
\label{gromov_conjugacy_thm}
Suppose $\tau  : M \rightarrow M$ is an expanding $C^1$- map of a compact manifold $M$.  Then there exist
\begin{enumerate}
\item 
an infra-nilmanifold $\mathcal{M}$, finite covers $\overline{M}$ of $M$ and $\overline{\mathcal{M}}$ of $\mathcal{M}$ such that $\overline{\mathcal{M}}= \Gamma \setminus G$ where $G$ denotes a simply connected nilpotent Lie group and $\Gamma$ denotes a lattice of $G$,
\item an expanding affine nilendomorphism $\overline{\tau} _l : \overline{\mathcal{M}} \to \overline{\mathcal{M}}$ which covers  an infra-nilendomorphism $\tau _l$ of  $\mathcal{M}$,
\item a bi-\holder homeomorphism $\phi: \overline{M} \to \overline{\mathcal{M}}$ such that $\overline{\tau} := \phi ^{-1} \circ \overline{\tau} _l \circ \phi $ covers $\tau$
and descends to a homeomorphism $M \to \mathcal{M}$ intertwining $\tau$ and $\tau_l$.
\end{enumerate}
\par Moreover,  the covering map  $\tilde{\tau} _l$ of $\overline{\tau}_l$ on $G$  is the automorphism of $G$ induced by the map $\overline{\tau}_l^{\ast} : \Gamma  \rightarrow \Gamma$. 
\end{theorem}


\subsection{Reduction to nilmanifolds}
\label{subsec-reduction-to-nilmanifolds}

\par  We discuss several basic properties of expanding maps and their  commuting maps.   We use these to reduce the proof of  our main result to the case when the semi-group has a fixed point and when the  linearization of the action of the semigroup action is on a nilmanifold rather than an infra-nilmanifold.  

First we slightly generalize work of Walters from \cite{walters1970}, cf. also \cite{dekimpe2012}.

\begin{proposition}
Suppose the semigroup $\Z_{+}^k$ acts by $\rho$ on a compact manifold $M$ with an expanding map $\rho (a)$.  Then the action is bi-\holder-conjugate to an action by affine infra-nilendomorphisms on an infra-nilmanifold $\mathcal{M}$. 
\end{proposition} 

\begin{proof}   By Theorem \ref{gromov_conjugacy_thm}, $\rho(a)$ is $C^0$-conjugate to an infra-nilendomorphism $\alpha$ by a homeomorphism $\phi$. Any such conjugacy is bi-\holder as is well-known.  For $b \in \Z_{+}^k$, $\beta := \phi \circ  \rho(b) \circ \phi ^{-1}$ is a smooth map of $\mathcal{M}$ that commutes with $\alpha$. Let $\overline{\mathcal{M}}$ denote the finite  nilmanifold cover of $\mathcal{M}$.  Then $\alpha$ lifts to an affine endomorphism $A$ of $\overline{\mathcal{M}}$, and $\beta$ lifts to a homeomorphism $B$ such that $A^{-1} B^{-1} A B$ is an element of the holonomy of $\overline{\mathcal{M}}$ over $\mathcal{M}$ and thus an affine endomorphism.  Thus $B$ and hence $\beta$ are affine by \cite[Corollary 1]{walters1970}.
\end{proof}

\begin{lemma}  
\label{lemma-finite-cover-reduction}
Let $\rho$ be a $\Z^k_+$ action on a compact manifold $M$ with an expanding map $\rho(a)$ which is \holder conjugate to an affine action $\rho _l$ on an infra-nilmanifold $\mathcal{M}$. Then there is a  sub-semigroup $\Sigma^+$ of finite index in $\Z_{+}^k$ which acts on a finite cover $\overline{\mathcal{M}}$ of $\mathcal{M}$ by nil-endomorphisms covering the restriction of the original action to $\Sigma^+$.  
\end{lemma}

\begin{proof} It suffices to prove that the linearization $\rho _l$ lifts since $\rho$ and $\rho _l$ are $C^0$-conjugate.   By \cite[Theorem 1]{Shub1969},  $\rho_l (a)$ has a fixed point $ p \in \mathcal{M}$.  Moreover, the set of fixed points of $\rho_l (a)$ is finite as $\rho_l (a)$ is expanding  and $\mathcal{M}$ is compact.  Hence there is a  sub-semigroup $\Sigma^+$ of finite index in $\Z_{+}^k$ which also fixes $p$.  As $\Sigma^+$ is finitely generated, we can find a finite cover $\overline{\mathcal{M}}$ of $\mathcal{M}$ such that all elements of $\Sigma^+$ lift to $\overline{\mathcal{M}}$ as affine nil-endomorphisms.  Furthermore, as $\Sigma^+$ is finitely generated, we can pick  lifts of generators of $\Sigma^+$ that all fix a given point $\overline{p}$ in the pre-image of $p$.  Since these lifts are determined by their derivative action at $\overline{p}$, we see that all the lifts of the generators of $\Sigma^+$ commute.  Thus they define a lift of the action of $\Sigma^+$ to $\overline{\mathcal{M}}$ which covers the $\Sigma^+$ action on $\mathcal{M}$, as desired.
 \end{proof}






Under our higher rank assumptions on the semi-group actions, we will show that the covering map $\overline{\rho}$ of $\rho$ is $\smooth$-conjugate to $\overline{\rho} _l$ by $\phi$.  This implies that $\rho$ is $\smooth$-conjugate to $\rho _l$, as desired.  Thus we  can always work with the finite covers $\overline{M}$ and $\overline{\mathcal{M}}$ and actions  $\overline{\rho}$ and $\overline{\rho} _l$ which have a common fixed point.  

\par If the dimension $\dim (M) \geq 5$, we can  make further reductions.  Indeed, by Davis' work on exotic differentiable structures on nilmanifolds \cite[Theorem A.0.1, Appendix]{kalinin_fisher_spatzier2013}
 and passing to a  finite  cover,  we  may   assume  the conjugacy is isotopic to a diffeomorphism $\psi : \overline{M} \to \Gamma\setminus G$.   Then  we can conjugate $\rho$  by $\psi$ to a smooth action on the nilmanifold $\Gamma\setminus G$.   We will deal separately with the case $\dim (M) \leq 4$ in \S \ref{lowdimension}.

These reductions  allows us to make the following hypotheses throughout except in \S \ref{lowdimension}.  


\subsection{Standing Assumption}
Henceforth, $M = \Gamma \setminus G$ will denote a compact nilmanifold and $\rho$ will denote  a $\smooth$ {\em genuinely higher rank} action of a semigroup $\Z _+ ^k$ with $k \geq 2$ on $M$ such that 
\begin{itemize}
\item $\rho (a)$ is an expanding map for some $a \in \Z _+ ^k$,
\item $\rho$ is \holder conjugate to an action of $\Z _+ ^k$ by affine nil-endomorphisms on $M$,
\item $\rho$ has a common fixed point.
\end{itemize}

Note that the linearization of $\rho$ is given by the induced action on the fundamental group $\Gamma$,  thanks  to  existence of a common fixed point.


\subsection{Solenoids and extended actions}   \label{solenoid}
We will define the solenoid $\s(M)$ of $M$, extend $\rho$ and $\rho_l$ to $\Z^k$ actions on $\s(M)$, and define Lyapunov exponents of $\rho$ and $\rho_l$ on $\s(M)$, following \cite{katok_spatzier1996}.
\par First recall Mal'cev's theorem from the  theory of nilpotent Lie groups (see \cite{raghunathan1972,green_tao2012,corwin-greenleaf}, for example) that any lattice $\Gamma$ of a nilpotent Lie group $G$ must be
arithmetic, i.e., there is a simply connected nilpotent algebraic group $N$ over $\Q$ such that $G = N(\R)$ and $\Gamma =N(\Z)$. Then $M = N(\Z)\setminus N(\R)$.
 \par By \cite{katok_spatzier1996}, the abstract solenoid $\s(M, \rho)$ of $M$ is naturally defined as follows: 
\begin{equation}
\label{equ:abstract-solenoid}
\s(M, \rho):= \left\{(z_{\n}) \in M^{\Z_{\leq 0}^k}: z_{\n +a} = \rho(a) z_{\n}: \text{ for all } a \in \Z^k_{+}\right\}.
\end{equation}
In other words, we attach each point on $M$ with all possible pasts with respect to all $a \in \Z^k_+$. On this space, one can easily define a $\Z^k$ action which extends the original $\Z^k_+$ action $\rho$ (see \cite[\S 3]{katok_spatzier1996} for details). The disadvantage is that it is hard to do concrete analysis and calculation with this definition. Therefore, we will give another definition and stick with it throughout the paper. 
\par Given $a \in \Z^k_+$, $\rho(a)$ can be extended to a homeomorphism from $N(\R)$ to itself (cf. \cite{Shub1969} and \cite{shub1970}). For a fixed $z \in M$, the preimage of $z$ with respect to $\rho(a)$ is $\{ \rho^{-1}(a)(n z) : n \in N(\Z) \}$. Therefore, to attach $z$ with a past with respect to $\rho(a)$ is the same as to attach $z$ with an element $n \in N(\Z)$. Moreover, if $\rho_l^{-1}(a)(n_1^{-1} n_2) \in N(\Z)$, then $\rho^{-1}(a)(n_1 z) = \rho^{-1}(a)(n_2 z)$. Taking this congruence condition into account and passing to the inverse limit for all possible $a \in \Z^k_+$, we will attach each $z \in M$ with several $p$-adic components $\xi_p \in N(\Z_p)$.

\par This discussion brings us the new definition of the solenoid $\s(M)$ of $M$: 

\begin{definition}[see~{\cite[Appendix]{katok_spatzier1996}}]
\label{alter_def_solenoid} \par $\quad$
 Let 
$$\s'(M):= N(\Z)\setminus \left( N(\R) \times \prod_{p \text{ prime }} N(\Z_p) \right),$$
where $N(\Z)$ acts on $N(\R) \times \prod_{p \text{ prime }} N(\Z_p)$ diagonally, and in the product, $p$ runs over all primes. For each 
prime number $p$, we define 
$$M_p := \left\{v \in N(\Z_p): \|\rho_{\ast}(a)(v)\|_p = \|v\|_p, \text{ for all } a \in \Z^k_{+}\right\},$$
where $\|\cdot\|_p$ denotes the $p$-adic norm. Then $\s(M)$ is defined as follows:
$$\s(M) := N(\Z)\setminus \left( N(\R) \times \prod_{p \text{ prime }} N(\Z_p)/M_p \right).$$
Equip $\s(M)$ with product structure.

\end{definition}
\begin{remark}
\label{rmk:solenoid}
\par $\quad$
\begin{enumerate}
\item $M_p = N(\Z_p)$ for all but finitely many $p$'s. Therefore, there exists a finite set $S$ of primes such that $\s(M) = N(\Z) \setminus \left( N(\R) \times \prod_{p \in S} N(\Z_p)/M_p \right)$.
\item Generalizing the argument from \cite[Lemma 8.2]{katok_spatzier1996}, we see that every quotient 
$N(\Z_p)/M_p$ is torsion free.  


\item For each prime $p$, we denote by $\nu_p$ the Haar measure on 
$N(\Z_p)$. By normalization, we assume that $\nu_p(N(\Z_p)) =1$. Let 
$\nu $ denote the Haar measure on $N(\R)$ and also the induced measure on 
$N(\Z) \setminus N(\R)$. By normalization, we assume that $\nu(N(\Z)\setminus N(\R)) =1$. Then the product measure $\nu \times \prod_{p \in S} \nu_p$ induces a probability measure on the solenoid $\s(M)$, which we denote by $\mu$. Define $\tilde{\nu} := \phi^{-1}_{\ast}(\nu)$ and $\tilde{\mu} := \tilde{\nu} \times \prod_{p\in S} \nu_p$, then $\tilde{\nu}$ is absolutely continuous with respect to $\nu$ (since $\phi$ is \holder), and $\tilde{\mu}$ is preserved by the action of $\rho$. Moreover, for any $a \in \Z^k$, the action of $\rho(a)$ is ergodic with respect to $\tilde{\mu}$ if and only if $\rho_l(a)$ is ergodic with respect to $\mu$. 
\item The definition above depends on the homotopy class of the action $\rho$ as the $M_p$'s do. Since throughout this paper we fix the homotopy type, i.e., the induced action $\rho_{\ast}$ on $N(\Z)$, we may regard $\s(M)$ as a fixed space.
\item We note that that  an infra-nilmanifold $M$ can be regarded as a finite index factor of a homogeneous space $N(\Z)\setminus N(\R)$ where $N$ denotes a nilpotent $\Q$-group. Once we define the solenoid $\s(N(\Z)\setminus N(\R))$ of 
$N(\Z)\setminus N(\R)$, the solenoid of $M$ is just the quotient of 
$\s(N(\Z)\setminus N(\R))$ by a finite group action. Thus solenoids for infra-nilmanifolds also have an explicit description. 
\end{enumerate}
\end{remark}

\par For nilpotent algebraic group $N$, we have the following version of Chinese remainder theorem:

\begin{lemma}
\label{lemma_chinese_remainder}
Given a finite subset of primes $S$, $\xi_p \in N(\Z_p)$ for $p \in S$ and $l_p \in \Z_{+}$ for $p \in S$, there exists $n \in N(\Z)$ such that
$$n^{-1}\xi_p \equiv \mathbf{0}\quad (\mathrm{mod} \quad p^{l_p})$$ 
for all $p \in S$.
\end{lemma}
\begin{proof}
We prove the statement by induction on the nilpotency degree of $N$.
\par If $N$ is abelian, this is just the Chinese remainder theorem as the action $n^{-1}\xi_p$ is a linear expression.
\par Suppose the statement holds if the nilpotency degree is $< d$. Now we assume that the nilpotency degree of $N$ is $d$.
 Take a nilpotent subgroup $N' $ of $N$ such that $N'\setminus N$ is abelian and the nilpotency degree of $N'$ is $d -1$. Then considering the image of $\xi_p$ on $N'(\Z_p)\setminus N(\Z_p)$, and applying the Chinese remainder theorem, we have that there exists $n_1 \in N(\Z)$ such that 
$$n_1^{-1}\xi_p  \equiv \mathbf{0} \quad (\mathrm{mod} \quad p^{l_p}) \in  N'(\Z_p)\setminus N(\Z_p),$$
for all $p \in S$. Now by applying inductive hypothesis to $N'$, we have there exists $n_2 \in N'(\Z)$ such that 
$$n_2^{-1} n_1^{-1} \xi_p \equiv \mathbf{0} \quad (\mathrm{mod}\quad p^{l_p}),$$
for all $p \in S$. Then $n = n_1 n_2$ satisfies our condition.
\par This proves the lemma.
\end{proof}

\par  The informal discussion before Definition \ref{alter_def_solenoid} may help with the next result and its proof. 
 
\begin{proposition}
$\rho$ and $\rho_l$ can be extended to $\Z^k$ actions on $\s(M)$.
\end{proposition}

\begin{proof} 
\par For $a \in \Z^k_{+}$, the action $\rho(a)$ on $\s(M)$ can be naturally defined as follows: for $\overline{z} = (z, (\xi_p)_{p \in S}) \in \s(M)$,
 $$\rho(a)(\overline{z}) :=(\rho(a)z, (\rho_{\ast}(a)\xi_p)_{p \in S}).$$
\par To extend $\rho$ to a $\Z^k$ action on $\s(M)$, it suffices to define the inverse of $\rho(a)$ for each $a \in \Z^k_+$. For $a \in \Z^{k}_{+}$, $\rho(a)$ can be extended to a homeomorphism of the universal covering $N(\R)$ of $M$ to itself (cf. \cite{Shub1969} and \cite{shub1970}). Therefore $\rho^{-1}(a)$ is well defined on $N(\R)$. Recall that $\rho_{\ast}(a)$ agrees with $\rho_l(a)$ when restricted to $N(\Z)$. Then we  define $\rho^{-1}(a)$ as follows: for $\overline{z}= (z, (\xi_p)_{p \in S}) \in \s(M)$, we may pick $l_p$ for each $p \in S$ such that $\rho_{\ast}^{-1}(a)(N(p^{l_p} \Z_p)) \subset N(\Z_p)$.  
By Lemma \ref{lemma_chinese_remainder}, we can find $n \in N(\Z)$ such that $n^{-1}\xi_p \equiv \mathbf{0}(\mathrm{mod} \quad p^{l_p})$. Let us write $(z, (\xi_p)_{p \in S}) = (n^{-1}z, (n^{-1}\xi_p)_{p \in S})$, then $\rho_{\ast}^{-1}(a)$ is well defined on each $p$-adic component. Thus, we can define
 $$\rho^{-1}(a)(z, (\xi_p)_{p \in S}) := (\rho^{-1}(a)(n^{-1}z) , (\rho^{-1}_{\ast}(a)(n^{-1}\xi_p))_{p \in S}).$$ 
\par The same extension works for $\rho_l$ as well.
\end{proof} 

\par The conjugacy $\phi : M \rightarrow M$ can be extended to a homeomorphism $\phi : \s(M) \rightarrow \s(M)$ as follows: on the real component, it is $\phi$, and on $p$-adic components, it is the identity map. It is easy to see that $\phi$ conjugates the extended actions $\rho$ and $\rho_l$ on $\s(M)$.

\subsection{Lyapunov exponents and coarse Lyapunov decomposition} We need the following notation to define Lyapunov exponents of $\rho$ and $\rho_l$.
\begin{definition}
\label{def:manifold-slice}
\par For $\overline{z} = (z, (z_p)_{p \in S}) \in \s(M)$ and an open neighborhood $U_z \subset \s(M)$ of $\overline{z}$, let $M(\overline{z})$ denote the connected component of $\overline{z}$ in $U_z$. It is easy to see that $M(\overline{z})$ is of form $\{(y, (z_p)_{p \in S}) : y \in U \subset M\}$, where $U \subset M$ denotes an open neighborhood of $z$ in $M$. It is homeomorphic to an open set of $M$. Let us call $M(\overline{z})$ a {\em manifold slice} passing through $\overline{z}$.
\end{definition}

The smoothness of a map defined on $\s(M)$ is defined as follows:  
\begin{definition}
\label{def:smooth-map-on-solenoids}
We say a continuous map defined on $\s(M)$ is $\smooth$ if it is $\smooth$ when restricted to every manifold slice.
\end{definition}

\begin{remark}
\label{rmk:smooth-zk-action}
It is easy to see that for each $a \in \Z^k$, $\rho(a)$ is $\smooth$.   Indeed, $\rho(a)$ maps every manifold slice $M(\overline{z})$ to another manifold slice 
$M(\overline{z}')$, and when restricted to the manifold slice, $\rho(a)$ is smooth (because the map only depends on the real component).  The same holds for $\rho_l$.
\end{remark}

\begin{definition}
\label{def_lyapunov_exponent}
\par Since $\rho_l$ acts on $\s(M)$ by affine nilendomorphisms, it naturally induces a $\Z^k$ action on the nilpotent Lie group $N(\R)$ by automorphisms, which we still denote by $\rho_l$.  Let $D\rho_l$ denote the action on $\mathfrak{n}(\R)$ induced by $\rho_l$. A character $\chi \in (\R^k)^{\ast}$ is called a real {\bf Lyapunov
exponent} of $\rho_l$ if the real {\bf Lyapunov subspace} corresponding to $\chi$ defined as follows:
$$\sigma^{\chi}:= \left\{v \in \mathfrak{n}(\R): \lim_{\|a\|\rightarrow \infty} \frac{\log \|D\rho_l(a) v\| - \chi(a)}{\|a\|} =0\right\}$$
is nontrivial. 
\par Let $D\rho_l$ denote the action on $\mathfrak{n}(\Q_p)$ induced by $\rho_l$. Note that $\rho$ induces the same action on $p$-adic components, so $D\rho_l$ is also the action induced by $\rho$. A character $\chi \in (\R^k)^{\ast}$ is called a $p$-adic {\bf Lyapunov exponent} of $\rho_l$ (and also $\rho$) if the $p$-adic {\bf Lyapunov subspace} corresponding to $\chi$ defined as follows:
$$\sigma^{\chi}:= \left\{v \in \mathfrak{n}(\Q_p): \lim_{\|a\|\rightarrow \infty} \frac{\log \|D\rho_l(a) v\|_p - \chi(a)}{\|a\|} =0\right\}$$
is nontrivial. 
\par For $\overline{z} \in \s(M)$, let $\T_{\overline{z}}(M)$ denote the tangent space of the {\em manifold slice} passing through $\overline{z}$ based at $\overline{z}$. A character $\chi \in  (\R^k)^{\ast}$ is called a real {\bf Lyapunov exponent} of $\rho$ if for $\tilde{\mu}$-a.e. $\overline{z} \in \s(M)$, the real {\bf Lyapunov distribution} corresponding to $\chi$ defined as follows:
$$ E_{\overline{z}}^{\chi} := \left\{ v \in \T_{\overline{z}}(M): \lim_{\|a\| \rightarrow \infty} \frac{\log \|D\rho(a) (v)\| - \chi(a)}{\|a\|} =0 \right\}$$
is nontrivial.
\end{definition}



\begin{notation}
\label{notation_lyapunov_exponent}
To distinguish the {\bf Lyapunov exponents} of $\rho$ and $\rho_l$, later in this paper, we denote {\bf Lyapunov exponents} of $\rho$ by $\chi, \chi_1, \chi_2, \dots,$ and denote {\bf Lyapunov exponents} of $\rho_l$ by $\chi^l, \chi^l_1, \chi^l_2, \dots$. Since $p$-adic 
{\bf Lyapunov exponents} of $\rho$ and $\rho_l$ coincide, we do not distinguish the above two notions in $p$-adic directions. We say a {\bf Lyapunov exponent} $\chi$ (or $\chi^l$) is of type $\K$ ($\K = \R $ or $\Q_p$) if the corresponding {\bf Lyapunov distribution} (or {\bf Lyapunov subspace}) is in $\mathfrak{n}(\K)$. For $\K= \R$ or $\Q_p$, let $T(\K)$ denote the set of type $\K$. 
\end{notation}

\begin{remark}
\par $\quad$
\begin{enumerate}
\item One can prove that (see \cite{hertz_wang2014}) 
 
 $$\mathfrak{n}(\R) = \bigoplus_{\chi^l \in T(\R)} \sigma^{\chi^l}.$$
 \item Since $\rho$ contains an expanding element, it is ergodic with respect to $\tilde{\mu}$. By Multiplicative Ergodic Theorem, there exist finitely many 
 {\bf Lyapunov exponents } $\chi$'s of type $\R$, a set of full $\tilde{\mu}$-measure set $\mathcal{P}\subset \s(M)$, and a $\rho$-invariant measurable splitting of the bundle $\mathcal{E}(\s(M)) := \bigcup_{\overline{z} \in \s(M)} \T_{\overline{z}} (M) = \bigoplus E^{\chi}$ over $\mathcal{P}$ such that for all $a\in \Z^k$ and $v \in E^{\chi}\setminus\{\mathbf{0}\}$, 
 $$\lim_{n \rightarrow \infty} n^{-1} \log \frac{\|D\rho(na) v\|}{\|v\|} = \chi(a). $$
 We refer to \cite{kalinin_katok2001}, \cite{kalinin_sadovskaya2006}, and \cite{kalinin_fisher_spatzier2013} for details. 
 \end{enumerate}
\end{remark}
\begin{definition}
\label{def_coarse_lyapunov}
For a {\bf Lyapunov exponent} $\chi^l$ of $\rho_l$, we define the {\bf coarse Lyapunov subspace} associated with $\chi^l$ as follows:
$$\sigma^{[\chi^l]} := \bigoplus_{\chi_1^l  = c \chi^l, c > 0} \sigma^{\chi_1^l }.$$
The corrosponding decomposition 
$$\mathfrak{n}(\R) = \bigoplus_{\chi^l \in T(\R)} \sigma^{[\chi^l]}$$
is called the {\bf coarse Lyapunov decomposition} of the real 
component of $\s(M)$.
\par Similarly, for a {\bf Lyapunov exponent} $\chi$ of $\rho$, we define the {\bf coarse Lyapunov distribution} associated with $\chi$ as follows:
\[E^{[\chi]} := \bigoplus_{\chi_1 = c \chi, c>0} E^{\chi_1}.\]
\end{definition}
\begin{remark}
\label{remark_bracket}
\par  One can prove that $[\sigma^{\chi^l_1}, \sigma^{\chi^l_2}] \subset \sigma^{\chi^l_1 + \chi^l_2}$ if $\chi^l_1$ and $\chi^l_2$ are of the same type, cf. \cite[Lemma 2.3]{hertz_wang2014}. Therefore, each $\sigma^{[\chi^l]}$ ($\chi^l$ can be real or $p$-adic) is a Lie subalgebra (of $\mathfrak{n}(\R)$ or $\mathfrak{n}(\Q_p)$).
Let $V^{[\chi^l]}$ denote the corresponding Lie subgroup (of $\mathfrak{n}(\R)$ or $\mathfrak{n}(\Q_p)$), which will be called a {\bf coarse Lyapunov subgroup}.
\end{remark}
\begin{definition}
 We define a {\bf Weyl chamber} of $\rho$ (respectively $\rho_l$) to be a connected component of $\R^k \setminus \bigcup_{\chi} \mathrm{ker}\chi$ (respectively $\R^k \setminus \bigcup_{\chi^l} \mathrm{ker}\chi^l$), and a  {\bf real Weyl chamber} to be a connected component of $\R^k \setminus \bigcup_{\chi \in T(\R) } \mathrm{ker} \chi$ (respectively $\R^k \setminus \bigcup_{\chi^l \in T(\R) } \mathrm{ker} \chi^l$).
\end{definition}

Thanks to the existence of the conjugacy $\phi$, we have the following correspondence between {\bf coarse Lyapunov distributions} of $\rho$ and {\bf coarse Lyapunov subgroups} of $\rho_l$.
\begin{proposition}
\label{lyapunov_exponent_correspondence}
The {\bf coarse Lyapunov distributions} of $\rho$ and the {\bf coarse Lyapunov subgroups} of $\rho_l$ are in one-to-one correspondence to each other. A pair of corresponding 
{\bf coarse Lyapunov distribution} and {\bf coarse Lyapunov subgroup} have the same dimension and positively proportional {\bf coarse Lyapunov exponents}. In consequence, $\rho$ and $\rho _l$ have the same Weyl chambers and the same real Weyl chambers.
\end{proposition}

\begin{proof}
See \cite[Lemma 4.9]{hertz_wang2014} or \cite[Proposition 3.2]{kalinin_fisher_spatzier2013}.
\end{proof}



\subsection{The cohomological equation \uppercase\expandafter{\romannumeral1}}
\label{cohomologicalequation1}

 Recall that $\rho(a)$ and $\rho_l(a)$ are $\smooth$ for any $a \in \Z^k$ (see Definition \ref{def:smooth-map-on-solenoids} and Remark \ref{rmk:smooth-zk-action}).

\par Because the conjugacy $\phi: \s (M) \rightarrow \s (M)$ is homotopic to $\id$, for any $a \in \Z ^k$, 
$\rho_l(a)^{-1} \rho(a)$ is homotopic to $\id$. We write 
$$\rho_l(a)^{-1}\rho(a)(\overline{z}) = \overline{z} Q_a(\overline{z}).$$
where $Q_a(\overline{z}) \in N(\R)\times \prod_{p \in S } N(\Z_p)$. It is easy to see that $Q_a$ is $\smooth$ since both $\rho(a)$ and $\rho_l(a)$ are. Since $\rho(a)$ and $\rho_l(a)$ are identical on $p$-adic components, $Q_a(\overline{z})$ is only nontrivial on the 
real component, so it can be regarded as a $\smooth$ map $Q_a: \s(M) \rightarrow N(\R)$. In fact, for each $p \in S$, let $l_p(a) \geq 0$ be an integer such that $\rho_l(a)(N(p^{l_p(a)} \Z_p)) \subset N(\Z_p)$, then if $\overline{z} = (z, (\xi_p)_{p \in S})$ and $\overline{z}' = (z', (\xi'_p)_{p \in S})$ satisfy that $z = z'$ and $\xi_p \equiv \xi'_p (\mod p^{l_p(a)})$ for each $p \in S$, then $Q_a(\overline{z}) = Q_a(\overline{z}')$. In other words, $Q_a$ can be treated as a $\smooth$ map defined on the finite cover $N(n \Z) \setminus N(\R)$ of $M$ where $n = \prod_{p \in S} p^{l_p(a)}$. Writing $\phi(\overline{z}) = \overline{z} h(\overline{z})$, then $h(\overline{z})$ is also trivial on every $p$-adic component, and moreover, it only depends on the real component. Therefore, we can regard $h$ as a map from $M$ to $N(\R)$. Since $\phi$ conjugates $\rho$ to $\rho_l$, we have that for any $a \in \Z^k$ and any $\overline{z} \in \s(M)$, the following holds:
\begin{equation}
\label{cohomological_equation}
\begin{array}{rl}
\overline{z} h(\overline{z}) & = \phi(\overline{z}) = \rho_l(a)^{-1}\phi(\rho(a)\overline{z})  \\
 & = \rho_l(a)^{-1}(\rho(a)\overline{z} h(\rho(a)\overline{z})) \\
& =   \overline{z} Q_a(\overline{z}) \rho_l(a)^{-1} (h (\rho(a)\overline{z})).
\end{array}
\end{equation}
In \cite{hertz_wang2014}, the following lemma is proved. We follow its proof closely.




\begin{lemma}[See~{\cite[Lemma 3.7]{hertz_wang2014}}]
\label{lemma_cohomological}
There exists a $C^{\infty}$ map $Q'_a: \s(M) \rightarrow N(\R)$ such that 
$$h(\overline{z}) = Q'_a(\overline{z}) \rho_l(a)^{-1} (h(\rho(a)\overline{z})),$$
for all $\overline{z} \in \s(M)$.
\end{lemma}
\begin{proof}
Let us first prove the following claim: 
\par If a map $f : \s(M) \rightarrow N(\R)$ is continuous on each manifold slice and satisfies that $\overline{z}f(\overline{z}) = \overline{z}$ for all $\overline{z} \in \s(M)$, then there exists
$\gamma_0 \in N(\Z) \cap Z(N(\R))$ (where $Z(N(\R))$ denotes the center of $N(\R)$) such that $f(\overline{z}) = \gamma_0$ for all $\overline{z} \in \s(M)$.  
\par In fact, $f$ can be lifted to a function $f: N(\R) \times \prod_{p \in S } N(\Z_p) \rightarrow N(\R)$ such that $f(\gamma g) = f(g)$ for all $\gamma \in N( \Z)$. Then for all $g \in N(\R) \times \prod_{p \in S}N(\Z_p)$, $N( \Z) g f(g) = N( \Z) g$, or equivalently, $g f(g) g^{-1} \in N( \Z)$. Let $g$ varies in a manifold slice. Since $N( \Z)$ is discrete and $g f(g) g^{-1}$ is continuous, we have that $g f(g) g^{-1} $ must be a constant $\gamma_0 \in N( \Z)$. Thus $g f(g) g^{-1} $ is locally constant.  Since $N(\Z _p)$ is compact, $g f(g) g^{-1} $ is constant, and hence 
$g f(g) g^{-1} = \gamma_0$ for any $g \in N(\R) \times \prod_{p \in S} N(\Z_p)$. So $f(g) = g^{-1} \gamma_0 g$. Since $f(\gamma g) = f(g)$, we have that $\gamma_0$ commutes with $\gamma \in N(\Z)$ for all $\gamma \in N(\Z)$. Thus $\gamma_0 \in Z(N(\R))$ since $N(\Z)$ is Zariski dense in $N(\R)$. This implies that $f(g) = \gamma_0$ for all $g \in N(\R) \times \prod_{p \in S} N(\Z_p)$. This proves the claim.
\par From (\ref{cohomological_equation}), we have that 
$$ \overline{z} = \overline{z} Q_a(\overline{z}) \rho_l(a)^{-1} (h(\rho(a)\overline{z})) h^{-1}(\overline{z}).$$
Let $$f(\overline{z}):= Q_a(\overline{z}) \rho_l(a)^{-1} (h(\rho(a)\overline{z})) h^{-1}(\overline{z}).$$
Then $f: \s(M) \rightarrow N(\R)$ is continuous on each manifold slice. By the above claim, we conclude that 
$$Q_a(\overline{z}) \rho_l(a)^{-1} (h(\rho(a)\overline{z})) h^{-1}(\overline{z}) = \gamma_0$$
for some $\gamma_0 \in N(\Z) \cap Z(N(\R))$. Then $Q'_a(\overline{z}) := \gamma_0^{-1} Q_a(\overline{z})$ is $\smooth$ and satisfies 
$$h(\overline{z}) = Q'_a(\overline{z}) \rho_l(a)^{-1} (h(\rho(a)\overline{z})).$$
\par This completes the proof.

\end{proof}

\par To prove $\phi$ is $C^{\infty}$, it suffices to show that $h(\overline{z})$ is $C^{\infty}$. 


\begin{definition}
Let $\mathcal{F}$ be a foliation of $M$ with of $\smooth$ leaves.   We consider derivatives of order $k$ of functions or distributions $f$ along $\mathcal{F}$, which we denote by $\partial^k_{\mathcal{F}} f$. For $\theta \in (0,1)$, let $C^{\infty, \theta, \ast}_{\mathcal{F}}$ denote the space of distributions $f$ on $M$ such that  all partial derivatives of $f$ of any order along $\mathcal{F}$ exist as distributions on the space of $\theta$-\holder functions. Let $C^{\infty, \theta}_{\mathcal{F}}$ denote the space of $\theta$-\holder functions $f$ on $M$ such that all partial derivatives of $f$ of any order along $\mathcal{F}$ are $\theta$-\holder.

\end{definition}

\begin{notation}
\label{notation_coarse_lyapunov_subspace}
Throughout this paper, we will denote {\bf coarse Lyapunov subgroups} corresponding to {\bf coarse Lyapunov exponents} $[\chi^l], [\chi^l_1], [\chi^l_2],\dots$ of $\rho_l$ by $V, V_1 , V_2,\dots$.
\end{notation}


\subsection{Outline of the proof}
\label{outline_of_the_proof}
\par We briefly describe the basic idea (developed in \cite{kalinin_fisher_spatzier2013} and \cite{hertz_wang2014}) to establish the smoothness of $h$. 
\par  We first establish the smoothness when $M$ is a torus.
\par To this end we first show that for every {\bf coarse Lyapunov exponent} $[\chi]$ of $\rho$, the corresponding {\bf coarse Lyapunov distribution} $E^{[\chi]}$ admits a \holder foliation consisting of $\smooth$ leaves. In order to show this, we make use of our exponential mixing result for solenoids (Corollary \ref{cor:exponential-mixing}) proved in \S \ref{exponential_mixing} and the techniques and results developed by Rodriguez Hertz and Wang \cite{hertz_wang2014}.
\par Let $V$ be a coarse Lyapunov subgroup corresponding to $[\chi^l]$. We define $h_V(\overline{z})$ to be the projection of $h(\overline{z})$ on $V$. We want to show that for any {\bf coarse Lyapunov subgroup} $V$, $h_V(\overline{z})$ is $\smooth$.
\par From (\ref{cohomological_equation}) we get the corresponding equation for $h_V(\overline{z})$:

\begin{equation}
\label{subspace_cohomological_equation}
\begin{array}{rcl} h_V(\overline{z}) & = & (Q'_a(\overline{z})+  \rho_l(a)^{-1} h_V(\rho(a)\overline{z}))_V\\
 & = & (Q'_a(\overline{z}))_V + \rho_l(a)^{-1}_V h_V(\rho(a) \overline{z}),
\end{array}
\end{equation}
where $ (Q'_a(\overline{z}))_V$ denotes the projection of $ Q'_a(\overline{z})$ on $V$. 
By iterating (\ref{subspace_cohomological_equation}), one finds  the solution by the following formal series:
\begin{equation}
\label{solution_equation}
h_V = \sum_{i=0}^{\infty} \rho_l(a)^{-i} \Phi \circ \rho(a)^{i} ,
\end{equation}
where $\Phi(\overline{z}):=  (Q'_a(\overline{z}))_V$. We will apply the exponential mixing result again to prove that for any {\bf coarse Lyapunov foliation} $\mathcal{V}_1$, $h_V \in C^{\infty, \theta, \ast}_{\mathcal{V}_1}$. In the case of $\Z^k$ actions, exponential mixing was established by Fisher, Kalinin and Spatzier \cite{kalinin_fisher_spatzier2013} for tori and by Gorodnik and Spatzier \cite{gorodnik_spatzier2014}, \cite{gorodnik_spatzier_acta} for nilmanifolds. 
\par By variations of results of Rauch and Taylor in \cite{rauch_taylor2005} proved by Fisher, Kalinin and Spatzier \cite{kalinin_fisher_spatzier2013}, and by Rodriguez Hertz and Wang \cite{hertz_wang2014}, $h_V \in C^{\infty, \theta, \ast}_{\mathcal{V}_1}$ for all possible {\bf coarse Lyapunov foliations} $\mathcal{V}_1$ will imply that $h_V$ is $\smooth$. Then it follows that $h$ is $\smooth$ since $h$ can be written as the sum of $h_V$'s for all {\bf coarse Lyapunov subgroups} $V$. This will prove the smoothness when $M$ is a torus. 
\par For the general case, we follow the approach of Margulis and Qian \cite{margulis_qian2001}. We consider the derived series of $N$:
$$N = N_0 \supset N_1 \supset \cdots \supset N_{k-1} \supset N_k = \{0\},$$
and prove the smoothness of $h$ by induction on $k$.

\section{Exponential mixing for extended $\Z^k$-actions on solenoids}
\label{exponential_mixing}
In this section, we will prove Theorem \ref{thm:exponential-mixing-intro} and Corollary \ref{cor:exponential-mixing}.  As we discussed in \S \ref{preliminaries}, Corollary \ref{cor:exponential-mixing} is crucial to establish the smoothness of $\phi$.


\subsection{Preparation for the proof} We need some preparation before proving the theorem.
\begin{definition}
Let us denote $[N, N]$ by $N'$. For $\K = \R \text{ or } \Q_p$, let us define
$$\pi: N(\K) \rightarrow N'(\K)\setminus N(\K) \cong \K^l$$
to be the canonical projection. Let 
$$D\pi : \mathfrak{n}(\K) \rightarrow \K^l$$ 
denote the derived map of $\pi$ on the Lie algebras.
\par Let us define $M' := N'(\Z) \setminus N'(\R) $ and $M_0 : = (N'(\Z)\setminus N(\Z)) \setminus (N'(\R)\setminus N(\R))$. Then $M_0$ is a torus and $M$ is a bundle over $M_0$ with $M'$ fibers. We call $M_0$ the maximal torus factor of $M$.
\end{definition}
\par  We will need the following effective equidistribution result for box maps on nilmanifolds.
\begin{theorem}[See~{\cite[Theorem 2.1]{gorodnik_spatzier2014}}]
\label{thm_equidistribution_box}
Let $w_1, w_2, \dots, w_r$ be $r$ linearly independent vectors in the Lie algebra $\mathfrak{n}(\R)$. Let $v \in \mathfrak{n}(\R)$ be a fixed vector. We define the box map
\begin{equation}
\label{equa_box_map}
\iota: B=[0,T_1]\times [0,T_2]\times \cdots \times [0, T_r]\rightarrow \mathfrak{n}(\R)
\end{equation}
as follows:
$$(t_1, t_2,\dots, t_r) \in B \mapsto v + t_1 w_1 + t_2 w_2 + \cdots + t_r w_r . $$
There exist constants $L_1, L_2 >0$ such that for every $\delta \in (0, 1/2)$, every $u \in \mathfrak{n}(\R)$, every $x \in M= N(\Z)\setminus N(\R)$, every $f \in C^{\theta}(M)$ and every box map $\iota: B \rightarrow \mathfrak{n}(\R)$, one of the following holds:
\begin{enumerate}[label=\textbf{A.\arabic*}]
 \item \label{1st} 
 $$\left | \frac{1}{|B|}\int_B f(x \exp(\iota(t)) \exp(u))\dd t - \int_M f(x) \dd x  \right| \leq \delta \|f\|_{\theta}.$$
 \item \label{2nd} There exists $z \in \Z^l \setminus \{\mathbf{0}\}$ such that 
 $\|z\| \ll \delta^{-L_1}$ and 
 $$|\langle z, D\pi (w_i) \rangle| \ll \delta^{-L_2}/T_i$$ for every $i=1,2,\dots, r$.
\end{enumerate}
\end{theorem}
\par Another important result we will need is the following effective equidistribution result for polynomial orbits on nilmanifolds, proved by Green and Tao \cite{green_tao2012}.
\begin{theorem}[See~{\cite[Theorem 8.6]{green_tao2012}}] 
\label{thm_equidistribution_polynomial}
For $N = (N_1, N_2, \dots, N_r) \in \Z_{+}^r $, $[N]$ denotes the box set $[N_1]\times [N_2]\times \cdots \times [N_r]$, where $[N_i]:= \{0,1,\dots, N_i\} \subset \N$.
Let 
$p: [N] \rightarrow N(\R) $ be a polynomial map. For $i=1,2,\dots, r$, let 
$e_i \in [N]$ denote the vector with $1$ on the $i$th component and $0$ on other components, and let $\partial_i \pi(p(n)) := \pi(p(n)) - \pi(p(n-e_i))$. Then there exist constants $L_1, L_2 >0$ such that for every $\delta >0$ and every $f \in C^{\theta}(M)$, one of the following holds:
\begin{enumerate}[label=\textbf{AA.\arabic*}]
 \item \label{first} 
 $$\left | \frac{1}{N_1\cdots N_r}\sum_{n \in [N]} f(N(\Z)p(n)) - \int_M f(x) \dd x  \right| \leq \delta \|f\|_{\theta}.$$
 \item \label{second} There exists $z \in \Z^l \setminus \{\mathbf{0}\}$ such that 
 $\|z\| \ll \delta^{-L_1}$ and 
 $$\dist (|\langle z, \partial_i\pi(p(n))\rangle|, \Z) \ll \delta^{-L_2}/N_i$$ for all $i=1,2,\dots, r$ and $n \in [N]$. Here $\dist(x, \Z): = \min_{z \in \Z} \{|x-z|\}$.
\end{enumerate}
\end{theorem}
\begin{remark}
\par $\quad$
\begin{enumerate}
\item The statement in \cite{green_tao2012} is different from the above theorem. For example, the function is assumed to be Lipschitz continuous, and case \ref{second} above is stated differently. The statement above follows the one stated and applied in \cite{gorodnik_spatzier2014}. To get the above modified version from the statement in \cite{green_tao2012} , one needs to approximate \holder functions by Lipschitz functions, keeping control of all the desired estimates. We refer to \cite{gorodnik_spatzier2014} for details.
\item The non-effective version of the equidistribution of polynomial orbits in nilmanifolds is proved by Leibman \cite{leibman2005}.
\end{enumerate}
\end{remark}
\par We introduce some notation.
\begin{notation}
\par The $\Q$-structure on $N$ induces a $\Q$-structure on its Lie algebra $\mathfrak{n}$. One can construct a basis, a so-called Mal'cev basis, $\{e_1, e_2, \dots, e_d\}$ of 
$\mathfrak{n}(\Q)$ such that 
$$N(\Z) := \exp (\Z e_1 + \Z e_2 + \cdots + \Z e_d),$$
and 
$$F := \exp([0,1]e_1)\exp([0,1]e_2)\cdots \exp([0,1]e_d)$$
is a fundamental domain for $N(\Z)\setminus N(\R)$.
\par For $x \in M = N(\Z)\setminus N(\R)$ and $\epsilon >0$ small enough, $U(x, \epsilon)$ denotes the box set 
$$N(\Z)x\exp([0,\epsilon]e_1 + [0,\epsilon]e_2 + \cdots +[0,\epsilon]e_d) \subset N(\Z)\setminus N(\R).$$
  
\end{notation}
\begin{definition}[see \cite{gorodnik_spatzier2014}]
For $c, L >0$, we call $w \in \R^l$ $(c, L)$-Diophantine if 
$$|\langle z , w \rangle| \geq c \|z\|^{-L}$$
for all $z \in \Z^l \setminus \{\mathbf{0}\}$.

\end{definition}

The following lemma is proved in \cite{gorodnik_spatzier2014}.
\begin{lemma}[see~{\cite[Lemma 3.3]{gorodnik_spatzier2014}}]
\label{lemma_exist_algebraic_vector}
Let $V \subset \R^l$ be a subspace defined over $\overline{\Q}\cap \R$ such that $V$ is not contained in any proper subspace defined over $\Q$. Then there exists $\mathbf{w} \in V \cap \overline{\Q}^l$ whose coordinates are real numbers linearly independent over $\Q$.
\end{lemma}
For $p$-adic vector spaces, we have the following similar result:
\begin{lemma}
\label{lemma:exist-p-adic-algebraic-vector}
For any prime $p$, let $\Q^{al}_p \subset \Q_p$ denote the field of $p$-adic algebraic numbers. Let $V \subset \Q_p^l$ be a subspace defined over $\Q^{al}_p$ such that $V$ is not contained in any proper subspace defined over $\Q$. Then there exists $\mathbf{w} \in V\cap (\Q^{al}_p)^l$ whose coordinates are linearly independent over $\Q$.
\end{lemma}
\begin{proof}
The proof is the same as that of Lemma \ref{lemma_exist_algebraic_vector} (see \cite[Lemma 3.3]{kalinin_fisher_spatzier2013}). We will include the proof for completeness.
\par Since $V$ is defined over $\Q^{al}_p$, we can choose a basis $\{\mathbf{u}_i: i =1,2,\dots, s_1\}$ of $V$ with coordinates in $\Q^{al}_p$. Let $K \subset \Q^{al}_p$ denote the field generated by their coordinates. Then $K$ is a finite extension of $\Q$ in $\Q_p$. Then we may choose $\alpha_1, \alpha_2, \dots, \alpha_{s_1} \in \Q_p^{al}$ which are linearly independent over $K$. Let $\mathbf{w} := \sum_{i=1}^{s_1} \alpha_i \mathbf{u}_i$. Let us denote $\mathbf{u}_i = (u_{i 1}, u_{i 2}, \dots, u_{i l} )$ for $i=1, 2, \dots, s_1$. Then for any $\mathbf{z} =(z_1, \dots z_l) \in \Q^l$, we have that 
$$\langle \mathbf{z}, \mathbf{w} \rangle = \sum_{j=1}^l z_j \left( \sum_{i=1}^{s_1} \alpha_i u_{i j} \right) = \sum_{i=1}^{s_1} \left( \sum_{j=1}^l z_j u_{i j}\right) \alpha_i .$$
Since $\{\alpha_1, \dots, \alpha_{s_1}\}$ are linearly independent over $K$, we have that $\langle \mathbf{z}, \mathbf{w} \rangle \neq 0$ unless 
$$ \sum_{j=1}^l z_j u_{i j} = 0 , \text{ for all } i = 1,2, \dots, s_1,$$
i.e., $\langle \mathbf{z}, W \rangle =0$. Since we assume that $V$ is not defined over $\Q$, we have proved that the coordinates of $\mathbf{w}$ are linearly independent over $\Q$.
\par This completes the proof.
\end{proof}

\par By \cite[Theorem 7.3.2]{bombieri_gubler2006}, a vector $\mathbf{w} \in \R^l$ whose coordinates are algebraic numbers that are linearly independent over $\Q$ is $(c, L)$-Diophantine for some $c, L >0$. 

\par For $p$-adic vectors, we have the same result:
\begin{lemma}
\label{lemma:p-adic-diophantine}
Let $\mathbf{w} \in (\Q^{al}_p)^l$ such that its coordinates are linearly independent over $\Q$. Then $\mathbf{w}$ is $(c, L)$-Diophantine for some constants $c, L >0$. 
\end{lemma}
\begin{proof}
\par Let $\mathbf{w} =(w_1, w_2, \dots , w_l)$. One can find a polynomial $P(x_1, x_2, \dots, x_l)$ with integer coefficients such that the linear form $R(x_1, x_2, \dots, x_l) := \sum_{j=1}^l w_{j} x_j$ is a factor of $P$ and $P(z'_1, z'_2, \dots, z'_l ) \neq 0$ for any nonzero integer vector $z' =(z'_1, z'_2 , \dots, z'_l)$ (just take the product of all Galois conjugates of the linear form $R(x_1 , x_2 , \dots, x_l)$). Let $K' \subset \Q^{al}_p$ denote the field generated by the coordinates of $\mathbf{w}$. For any nonzero integer vector $z' = (z'_1, z'_2, \dots, z'_l)$, we have that $P(z'_1, z'_2, \dots , z'_l)$ is a nonzero integer. Moreover, since $P$ has coefficients in $\Z$, we have that $|P(z'_1, z'_2, \dots, z'_l)| \leq c_1 \|z'\|^{L}$ for some constants $ c_1, L>0$ depending on $P$. Therefore 
$$\|P(z'_1, z'_2, \dots, z'_l)\|_p \geq |P(z'_1, z'_2, \dots, z'_l)|^{-1} \geq c_1^{-1} \|z'\|^{-L}.$$ 
Put $P = R Q$, then since $R$ and $P$ have coefficients in $K$, so does $Q$. It is easy to see that there exists a constant $c_2 >0$, such that for any nonzero integer vector $z' = (z'_1, z'_2, \dots z'_l)$, 
$\|Q(z'_1, z'_2, \dots , z'_l)\|_p \leq c_2$ (in fact, $c_2$ is determined by the coefficients of $Q$). Therefore, for any $z' \in \Z^l \setminus\{\mathbf{0}\} $,
$$\|\langle z', \mathbf{w} \rangle\|_p = \|R(z'_1, z'_2, \dots, z'_l)\|_p = \|P(z'_1, z'_2, \dots, z'_l)\|_p \|Q^{-1}(z'_1, z'_2, \dots, z'_l)\|_p \geq c \|z'\|^{-L},$$
where $c = c_1^{-1} c_2^{-1}$.
\par This completes the proof.
\end{proof}
Therefore, the vector $\mathbf{w}$ we get from Lemma \ref{lemma_exist_algebraic_vector} (or Lemma \ref{lemma:exist-p-adic-algebraic-vector}) is $(c, L)$-Diophantine for some constants $c, L >0$.

\subsection{Ergodicity of the action}
\par Before proving Theorem \ref{thm:exponential-mixing-intro}, let us first study ergodicity of the $\Z^k$ action $\rho_l$ on the solenoid $\s(M)$. It is proved in \cite{starkov1999} and \cite{hertz_wang2014} that a $\Z^k$ action on a nilmanifold $M$ by automorphisms is {\em genuinely higher rank} if and only if there exists a $\Z^2$ subgroup of $\Z^k$  all of whose nontrivial elements act ergodically. We will prove a similar result for $\Z^k$ actions on solenoids.  Then we will prove Corollary \ref{cor:exponential-mixing} assuming Theorem \ref{thm:exponential-mixing-intro}. 
\begin{definition}
\label{def_irreducible}
We say $\rho_l : \Z^k_{+} \curvearrowright M$ is {\em irreducible} if for any sub-semigroup $\Gamma$ of $\Z^k_{+}$ of finite index, there are no proper $\Gamma$-invariant sub-nilmanifolds of $M$ with positive dimension.
\end{definition}

\begin{lemma}
\label{lemma:irreducible-ergodic}
Let $\rho_l: \Z^k_+ \acton M$ be an action by affine endomorphisms. Suppose $\rho_l$ is irreducible. We extend $\rho_l$ to a $\Z^k$ action on $\s(M)$, then every nontrivial $\rho_l(a)$ acts ergodically on $\s(M)$.
\end{lemma}
\begin{proof}
 By Parry's theorem \cite{parry1969}, it suffices to show the statement when $M$ is a torus $\Z^{\dim M}\setminus \R^{\dim M}$ (Although Parry's theorem only takes care of nilmanifolds, its proof works for solenoids).  Then every $\rho_l(a)$ can be extended to an action on $\R^{\dim M}$ and identified as an element in $\GL(\dim M, \Q)$. 
\par We first claim that every $\rho_l(a)$ is semisimple. In fact, if $\rho_l(a)$ is not semisimple, then $\rho_l(a) = s(a) u(a) = u(a) s(a) $, where $s(a) \in \GL(\dim M, \Q)$ is semisimple and $u(a) \in \GL(\dim M, \Q)$ is unipotent. Then obviously, 
$\mathrm{Fix}(u(a)) := \{v \in \R^{\dim M}: u(a) v = v\}$ is a nontrivial proper rational subspace of $\R^{\dim M}$ invariant under $\rho_l(\Z^k)$. Then $\mathrm{Fix}(u(a))$ will define a proper $\rho_l(\Z^k)$ invariant subtorus of $M$ with positive dimension, which contradicts our assumption. This shows the claim.
\par We next claim that every nontrivial $\rho_l(a)$ does not admit any nontrivial proper invariant subtori of $M$. Otherwise, take a nontrivial proper $\rho_l(a)$ invariant subtorus with minimal dimension, say $M'$. Then $M'$ corresponds to a minimal $\rho_l(a)$-invariant rational subspace $V'$ of $\R^{\dim M}$. For any $b \in \Z^k$, it is easy to see that $V'\cap \rho_l(b)V'$ is also a $\rho_l(a)$-invariant subspace defined over $\Q$. Since $V'$ is assumed to be minimal, we have $\rho_l(b)V' = V'$ or $\rho_l(b)V' \cap V' = \{\mathbf{0}\}$. For the same reason, for any $b,b' \in \Z^k$, we have that 
$D\rho_l(b) V' =\rho_l(b') V' $ or $\rho_l(b)V' \cap \rho_l(b')V' = \{\mathbf{0}\}$. Thus there are only finitely many possible Lie algebras $V"$ that $\rho_l(b)V'$ can be. This implies that there exists a subgroup $\Gamma$ of $\Z^k$ of finite index such that $V'$ is $\rho_l(\Gamma)$-invariant. This contradicts the assumption on irreducibility. Therefore, no $\rho_l(a)$  leaves any nontrivial proper subtori invariant.
\par Note that the dual space of $\s(\Z^{\dim M} \setminus \R^{\dim M})$ is $(\Z[\frac{1}{n}])^{\dim M}$ where $n = \prod_{s \in S} p$. For any nontrivial $\rho_l(a)$, if $\rho_l(a)$ is not ergodic, then there exists $z \in (\Z[\frac{1}{n}])^{\dim M}$ fixed by $\rho_l(a)$. This implies that $\rho_l(a)$ fixes a proper subtorus $M'$ of $M$, which contradicts the previous claim.
\end{proof}
\begin{lemma}
\label{lemma:ergodic-element}
Let $\rho_l : \Z^k \acton \s(M)$ be a {\em genuinely higher rank} action as above. Then there exists $a \in \Z^k$ such that $\rho_l(a)$ is ergodic.
\end{lemma}
\begin{proof}
For the same reason as above, we may assume that $M$ is a torus $\Z^{\dim M} \setminus \R^{\dim M}$. By \cite[Corollary 6]{starkov1999}, we may assume that every $\rho_l(a) \in \GL(\dim M, \Q)$ is semisimple.
\par By passing to a subgroup of $\Z^k$ of finite index, we can decompose $M$ into almost direct product of $\rho_l$-invariant irreducible subtori:
$$M = M_1 \times \cdots \times M_s.$$
For each $i=1,\dots, s$, $\rho_l(a)$ acts either ergodically or trivially on $M_i$. Our goal is to find $a \in \Z^k$ such that $\rho_l(a)$
is not trivial on each $M_i$.
\par For contradiction, we suppose that every $\rho_l(a)$ acts trivially on some $M_i$. Define
\[
	\mathcal{F}_i := \{a \in \Z^k : \rho_l(a) \text{ acts trivially on } M_i\},
\]
then every $\mathcal{F}_i$ is a subgroup of $\Z^k$ and $\Z^k = \bigcup_{i=1}^s \mathcal{F}_i$. This implies that some $\mathcal{F}_i$ is a subgroup of $\Z^k$ of finite index, which contradicts our higher rank assumption (since the restriction of $\rho_l$ on $M_i$ is essentially trivial).
\end{proof}
\begin{definition}
We call an integer triple $(k,l,m)\in \Z^3$ {\em primitive} if $k,l,m$ do not have nontrivial divisor.
\end{definition}
\begin{lemma}
\label{lemma:non-ergodic-triples}
Let semisimple elements $A, B, C \in GL(d , \Q)$ commute and act on $\s(\Z^d \setminus \R^d)$. Suppose for any $i , j \in \Z$,
$A^i B^j$ is ergodic unless $i=j=0$. Then there exist at most finitely many {\em primitive} triples $(k,l,m) \in \Z^3$ such that $A^k B^l C^m$ is not ergodic.
\end{lemma}
\begin{proof}
For $ T \in \GL(d, \Q) $, define
\[\Fix(T):= \{v \in \R^d: T v = v \}.\] 
Let $S = A^k B^l C^m$ be a non-ergodic element. Then for some $r \in \Z$, $\Fix(S^r)$ is nontrivial. Let $V = \Fix(S^r)$. Then $V$ is rational and invariant under the action of $A$, $B$ and $C$.
\par If 
$V = \R^d$, then there exists only one {\em primitive} non-ergodic triple. In fact, for any {\em primitive} non-ergodic triple 
$(k_1, l_1, m_1)$, there exist $r_1 \in \Z$ and $v_1 \in \R^d\setminus \{\mathbf{0}\}$, such that $A^{r_1 k_1} B^{r_1 l_1} C^{r_1 m_1} v_1 = v_1$.
Since $A^{r k} B^{r l} C^{r m} v_1 = v_1$ and $A^i B^j$ is ergodic for any $(i,j) \neq (0,0)$, we have that $k_1/ k = l_1/l = m_1/m$. 
\par If $V \neq \R^d$, then there exists a rational nontrivial $A,B,C$-invariant subspace $V'$ such that $\R^d = V \oplus V'$. $V$ and $V'$ correspond to nontrivial subtori $T$ and $T'$, respectively, such $\Z^d \setminus \R^d = T \oplus T'$ and $\s(T)$ and $\s(T')$ are both $A,B,C$-invariant. Since $A^i B^j$ is ergodic on $\s(T)$ and $\s(T')$ for every $(i,j) \neq (0,0)$ and $A^k B^l C^m$ is ergodic on $\s(\Z^d \setminus \R^d)$ if and only if it is ergodic on both $\s(T)$ and $\s(T')$, we can complete the proof by induction.
\end{proof}

\begin{proposition}
\label{prop_starkov}
Let $\rho_l$ be a $\Z^k$ action on the solenoid $\s(M)$ of a nilmanifold $M$, extended from a
$\Z^k_{+}$ action on $M$ by affine endomorphisms. If $\rho_l$ is {\em genuinely higher rank}, then there exists a subgroup $\Sigma$ of $\Z^k$ isomorphic to $\Z^2$ consisting of ergodic elements.
\end{proposition}
\begin{proof}
\par We will prove the following stronger statement: for any ergodic element $\rho_l(a)$, there exists a subgroup $\Sigma \cong \Z^2$ containing $a$ which consists of ergodic elements. 
  
 \par For the same reason as above, throughout this proof, we will assume that $M$ is a torus $\Z^{\dim M} \setminus \R^{\dim M} $ and every $\rho_l(a)\in \GL(\dim M, \Q)$ is semisimple.
 

 \par By passing to a subgroup of $\Z^k$ of finite index, we can decompose $M$ into almost direct product of $\rho_l$-invariant irreducible subtori:
$$M = M_1 \times \cdots \times M_s.$$
 \par Let us prove the statement by induction on $s$. 
 \par When $s =1$, the action $\rho_l$ is irreducible. Then the statement follows from Lemma \ref{lemma:irreducible-ergodic}. In fact, every nontrivial $\rho_l(a)$ is ergodic.
 
 \par Suppose the statement holds for $s -1$, we shall prove the statement for $s$.
\par By Lemma \ref{lemma:ergodic-element}, there exist ergodic elements. Take an ergodic element $a \in \Z^k$, we want to show that there exists a subgroup $\Sigma \cong \Z^2$ containing $a$ which consists of ergodic elements. By inductive assumption, 
 there exist $b_1 , b_2 \in \Z^k$ such that the restriction of $\rho_l(a)$ and $\rho_l(b_1)$ to $\s(M_1)$ generate a $\Z^2$ action consisting of ergodic elements, and the restriction of $\rho_l(a)$ and 
 $\rho_l(b_2)$ to $\s(M'_1)$ generate a $\Z^2$ action consisting of ergodic elements. 
 
 \par By Lemma \ref{lemma:non-ergodic-triples} applied to $M_1$ and $M'_1$, there are at most finitely many {\em primitive} triples $(k , l, m) \in \Z^3$ such that the restriction of 
  $\rho_l(k a + l b_1 + m b_2) $ onto $\s(M_1)$ or $\s(M'_1)$ is not ergodic. This implies that for all but finitely many {\em primitive} triples $(k,l,m) \in \Z^3$, $\rho_l(k a + l b_1 + m b_2 )$ is ergodic on $\s(M)$. This implies that there exists a $\Z^2$ subgroup of $\Z^k$ containing $a$ which consists of ergodic elements.
  \par This completes the proof. 
\end{proof}
\begin{proof}[Proof of Corollary \ref{cor:exponential-mixing} assuming Theorem \ref{thm:exponential-mixing-intro}]
 By our assumption, $\rho_l$ is {\em genuinely higher rank}, then by Proposition \ref{prop_starkov}, there exists a subgroup $\Sigma \cong \Z^2$ of $\Z^k$ such that for every $a \in \Sigma$, $\rho_l(a)$ is ergodic. Then by Theorem \ref{thm:exponential-mixing-intro}, there exist constant $a_1 > 0$ and $\eta' >0$, such that for every $a \in \Sigma$, any $f \in C^{\theta}(M)$, considered as a function on $\s(M)$, and any $g \in C^{\theta}(\s(M))$, \eqref{equ:exponential-mixing} holds. Recall that $\rho_l$ and $\rho$ are conjugate via the bi-\holder conjugacy $\phi$. Thus if the exponential mixing holds for $\rho_l$ and $\mu$, it also holds for $\rho$ and $\tilde{\mu} = \phi_{\ast}^{-1}(\mu)$. This proves Corollary \ref{cor:exponential-mixing}.
\end{proof}
\subsection{Maximal expanding factor}
\par Assuming $\rho_l(a)$ is ergodic for every $a \in \Z^k$, we will study the maximal expanding factor of every $\rho_l(a)$.
\begin{lemma}
\label{lemma_max_lyapunov_exponents}
For $a \in \Z^k$, let $S(a) := \max_{\chi^l} \{|\chi^l(a)|\}$ where $\chi^l$ runs over all {\bf Lyapunov exponents} of $\rho_l$. Then 
$$\inf_{a \in \Z^k} \{S(a)\} >0 .$$
\end{lemma}
\begin{proof}
By passing from  the action $\rho_l$ to its maximal torus factor $M_0$, we can reduce the proof to the case that $M$ is a torus $\Z^{\dim M} \setminus \R^{\dim M}$, again using Parry's theorem \cite{parry1969} that a nilmanifold endomorphism is ergodic precisely when its projection to the maximal toral factor is ergodic. Then every $\rho_l(a), a \in \Z^k$, can be expressed as an element in $\GL(\dim M, \Q)$.
\par We first prove the statement assuming every $\rho_l(a), a \in \Z^k$, is semisimple. Let $\K$ denote $\R$ or $\Q_p$. Then
$$\mathfrak{n}(\K) = \K^{\dim M} = \bigoplus_{\chi^l \in T(\K)} \sigma^{\chi^l}.$$
For $a \in \Z^k$, since $\rho_l(a)$ is semisimple, $\sigma^{\chi^l}$ is a generalized eigenspace of $\rho_l(a)$ with generalized eigenvalue $\chi^l(a)$.

For contradiction, suppose that there exists a sequence 
$\{a_r: r \in \N\}$ such that $\|a_r\|\rightarrow \infty$ as $r \rightarrow \infty$, and $S(a_r) \rightarrow 0$ as $r \rightarrow \infty$. Let $K \subset \overline{\Q}$ be a finite field extension of $\Q$ such that every $\rho_l(a_r)$ is diagonalizable in $\GL(\dim M, K)$. So we may assume that every 
$\rho_l(a_r)$ is a diagonal matrix in $\GL(\dim M , K)$, denoted by $\mathrm{diag}\{a_r(i,i): 1 \leq i \leq \dim M\}$. Then $S(a_r) \rightarrow 0$ implies that for any prime ideal $\mathfrak{p}$ of the ring of algebraic integers $\mathcal{O}_K$, $|a_r(i,i)|_{\mathfrak{p}} \rightarrow 1$ as $r \rightarrow \infty$, for all $i=1,2,\dots, \dim M$. Moreover, $|a_r(i,i)| \rightarrow 1$ as $r \rightarrow \infty$ for all $i=1,2,\dots, \dim M$. There are only finitely many units in $K$ with absolute value equal to $1$, thus there exists an element $u \in \GL(\dim M, K)$ with some power of $u$ is identity, such that passing to a subsequence, 
$$\rho_l(a_r) \rightarrow u, \text{ as } r \rightarrow \infty$$
in $\GL(\dim M, \mathbb{C})$ and also in $\GL(\dim M, K_{\mathfrak{p}})$ for any prime ideal $\mathfrak{p}$ in $\mathcal{O}_K$.

Thus $\rho_l(a_r) \rightarrow u $ as $l \rightarrow \infty$ in $\GL(\dim M, \mathbb{A}_{K})$ where $\mathbb{A}_{K}$ denotes the adeles of $K$. 

Since $\GL(\dim M, K)$ is discrete in $\GL(\dim M, \mathbb{A}_K)$, we will have $\rho_l(a_r) = u$ for $r$ large enough. This implies that for $r$ large enough $\rho_l(a_r)=u \in \GL(\dim M, \Q)$ and some power of $u$ is identity. This contradicts the assumption that every $\rho_l(a_r)$ is ergodic since some power of $\rho_l(a_r)$ is identity. This shows the statement assuming every $\rho_l(a_r)$ is semisimple.
\par Now we prove the statement in general. Suppose  $\Z^k$ is generated by $a_1, a_2, \dots, a_k$. Consider the Jordan decomposition of $\rho_l(a_1)$: $\rho_l(a_1) = b_1 c_1$ with $b_1$ semisimple and $c_1$ unipotent. Since $c_1 \in \GL(\dim M , \Q)$, the eigenspace of $c_1$ with eigenvalue $1$, which we denote by $W_1$, is nontrivial and defined over $\Q$. Also, $W_1$ is $\rho_l(\Z^k)$-invariant, and the restriction of $\rho_l(a_1)$ onto $W_1$ is semisimple.   Repeating this argument, we can find a sequence of rational $\rho_l(\Z^k)$-invariant subspaces $W_k \subset \cdots \subset W_2 \subset W_1$ such that for $i=1,2,\dots, k$, the restriction of $\rho_l(a_i)$ on $W_i$ is semisimple. Then the restriction of $\rho_l(\Z^k)$ on $W_k$ is semisimple.
By the special case above, $\inf_{ a \in \Z^k \setminus \{\mathbf{0}\}}\{ S(a |_{W_k})\} >0$. Then the statement follows since 
$$\inf_{ a \in \Z^k \setminus \{\mathbf{0}\}}\{ S(a )\} \geq \inf_{ a \in \Z^k \setminus \{\mathbf{0}\}}\{ S(a |_{W_k})\}.$$
\end{proof}

\begin{lemma}
\label{lemma_max_lyapunov_real}
For all $a \in \R^k \setminus \{\mathbf{0}\}$, we  define $S(a):= \max_{\chi^l} \{|\chi^l(a)|\}$.
Then $\sigma := \frac{1}{2} \inf \{S(a) : a \in \R^k, \|a\| = 1\}$ is positive.
\end{lemma}
\begin{proof}
Let us first prove $S(a) >0$ for every $a \in \R^k\setminus \{\mathbf{0}\}$. Suppose $S(a) =0$ for some $a \in \R^k \setminus \{\mathbf{0}\}$. Since the line $\{t a : t \in \R\}$ comes arbitrarily close 
to integer points in $\Z^k$, we can find $t_l \in \R$ and $a_l \in \Z^k$ with $a_l - t_l a \rightarrow \mathbf{0} $ as 
$l \rightarrow \infty$. Since $S(t_l a) =0$, we have $S(a_l) \rightarrow 0$ as $l \rightarrow \infty$. This contradicts Lemma 
\ref{lemma_max_lyapunov_exponents}. This proves that $S(a) >0$. Then the statement follows as $S$ is continuous.
\end{proof}
\begin{remark}
\label{rmk:maximal-expanding} Note that there are only finitely many {\bf Lyapunov exponents} for $\rho_l$ and for any $a \in \Z^k$, 
$\sum_{\chi^l} \chi^l (a) =0$. Thus, the above lemma implies that there exists a constant $L' >0$ such that
for any $a \in \Z^k$, there exists a {\bf Lyapunov exponent} $\chi^l$ such that $\chi^l(a) \geq L' \|a\|$.
\end{remark}


\subsection{Proof of Theorem \ref{thm:exponential-mixing-intro}} 
Now we are ready to prove Theorem \ref{thm:exponential-mixing-intro}.
\par We first prove the theorem for irreducible actions, and then deal with the general case. 
\subsection*{Irreducible Case} In this case we assume that the action $\rho_l: \Z^k  \acton \s(M)$ is irreducible.
\par We deal with the following two cases separately:
\begin{enumerate}[label=\textbf{Case \arabic*}]
\item \label{case_1} Let $L'>0$ be the constant given in Remark \ref{rmk:maximal-expanding}. For all $a \in \Z^k$, there exists a real {\bf Lyapunov exponent} $\chi^l$ such that $\chi^l(a) \geq L' \|a\|$.
\item \label{case_2}  \ref{case_1} fails.
\end{enumerate}
\begin{proof}[Proof for \ref{case_1}]
\par For this case, the proof is more or less the same as that in \cite{gorodnik_spatzier2014}.
\par Suppose $\chi^l_1(a) = \max \{\chi^l(a) : \chi^l  \in T(\R)\}$. Then by our asssumption, $\chi^l_1 (a) \geq L' \|a\|$. Then $\sigma^{\chi^l_1} \subset \mathfrak{n}(\R)$ is the generalized eigenspace of $D\rho_l(a)$ with generalized eigenvalue $e^{\chi^l_1(a)}$. It is easy to see that $\sigma^{\chi^l_1}$ is defined over $\overline{\Q} \cap \R$.
Thus $W' := D\pi(\sigma^{\chi^l_1}) \subset \R^l$ is also defined over $\overline{\Q} \cap \R$. Then by Lemma \ref{lemma_exist_algebraic_vector}, there exists
$w \in W'$ whose coordinates are algebraic numbers that are linearly independent over $\Q$. As discussed after Lemma \ref{lemma:p-adic-diophantine}, there exist constants $c, L >0$ such that 
$$|\langle w , z \rangle| \geq c \|z\|^{-L}, \text{ for all } z \in \Z^l\setminus\{\mathbf{0}\}. $$
\par  
Let us fix a small constant $\epsilon >0$. For each $p \in S$, we pick an integer $l_p(a) \geq 0$ such that $p^{-l_p(a)} \leq \epsilon$, and 
$$\rho_l(a)(N(p^{l_p(a)}\Z_p)) \subset N(\Z_p).$$
\par We cut $\s(M)$ into small pieces along $p$-adic directions:
$$\s(M) = \bigcup_{j} B_{j},$$
where 
$$B_{j} := N(\Z) \setminus \mathcal{F} \times \prod_{p \in S} \xi_j(p)N(p^{l_p(a)}\Z_p)/M_p,$$
for a fixed fundamental domain $\mathcal{F} \subset N(\R)$ of $N(\Z)\setminus N(\R)$ and some $\xi_j(p) \in N(\Z_p)$. 
\par We fix a basis $\{w_1 ,\dots, w_{s_1} \}$ of $\sigma^{\chi^l_1}$ and extend it to a fixed basis $\{w_1 ,\dots, w_{s_1}, v_1 , \dots , v_{s_2}\}$ of $\mathfrak{n}(\R)$. For $\epsilon >0$, we define 
$$C(\epsilon) := W(\epsilon) + B(\epsilon) \subset \mathfrak{n}(\R),$$
where $W(\epsilon) := [-\epsilon , \epsilon] w_1 + \cdots + [-\epsilon , \epsilon] w_{s_1}$ and $B(\epsilon) := [-\epsilon , \epsilon] v_1 + \cdots + [-\epsilon, \epsilon] v_{s_2}$. Then for $x \in  N(\Z)\setminus N(\R)$ and $\epsilon >0$, we define 
 $$U(x, \epsilon):= x \exp ( C(\epsilon)) \subset N(\Z) \setminus N(\R).$$
\par We then cut each $B_i$ into small pieces along the real component. In other words, we write 
$$B_i = \bigcup_{j} B_{i,j},$$
where $B_{i,j} := U(x_i, \epsilon) \times \prod_{p \in S} \xi_j(p)N(p^{l_p(a)}\Z_p)/M_p$, for some $x_i\in N(\Z) \setminus N(\R)$. Then 
$$\int_{\s(M)} f(\rho_l(a)\overline{z}) g(\overline{z}) \dd \mu(\overline{z}) = \sum_{i,j} \int_{B_{i,j}} f(\rho_l(a)\overline{z}) g(\overline{z}) \dd \mu(\overline{z}).$$
It is easy to see that 
$$\int_{B_{i,j}} f(\rho_l(a)\overline{z}) g(\overline{z}) \dd \mu(\overline{z}) = \left( g(\overline{z}_{i,j})+O(\epsilon^{\theta}\|g\|_{\theta} ) \right) \int_{B_{i,j}} f(\rho_l(a)\overline{z})\dd \mu(\overline{z})$$
where $\overline{z}_{i,j} = (x_i, (\xi_j(p))_{p \in S}) \in B_{i,j}$. By Lemma \ref{lemma_chinese_remainder}, for each $j$, we may choose $n_j \in N(\Z)$ such that $n_j^{-1}\xi_j(p) \in N(p^{l_p(a)} \Z_p)$ for all $p \in S$. Therefore, 
$$B_{i,j} = U(n_j^{-1} x_i , \epsilon) \times \prod_{p \in S} N(p^{l_p(a)}\Z_p)/M_p.$$
Since $\rho_l(a)(N(p^{l_p(a)}\Z_p)) \subset N(\Z_p)$ and since the value of the function $f$ only depends on its projection on $M$, we have that 
$$\int_{B_{i,j}} f(\rho_l(a)\overline{z}) \dd \mu (\overline{z}) = V \int_{U(n_j^{-1} x_i, \epsilon)} f(\rho_l(a) x) \dd \nu (x) ,$$
where $V = \prod_{p \in S} \nu_p(N(p^{l_p(a)}\Z_p))$.
Let $y_{i,j} := n_j^{-1} x_i$, then we have that 
$$\int_{U(n_j^{-1} x_i, \epsilon)} f(\rho_l(a) x) \dd \nu (x) = 
\int_{v  \in B(\epsilon)}\int_{w \in W(\epsilon)} f(y_{i,j} \exp(D \rho_l(a)v + D\rho_l(a)w )) \dd w \dd v.$$
We want to show that for all $v \in \mathfrak{n}(\R)$,
the integral 
\begin{equation}
\label{equa_integral_long_piece}
\frac{1}{\mathrm{Vol}(W(\epsilon))}\int_{w \in W(\epsilon)} f(y_{i,j} \exp (v + D\rho_l(a) w)) \dd w\end{equation}
estimates $\int_{N(\Z)\setminus N(\R)} f(x) \dd \nu (x)$ with error 
$O(e^{-\eta \|a\|} \|f\|_{\theta})$ for a constant $\eta >0$. Here $\Vol (\cdot)$ denotes the volume with respect to the normalised Lebesgue measure $\dd w$.
\par  Let $v \in \mathfrak{n}(\R)$ be fixed. We consider the box map
$$\iota : [-\epsilon , \epsilon ]^{s_1} \rightarrow \mathfrak{n}(\R):$$
 $$\mathbf{t} = (t_1, \dots, t_{s_1}) \in [-\epsilon , \epsilon ]^{s_1} \mapsto  v + t_1 w'_1 + \cdots + t_{s_1} w'_{s_1},$$
 where $w'_i = D\rho_l(a)(w_i)$. Then it is easy to see that 
 $$\frac{1}{\mathrm{Vol}( W(\epsilon))}\int_{w \in W(\epsilon)} f(y_{i,j} \exp (v + D\rho_l(a) w)) \dd w = (2\epsilon)^{-s_1} \int_{[-\epsilon, \epsilon]^{s_1}} f(y_{i,j}\exp(\iota(\mathbf{t}))) \dd \mathbf{t}.$$
Then the integral (\ref{equa_integral_long_piece}) is the integral of $f$ along the box map $\iota$ and based at the point $y_{i,j}$. For contradiction, suppose (\ref{equa_integral_long_piece}) does not estimate $\int_M f \dd \nu$ with error $O(e^{-\eta \|a\|} \|f\|_{\theta})$. Then by Theorem \ref{thm_equidistribution_box}, there exists $z \in \Z^l\setminus \{\mathbf{0}\}$ such that $\|z\| \ll e^{ L_1 \eta \|a\| }$ and 
$$|\langle z, D\pi (w'_i) \rangle| \ll e^{L_2 \eta \|a\|}/ \epsilon.$$
Since $\|w'_i\| = \|D \rho_l(a) w_i\| \asymp e^{\chi^l_1(a)} \|w_i\| \gg e^{L'\|a\|}$ and $\sigma^{\chi^l_1}$ is spanned by $\{w'_1 , \dots,  w'_{s_1}\}$, we have that for all $w \in \sigma^{\chi^l_1}$ with $\|w\| \asymp 1$, 
\begin{equation}
\label{equa_ineq_1}
|\langle z, D\pi(w) \rangle| \ll e^{(L_2 \eta -L')\|a\|}/\epsilon.
\end{equation}
On the other hand, by Lemma \ref{lemma_exist_algebraic_vector}, there exists $w \in W$ such that 
\begin{equation}
\label{equa_ineq_2}
|\langle D\pi (w), z \rangle| \gg \|z\|^{-L} \gg e^{-L_1 L \eta \|a\|}.
\end{equation}
Let $\epsilon = e^{-L_3 \|a\|}$ such that $0<L_3 < L'/2$.
Then (\ref{equa_ineq_1}) and (\ref{equa_ineq_2}) will lead to a contradiction if $L_3 + L_2 \eta -L' < - L_1 L \eta$. This shows that there exists constant $\eta >0$ such that, for all $v \in \mathfrak{n}(\R)$, integral (\ref{equa_integral_long_piece}) estimates $\int_M f \dd \nu$ with error $O(e^{-\eta \|a\|} \|f\|_{\theta})$. This implies that 
$$\begin{array}{cl} & \int_{U(y_{i,j}, \epsilon)} f(\rho_l(a) x) \dd \nu (x) \\
= & \int_{v \in B(\epsilon)}\Vol(W(\epsilon))(\int_M f \dd \nu + O(e^{-\eta \|a\|} \|f\|_{\theta})) \dd v \\
= & \Vol(B(\epsilon)) \Vol(W(\epsilon))(\int_M f \dd \nu + O(e^{-\eta \|a\|} \|f\|_{\theta})) \\
= & \nu(U(y_{i,j}, \epsilon))(\int_M f \dd \nu +O(e^{-\eta \|a\|} \|f\|_{\theta})).
\end{array}$$
Therefore, 
$$\begin{array}{cl} & \int_{B_{i,j}} f(\rho_l(a)\overline{z}) \dd \mu (\overline{z})\\
= &  \prod_{p \in S} \nu_p(N(p^{l_p(a)} \Z_p)) \nu(U(y_{i,j}, \epsilon)) (\int_M f \dd \nu +O(e^{-\eta \|a\|} \|f\|_{\theta})) \\
= &  \prod_{p \in S} \nu_p(N(p^{l_p(a)} \Z_p)) \nu(U(x_i, \epsilon)) (\int_M f \dd \nu +O(e^{-\eta \|a\|} \|f\|_{\theta})) \\
= & \mu(B_{i,j}) (\int_M f \dd \nu +O(e^{-\eta \|a\|} \|f\|_{\theta})) \\
= & \mu(B_{i,j}) (\int_{\s(M)} f \dd \mu +O(e^{-\eta \|a\|} \|f\|_{\theta})) .\end{array}$$
Finally,
$$\begin{array}{cl}
& \int_{\s(M)} f(\rho_l(a)\overline{z}) g(\overline{z}) d\mu (\overline{z}) \\
= & \sum_{i,j} \int_{B_{i,j}} f(\rho_l(a)\overline{z}) g(\overline{z}) d\mu (\overline{z}) \\
= &  \sum_{i,j} (g(\overline{z}_{i,j}) + O(\epsilon^{\theta} \|g\|_{\theta}))  \int_{B_{i,j}} f(\rho_l(a)\overline{z}) \dd \mu (\overline{z}) \\
= & \sum_{i,j} (g(\overline{z}_{i,j}) + O(\epsilon^{\theta} \|g\|_{\theta})) \mu(B_{i,j})
(\int_{\s(M)} f \dd \mu +O(e^{-\eta \|a\|} \|f\|_{\theta})) \\
= & \int_{\s(M)} f \dd \mu \sum_{i,j}g(\overline{z}_{i,j}) \mu(B_{i,j}) + O(\epsilon^{\theta} \|g\|_{\theta}) \int_{\s(M)} f \dd \mu \sum_{i,j} \mu(B_{i,j}) \\
 & + O(e^{-\eta \|a\|} \|f\|_{\theta}) \sum_{i,j} g(\overline{z}_{i,j}) \mu(B_{i,j}) 
 + O(\epsilon^{\theta} \|g\|_{\theta}) O(e^{-\eta \|a\|} \|f\|_{\theta}) \sum_{i,j} \mu(B_{i,j}) \\
 = & \int_{\s(M)} f \dd \mu (\int_{\s(M)} g \dd \mu + O(\epsilon^{\theta} \|g\|_{\theta})) +O(\epsilon^{\theta} \|g\|_{\theta}) \int_{\s(M)} f \dd \mu \\
 & + O(e^{-\eta \|a\|} \|f\|_{\theta}) (\int_{\s(M)} g \dd \mu + O(\epsilon^{\theta} \|g\|_{\theta})) + O(\epsilon^{\theta} e^{-\eta \|a\|} \|f\|_{\theta} \|g\|_{\theta}).
\end{array}$$
Since $|\int_{\s(M)} f \dd \mu | \leq \|f\|_{\theta}$, $|\int_{\s(M)} g \dd \mu | \leq \|g\|_{\theta}$ and $\epsilon = e^{-L_3 \|a\|}$, 
we have there exists a constant $\eta' >0$ such that 
$$\int_{\s(M)} f(\rho_l(a)\overline{z}) g(\overline{z}) \dd \mu (\overline{z}) = \int_{\s(M)} f \dd \mu \int_{\s(M)} g \dd \mu +O(e^{-\eta' \|a\|} \|f\|_{\theta}\|g\|_{\theta}).$$
\par This finishes the proof for \ref{case_1}.
\end{proof}

\begin{proof}[Proof for \ref{case_2}]  By Remark \ref{rmk:maximal-expanding}, there exists a {\bf Lyapunov exponent} $\chi^l_2 \in T(\Q_p)$ for some $p \in S$ such that $\chi^l_2(a) \geq L' \|a\|$. We may further assume that $\chi^l_2(a) = \max_{\chi^l}\{ \chi^l(a)\}$. Let $\sigma^{\chi^l_2} \subset \mathfrak{n}(\Q_p)$ denote the generalized eigenspace of $D\rho_l(a)$ with generalized eigenvalue $e^{\chi^l_2(a)}$. 
\par Let us fix a basis $\{w_1 , \dots , w_{s_1}\}$ of $\sigma^{\chi^l_2}$ and  extend $\{w_1 , \dots , w_{s_1}\}$ to a basis $\{w_1, \dots, w_{s_1}, v_1 , \dots , v_{s_2}\}$ of $\mathfrak{n}(\Q_p)$. Without loss of generality, we may assume that $\|w_i \|_p =1$ for $i =1, \dots, s_1$. Then 
$\{D\rho_l(a)w_1 , \dots, D\rho_l(a) w_{s_1}\}$ is also a basis of $\sigma^{\chi^l_2}$ and $\|D\rho_l(a) w_i\|_p = e^{\chi^l_2(a)}$ for $i=1,2,\dots, s_1$. Denote $e^{\chi^l_2(a)} = p^h$, and $D\rho_l(a) w_i = p^{-h} u_i$ for $i=1,\dots , s_1$. 
Then $\{u_1 , \dots , u_{s_1}\}$ is a basis of $\sigma^{\chi^l_2}$ and $\|u_i\|_p =1$ for $i=1,\dots, s_1$.
\par Pick $\delta>0$ small enough such that $\delta < e^{-L' \|a\|}$ and the diameter of $\rho_l(a)(U(x, \delta))$ is less than $e^{-L' \|a\|}$ for all $x \in M$. Let $\epsilon = e^{-L' \|a\|/2}$. For $S \ni q \neq p$, let $l_q(a) > 0$ 
denote the smallest integer such that $q^{-l_q(a)} \leq \epsilon$ and $\rho_l(a)(N(q^{l_q(a)} \Z_q)) \subset N(\Z_q)$. Let $l_p>0$ denote the smallest integer such that $p^{-l_p} \leq \epsilon$. Let us cut $\s(M)$ into small pieces as follows:
$$\s(M) = \bigcup_{j} B_j,$$
where $B_j := N(\Z) \setminus \xi_j(p)N(p^{l_p}\Z_p) \times U(x_j , \epsilon) \times \prod_{ S \ni q \neq p} \xi_j(q) N(q^{l_q(a)} \Z_q)$ for some $x_j \in N(\R)$, $\xi_j(p) \in N(\Z_p)$ and $\xi_j(q) \in N(\Z_q)$.  Since $e^{\chi^l_2(a)} = p^h$ and since $\chi^l_2(a)$ is maximal among all $\chi^l(a)$,  we have that $\rho_l(a) N(p^{h} \Z_p) \subset N(\Z_p)$. For a positive integer $h'$, define 
$$\mathfrak{u}_p(h') := \left\{ \begin{array}{l}t_1 w_1 + \cdots + t_{s_1} w_{s_1} + b_1 v_1 + \cdots + b_{s_2} v_{s_2} : \\ \|t_i\|_p \leq p^{-l_p} \text{ for } i =1,\dots, s_1,\text{ and } \|b_i\|_p \leq p^{-h'} \text{ for } i =1 ,\dots, s_2.\end{array}\right\} \subset \mathfrak{n}(\Z_p),$$
and $U_p(h') := \exp (\mathfrak{u}_p(h')) \subset N(\Z_p)$. 

\par Let us further cut each $B_j$ along the $p$-adic direction:
$$B_j = \bigcup_{i} B_{i,j},$$
where $B_{i,j} := N(\Z) \setminus \xi_{i,j}(p) U_p(h) \times U( x_j, \epsilon) \times \prod_{S \ni q \neq p} \xi_j(q)N(q^{l_q(a)} \Z_q) $ for some $\xi_{i,j}(p) \in N(\Z_p)$. By Lemma 
\ref{lemma_chinese_remainder}, there exists $n_{i,j} \in N(\Z)$ such that $n_{i,j}^{-1} \xi_{i,j}(p) \in U_p(h) $ and $n_{i,j}^{-1} \xi_j(q) \in N(q^{l_q(a)} \Z_q)$ for $q \neq p$. Then 


$$B_{i,j} = N(\Z) \setminus  U_p(h) \times U( y_{i,j}, \epsilon) \times \prod_{q \neq p} N(q^{l_q(a)} \Z_q),$$
where $y_{i,j} = n^{-1}_{i,j} x_j$.
Then 
$$\begin{array}{cl}
& \int_{\s(M)} f(\rho_l(a) \overline{z}) g(\overline{z}) \dd \mu(\overline{z}) \\
= & \sum_{i,j} \int_{B_{i,j}} f(\rho_l(a) \overline{z}) g(\overline{z}) \dd \mu(\overline{z}) \\
= & \sum_{i,j} (g(\overline{z}_{i,j}) + O(\epsilon^{\theta} \|g\|_{\theta})) \int_{B_{i,j}} f(\rho_l(a) \overline{z}) \dd \mu(\overline{z}),
\end{array}$$
where $\overline{z}_{i,j} = (x_j, \xi_{i,j}(p) , (\xi_j(q))_{S \ni q \neq p}) \in B_{i,j}$.
\par By the argument in the proof for \ref{case_1}, to prove the exponential mixing result it suffices to show that
$$\int_{B_{i,j}} f(\rho_l(a)\overline{z}) \dd \mu (\overline{z}) = \mu(B_{i,j}) \left(\int_{\s(M)} f \dd \mu + O(e^{-\eta \|a\|} \|f\|_{\theta})\right),$$
for a constant $\eta >0$. To estimate the above integral, we cut $U_p(h)$ further as follows:
$$U_p(h) = \bigcup_{l} \kappa_l (p) N(p^{h} \Z_p),$$
where $\kappa_l(p) \in U_p(h)$. According to this we can cut $B_{i,j}$ into small pieces:
$$B_{i,j} = \bigcup_{l} B_{i,j,l},$$
where 
$$B_{i,j,l} := N(\Z) \setminus \kappa_l(p) N(p^{h} \Z_p) \times U(y_{i,j} , \epsilon) \times \prod_{q\neq p} N(q^{l_q(a)} \Z_q).$$
Then
$$\int_{B_{i,j}} f(\rho_l(a)\overline{z}) \dd \mu (\overline{z}) =\sum_{l} \int_{B_{i,j,l}} f(\rho_l(a)\overline{z}) \dd \mu(\overline{z}).$$
Now let us look at $\int_{B_{i,j,l}} f(\rho_l(a)\overline{z}) \dd \mu(\overline{z})$ more carefully. First note that we can choose $\kappa_l(p)$ to run over 
elements in
$$\Delta(w_1, \dots, w_{s_1}):=\{\exp(t_1  w_1 + \cdots + t_{s_1}  w_{s_1}) : t_i =0,p^{l_p}\cdot 1, p^{l_p}\cdot 2 , \dots, p^{l_p}(p^{h-l_p} -1) \}.$$
By Lemma \ref{lemma_chinese_remainder}, for each $w_i$, there exists $w_i(h) \in \mathfrak{n}(\Z)$ such that 
$w_i(h) \equiv w_i (\mathrm{mod}\quad p^h)$ and $w_i(h) \equiv \mathbf{0} (\mathrm{mod} \quad q^{l_q(a)})$ for $q \neq p$.
Suppose $\kappa_l(p) = \exp(t_1  w_1 + \cdots + t_{s_1}  w_{s_1})$, then direct calculation shows that
$$\int_{B_{i,j,l}} f(\rho_l(a) \overline{z}) \dd \mu (\overline{z})  = V \int_{U(y_{i,j} , \epsilon)} f(\rho_l(a) (n^{-1}_l \overline{z})) \dd \nu (\overline{z}),$$
where $V := \nu_p(N(p^h \Z_p)) \times \prod_{S \ni q\neq p} \nu_q(N(q^{l_q(a)} \Z_q))$ and $n_l := \exp(t_1 w_1(h) +\cdots + t_{s_1} w_{s_1} (h))$. Note that the diameter of $\rho_l(a)U(y_{i,j} , \epsilon)$ is less than $e^{-L'\|a\|}$, we have that
$$\begin{array}{cl} & \int_{B_{i,j,l}} f(\rho_l(a) \overline{z}) \dd \mu (\overline{z}) \\
 =  & V \nu(U(y_{i,j} , \epsilon)) (f(\rho_l(a) (n^{-1}_l y_{i,j})) + O(e^{-L' \theta \|a\|} \|f\|_{\theta})) \\
 = & \mu(B_{i,j,l})  (f(\rho_l(a) (n^{-1}_l y_{i,j})) + O(e^{-L' \theta \|a\|} \|f\|_{\theta})) . \end{array}$$
 Therefore, to show the exponential mixing result, it suffices to show that the following summation
 $$\frac{1}{p^{(h- l_p) s_1}}\sum_{t_1, \dots  t_{s_1}} f(\rho_l(a) (\exp(-(t_1 p^{l_p} w_1(h) + \cdots + t_{s_1} p^{l_p} w_{s_1}(h))) y_{i,j}))$$
 estimates $\int_{\s(M)} f \dd \mu$ with error $O(e^{-\eta \|a\|} \|f\|_{\theta})$. There is a factor of $\frac{1}{p^{(h - l_p) s_1}}$
 in front of the summation because $\mu(B_{i,j,l}) = \frac{\mu(B_{i,j})}{ p^{(h- l_p) s_1}}$. 
 \par First note that 
 $$ f(\rho_l(a) (\exp(-(t_1 p^{l_p} w_1(h) + \cdots + t_{s_1} p^{l_p} w_{s_1}(h))) y_{i,j})) = f(\exp(- \sum_{j=1}^{s_1} t_j p^{l_p} D\rho_l(a) w_j(h)) z_{i,j}),$$
 where $z_{i,j} := \rho_l(a) y_{i,j}$. By our previous discussion $D\rho_l(a) w_j = p^{-h} u_j$, for $j=1,\dots, s_1$. Let $u_j(h) \in \mathfrak{n}(\Z)$ be such that $u_j(h)  \equiv u_j (\mathrm{mod} \quad p^h)$ and $u_j(h) \equiv \mathbf{0} (\mathrm{mod} \quad q^{l_q(a)})$ for $S \ni q \neq p$. Then the difference between $\exp(- \sum_{j=1}^{s_1} t_j p^{l_p} D\rho_l(a) w_j(h)) z_{i,j}$ and $\exp(- p^{-h + l_p}\sum_{j=1}^{s_1} t_j u_j(h) ) z_{i,j}$ is in $N(\Z)$. Hence 
 $$ f(\exp(- \sum_{j=1}^{s_1} t_j p^{l_p} D\rho_l(a) w_j(h)) z_{i,j}) = f(\exp(- p^{-h+l_p}\sum_{j=1}^{s_1} t_j u_j(h) ) z_{i,j}).$$
 Therefore we reduce our task to proving that 
 $$\frac{1}{p^{h s_1}}\sum_{t_1, \dots, t_{s_1}} f(\exp(- p^{-h+l_p}\sum_{j=1}^{s_1} t_j u_j(h) ) z_{i,j}) = \int_{M} f \dd \nu + O(e^{-\eta \|a\|} \|f\|_{\theta}).$$
 For $\mathbf{t} = (t_1 , \dots , t_{s_1}) \in [0, p^{h-l_p} -1]^{s_1}$, let us denote 
 $$P(\mathbf{t}) := \exp(- p^{-h+l_p}\sum_{j=1}^{s_1} t_j u_j(h) ) z_{i,j},$$
 then it is easy to see that $P: [0, p^{h-l_p} -1]^{s_1} \rightarrow N(\R)$ is a polynomial map. Then by Theorem \ref{thm_equidistribution_polynomial}, either 
 $$\frac{1}{p^{(h-l_p) s_1}} \sum_{\mathbf{t} \in [0, p^{h-l_p} -1]^{s_1}} f(P(\mathbf{t})) = \int_{M} f \dd \nu + O(e^{-\eta \|a\|} \|f\|_{\theta}),$$
which is case \ref{first} in Theorem \ref{thm_equidistribution_polynomial}, or there exists $z \in \Z^l \setminus \{\mathbf{0}\}$ such that $\|z\| \ll e^{L_1 \eta \|a\|}$ and 
 $$ \dist (|\langle z , \partial_i \pi (P(\mathbf{t}))\rangle|, \Z) \ll e^{L_2 \eta \|a\|} / p^{h- l_p}$$
 for all $i=1, \dots, s_1$ and $\mathbf{t} \in [0, p^{h-l_p} -1]^{s_1}$, which is case \ref{second}. Suppose the latter holds. Since $N'(\Q_p)\setminus N(\Q_p) \cong \Q_p^l$ is abelian, we may identify $\Q_p^l$ with its Lie algebra. For $v \in \mathfrak{n}(\Q_p)$, let $\tilde{v} \in \Q_p^l$ denote the projection of $v$ onto $\Q_p^l$ under $D\pi$. Then it is easy to see that 
 $$\partial_i \pi(P(\mathbf{t})) = p^{-h+l_p} \tilde{u}_i(h)$$
 for $i=1,\dots, s_1$. 
 Then 
 $$\dist (|\langle z , \tilde{u}_i(h)/p^{h-l_p} \rangle|,\Z) \ll e^{L_2 \eta \|a\|} / p^{h-l_p},$$
 for $i=1,\dots, s_1$. In other words, in $(\mathrm{mod}\quad p^{h-l_p})$ sense, $|\langle z, \tilde{u}_j(h) \rangle| \ll e^{L_2 \eta \|a\|}$. Let 
 $k$ denote a positive integer with  $(k, p) =1$, In the construction above, we can replace the basis $w_1, \ldots w_{s_1}$ by  $k w_1, \ldots k w_{s_1}$ as the generalized eigenspace $\sigma ^{\chi _2 ^l}$ is invariant under multiplication by $k$. Note that multiplication by $k$ does not change the p-adic norm of any vector since $(k,p)=1$.  Tracking the argument, we eventually replace the polynomial $P(t)$ by $ P ^k(\mathbf{t}) := \exp(- p^{-h+l_p}\sum_{j=1}^{s_1} t_j k u_j(h) ) z_{i,j},$.  By the same argument as for $P$,  either 
 the sum $\frac{1}{p^{(h-l_p) s_1}} \sum_{\mathbf{t} \in [0, p^{h-l_p} -1]^{s_1}} f(P(\mathbf{t}))   =  \int_{M} f \dd \nu + O(e^{-\eta \|a\|} \|f\|_{\theta}),$ and we are done, or 
  there exists $z(k)\in \Z^l \setminus\{\mathbf{0}\}$ with $\|z (k) \| \ll e^{L_1 \eta \|a\| }$ and 
 in $(\mathrm{mod}\quad p^{h-l_p})$ sense, 
 $$|\langle z(k), k \tilde{u}_i(h) \rangle| \ll e^{L_2 \eta \|a\|}$$
 for $i=1,\dots, s_1$. We choose $\eta$ small enough such that 
  $$(e^{L_1 \eta \|a\|})^{l} \times (e^{L_2 \eta \|a\|})^{s_1} \leq (p^{h -l_p})^{1/8}.$$
  Then apply the pigeonhole principle to conclude that there exist $k_1 , k_2 \in [- p^{(h-l_p)/4}, p^{(h-l_p)/4}]$ such that
  $z(k_1) = z(k_2)$ and 
  $$\langle z(k_1) , k_1 \tilde{u}_i(h) \rangle \equiv \langle z(k_2), k_2 \tilde{u}_i(h) \rangle \mod p^{h-l_p}$$
for $i=1,\dots , s_1$. Let us denote $z(k_1) = z(k_2)$ by $z$. Then $\|z\| \ll e^{L_1 \eta \|a\|}$ and 
$$\langle z , (k_1 - k_2) \tilde{u}_i(h) \rangle \equiv 0 \mod p^{h-l_p}$$ 
for $ i =1 ,\dots , s_1$. Since $|k_1 - k_2| < 2 p^{(h-l_p)/4}$, we have that $(k_1 - k_2 , p^{h-l_p}) \leq  p^{(h-l_p)/4}$. This shows that 
for some $h' \geq 3 (h-l_p)/4 $, 
$$\langle z, \tilde{u}_i(h)\rangle \equiv 0 \mod p^{h'}$$
for $i=1,\dots, s_1$.
\par 
Let $W \subset N'(\Q_p)\setminus N(\Q_p) = \Q_p^l$ denote the projection of $\sigma^{\chi^l_2}$ on $\Q_p^l$ under $\pi$. Apparently $W$ is defined over $\Q^{al}_p$ and not defined over $\Q$. Let us choose a basis $\{\tilde{u}_i: i=1,2,\dots, s_1\}$ of $W$ with coordinates in $\Q^{al}_p$. By Lemma \ref{lemma:exist-p-adic-algebraic-vector}, there exists $\tilde{u} = \sum_{i=1}^{s_1} \alpha_i \tilde{u}_i \in W \cap (\Q^{al}_p)^l$ whose coordinates are linearly indepedent over $\Q$. Moreover, multiplying  by suitable powers of $p$, we can choose $\alpha_i \in \Z_p$ for $i=1,2,\dots s_1$. 
Then we have for any $z' \in \Z^l$, $\langle z', \tilde{u _i} \rangle \neq 0$.  
\par On the one hand, since $\|\langle z, \tilde{u}_i \rangle\|_p \leq p^{-h'} \leq e^{-3L'\|a\|/8}$ for $i=1,2,\dots, s_1$, we have that $\|\langle z , \tilde{u} \rangle\|_p \leq p^{-h'}$ (by choosing each $\alpha_i$ from $\Z_p$).
\par On the other hand, by Lemma \ref{lemma:p-adic-diophantine}, there exists a constant $L >0$ such that for all $z' \in \Z^l \setminus \{\mathbf{0}\}$,
$$\|\langle z' , \tilde{u} \rangle\|_p \gg \|z'\|^{-L} \gg e^{-LL_1 \eta \|a\|}.$$

The above two inequalities will lead to contradiction if $L L_1 \eta < 3L'/8$. This shows that 
$$\frac{1}{p^{(h-l_p) s_1}} \sum_{\mathbf{t} \in [0, p^{h-l_p} -1]^{s_1}} f(P(\mathbf{t})) = \int_{M} f \dd \nu +O(e^{-\eta \|a\|} \|f\|_{\theta}),$$
for a constant $\eta >0$. By the argument in the proof for \ref{case_1}, this finishes the proof of exponential mixing for \ref{case_2}.
\par This completes the proof of irreducible part of Theorem \ref{thm:exponential-mixing-intro}.
\end{proof}
\par In fact, using the same argument, one can prove the following slightly stronger result: 
\begin{proposition}
\label{prop_stronger_irreducible_mixing}
Under the irreducible assumption as above, let $\beta$ be an automorphism of $N$ defined over $\Q$ such that $\beta = \id$ on 
$N'\setminus N $, then for the same constant $\eta' >0$ as above,
$$\int_{\s(M)} f(\beta(\rho_l(a)\overline{z}) ) g(\overline{z}) \dd \mu (\overline{z}) = \int_{\s(M)} f \dd \mu \int_{\s(M)} g \dd \mu  + O(e^{-\eta' \|a\|} \|f\|_{\theta} \|g\|_{\theta}),$$
 for any $f\in C^{\theta}(M)$ and any $g \in C^{\theta}(\s(M))$.
\end{proposition}
\begin{proof}
From the proof above, we see that the basic scheme of the argument goes as follows:  We first cut the whole space $\s(M)$ into small pieces along real  and $p$-adic directions, and then apply Theorem \ref{thm_equidistribution_box} or Theorem \ref{thm_equidistribution_polynomial} prove each small piece estimates the integral of the whole space with exponentially small error. The obstruction of the effective equidistribution is that $ D\beta(\rho_l(a)(D\pi(\sigma^{\chi^l})))$ lies in a rational linear
subspace. Since $\beta$ acts trivially on $N'\setminus N$, the proposition follows from the argument above. 
\end{proof}
\subsection*{General Case} We will prove the statement in general using induction on the dimension of $M$. By \cite[Lemma 3.5]{gorodnik_spatzier2014}, if the action $\rho_l$ on $M$ is not irreducible, then there exists a $\rho_l$-invariant normal subgroup $N_1$ of $N$ defined over $\Q$ satisfying the following:
\begin{enumerate}
 \item The restriction of $\rho_l$ on $N_1$ is irreducible.
 \item $[N, N_1]\subset [N_1, N_1]$.
 \end{enumerate} 
 Then $Y:= N(\Z) N_1(\R) \setminus N(\R)$ and $Z := N_1(\Z) \setminus N_1(\R)$ are both compact nilmanifolds, and moreover, $M$ fibers over $Y$ with fibers isomorphic to $Z$. Let $\mu_Y$ and $\mu_Z$ denote the normalized measures on the solenoids $\s(Y)$ and $\s(Z)$ respectively, defined as the measure $\mu$ on $\s(M)$. Then for any continuous function $f$ defined on $\s(M)$, we have the following disintegration formula:
$$\int_{\s(M)} f \dd \mu  = \int_{\s(Y)} \int_{\s(Z)} f(\overline{z}\overline{y}) \dd \mu_Z(\overline{z}) \dd \mu_Y (\overline{y}).$$
Since $N_1$ is $\rho_l$-invariant, $\rho_l$ defines transformations of $Y$ and $Z$. Then 
$$\begin{array}{rcl}
\int_{\s(M)} f(\rho_l(a) \overline{x}) g(\overline{x}) \dd \mu (\overline{x}) & = & \int_{\s(Y)} \left( \int_{\s(Z)} f(\rho_l(a)(\overline{z}) \rho_l(a)(\overline{y})) g(\overline{z} \overline{y}) \dd \mu_Z (\overline{z}) \right) \dd \mu_Y (\overline{y}) \\ 
 & = & \int_{\mathcal{F}} \left( \int_{\s(Z)} f(\rho_l(a) (\overline{z}) \rho_l(a) (\overline{h}) ) g(\overline{z} \overline{h}) \dd \mu_Z (\overline{z}) \right)\dd m_{\mathcal{F}}(\overline{h}),
\end{array}$$
where $\mathcal{F} \subset N(\R) \times \prod_{p \in S} N(\Z_p)$ is a bounded fundamental domain for $\s(Y)$, and $m_{\mathcal{F}}$ denotes the measure on $\mathcal{F}$ induced by $\mu_Y$.
\begin{claim}
There exists a constant $\eta' >0$ such that for every $\overline{h} \in \mathcal{F}$, 
$$\int_{\s(Z)} f(\rho_l(a)(\overline{z}) \rho_l(a)(\overline{h}) ) g(\overline{z} \overline{h}) \dd \mu_Z (\overline{z}) = \int_{\s(Z)} f(\overline{z} \rho_l(a)(\overline{h}) ) \dd \mu_Z(\overline{z}) \int_{\s(Z)} g(\overline{z} \overline{h}) \dd \mu_Z(\overline{z}) + O(e^{-\eta' \|a\|} \|f\|_{\theta} \|g\|_{\theta}) .$$
\end{claim}
\begin{proof}[Proof of the claim]
Let us write 
$$\rho_l(a)(\overline{h}) =  \lambda \delta \alpha \text{ with } \alpha \in \mathcal{F}, \delta \in N_1(\R) \times \prod_{p \in S} N_1(\Z_p) \text{ and } \lambda \in N(\Z).$$
Then 
$$\int_{\s(Z)} f(\rho_l(a)(\overline{z}) \rho_l(a)(\overline{h})) g(\overline{z} \overline{h}) \dd \mu_Z(\overline{z}) = \int_{\s(Z)} f(\beta(\rho_l(a)(\overline{z})) \delta \alpha) g(\overline{z}\overline{h}) \dd \mu_Z(\overline{z}),$$
where $\beta$ denotes the transformation of $\s(Z)$ induced by the automorphism $m \mapsto \lambda^{-1} m \lambda, m \in N_1(\R) \times \prod_{p \in S} N(\Z_p)$. Then obviously $\beta$ is defined over $\Q$. Since $[N, N_1] \subset [N_1, N_1]$, 
$\beta$ acts trivially on $[N_1, N_1]\setminus N_1$. Let 
$$\phi_0(\overline{z}) : = f( \overline{z} \delta \alpha) \text{ and } \phi_1(\overline{z}) := g( \overline{z} \overline{h})   \text{ with } \overline{z} \in \s(Z).$$
Then we have 
$$\|\phi_0\|_{\theta} \ll \|f\|_{\theta} \text{ and } \|\phi_1\|_{\theta} \ll \|g\|_{\theta},$$
and 
$$\int_{\s(Z)} \phi_1 \dd \mu_Z = \int_{\s(Z)} f( \overline{z} \rho_l(a)(\overline{h}) ) \dd \mu_Z(\overline{z}).$$
Since every $\rho_l(a)$ is still ergodic when restricted to $\s(Z)$, we can apply Proposition \ref{prop_stronger_irreducible_mixing} to conclude that there exists a constant $\eta'>0$
such that 
$$\begin{array}{rcl}
\int_{\s(Z)} f(\rho_l(a)(\overline{z}) \rho_l(a)(\overline{h})) g(\overline{z}\overline{h}) \dd \mu_Z(\overline{z}) & = & \int_{\s(Z)} \phi_0(\beta(\rho_l(a)(\overline{z}))) \phi_1(\overline{z}) \dd \mu_Z(\overline{z}) \\
 & = & \int_{\s(Z)} \phi_0 \dd \mu_Z \int_{\s(Z)} \phi_1 \dd \mu_Z +O(e^{-\eta' \|a\|} \|\phi_0\|_{\theta} \|\phi_1\|_{\theta}) \\
 & = & \int_{\s(Z)} f(\overline{z} \rho_l(a)(\overline{h})) \dd \mu_Z(\overline{z}) \int_{\s(Z)} g(\overline{z} \overline{h}) \dd \mu_Z(\overline{z}) + O(e^{-\eta' \|a\|} \|f\|_{\theta} \|g\|_{\theta})
\end{array}$$
uniformly over $\overline{h} \in \mathcal{F}$. This proves the claim.
\end{proof}
Let us define $\overline{f}(\overline{y}) : = \int_{\s(Z)} f(\overline{z} \overline{y}) \dd \mu_Z (\overline{z}) $ and $\overline{g}(\overline{y}) := \int_{\s(Z)} g(\overline{z} \overline{y}) \dd \mu_Z(\overline{z})$ 
for $\overline{y} \in \s(Y)$. Then by the above claim, we conclude that
$$\int_{\s(M)} f(\rho_l(a)(\overline{x})) g(\overline{x}) \dd \mu (\overline{x}) = \int_{\s(Y)} \overline{f}(\rho_l(a)(\overline{y})) \overline{g}(\overline{y}) \dd \mu_Y (\overline{y}) +O(e^{-\eta' \|a\|} \|f\|_{\theta} \|g\|_{\theta}).$$
Since $\dim Y < \dim M$, the statement follows by induction.

\section{Coarse Lyapunov foliations and smooth leaves}
\label{smoothness of leaves}

In this section, we will prove that for every {\bf coarse Lyapunov exponent} $[\chi]$, the corresponding {\bf coarse Lyapunov distribution} $E^{[\chi]} := \bigoplus_{\chi \in [\chi]} E^{\chi}$ admits a \holder foliation with $\smooth$ leaves. We follow the argument developed by Rodriguez-Hertz and Wang \cite{hertz_wang2014} with minor modification for our setup.
\par We first make the following definition:
\begin{definition}
For $a \in \Z^k$, we say $\rho(a)$ is uniformly hyperbolic if there exists a constant $\lambda >1$ and a $\rho$-invariant splitting of the bundle $\mathcal{E} (\s(M)) = \bigcup_{\overline{z} \in \s(M)} \T_{\overline{z}}(M)$:
 $$\mathcal{E}(\s(M)) = \mathcal{E}_a^s \oplus \mathcal{E}_a^u,$$
 such that
 
 $$\begin{array}{ll} \|D\rho(a)(v)\| \leq \lambda^{-1} \|v\|  & \text{ for } v \in \mathcal{E}_a^s \\ 
 \|D\rho(a)(v)\| \geq \lambda \|v\| & \text{ for } v \in \mathcal{E}_a^u. \end{array} $$ 

\end{definition}
\begin{remark}
When $\rho(a)$ is uniformly hyperbolic, both $\mathcal{E}^s_a$ and $\mathcal{E}^u_a$ admit \holder foliations with $\smooth$ leaves. We denote the corresponding foliations by $\mathcal{W}^s_a$ and $\mathcal{W}^u_a$, respectively. For $\overline{z} \in \s(M)$ and $\square = s \text{ or } u$, we denote by $\mathcal{W}^{\square}_a(\overline{z})$ the leaf of $\mathcal{W}^{\square}_a$ passing through
$\overline{z}$.
\end{remark}
\par The basic idea to establish the smoothness of leaves goes as follows: we start with a real {\bf Weyl chamber} $\mathcal{C}_0 \subset \R^k$ with an element $a \in  \Z^k \cap \mathcal{C}_0$ such that $\rho(a)$ is uniformly hyperbolic. We shall prove that any real {\bf Weyl chamber } $\mathcal{C}$ adjacent to $\mathcal{C}_0$ also contains a uniformly hyperbolic element $a' \in \Z^k$. Therefore every real {\bf Weyl chamber} of the action $\rho$ contains a uniformly hyperbolic element. For every {\bf coarse Lyapunov exponent} $[\chi]$, we choose two adjacent 
{\bf Weyl chambers} $\mathcal{C}_1$ and $\mathcal{C}_2$ such that $\mathrm{ker}\chi$ is the only {\bf Weyl chamber wall} separating $\mathcal{C}_1$ and $\mathcal{C}_2$. Take $a_i \in \Z^k \cap \mathcal{C}_i$ such that $\rho(a_i)$ is uniformly hyperbolic. Without loss of generality, we may assume that $\chi(a_1) <0$ and $\chi(a_2) >0 $. Then the intersection $\mathcal{W}^s_{a_1}\cap \mathcal{W}^u_{a_2} $ defines a \holder foliation, and passing through every $\overline{z} \in \s(M)$, the intersection $\mathcal{W}^s_{a_1}(\overline{z}) \cap \mathcal{W}^u_{a_2}(\overline{z})$ is $\smooth$. From our assumption it is easily seen that $\mathcal{W}^{s}_{a_1} \cap \mathcal{W}^u_{a_2} $ corresponds to the distribution $E^{[\chi]}$. This proves that $E^{[\chi]}$ admits a \holder foliation with $\smooth$ leaves. To show that $\mathcal{C}$ contains a uniformly hyperbolic element, we basically follow the argument by Rodriguez Hertz and Wang \cite{hertz_wang2014} with minor modifications.

 \subsection{Correspondence between foliations of $\rho$ and $\rho_l$}
\par Let us fix a real {\bf Weyl chamber } $\mathcal{C}_0$ and choose $a \in \Z^k \cap \mathcal{C}_0$ such that $\rho(a)$ is uniformly hyperbolic.


 \par Recall that for each $\overline{z} \in \s(M)$, $M(\overline{z})$ denotes a {\em manifold slice} passing through $\overline{z}$. Then the leaves $\mathcal{W}^s_a(\overline{z})$ and $\mathcal{W}^u_a(\overline{z})$ are given by 
$$\mathcal{W}^s_a(\overline{z})= \{\overline{y} \in M(\overline{z}):\lim_{k \rightarrow +\infty} \dist(\rho(a)^k(\overline{y}), \rho(a)^k(\overline{z})) =0 \},$$
and 
$$\mathcal{W}^u_a(\overline{z}) = \{\overline{y} \in M(\overline{z}):\lim_{k \rightarrow +\infty} \dist(\rho(a)^{-k}(\overline{y}), \rho(a)^{-k}(\overline{z})) =0\}.$$


\par Now, let us turn our attention to the affine action $\rho_l(a)$. Let $\mathfrak{g}^s_a(\R)$ (and $\mathfrak{g}^u_a(\R)$, respectively) denote the direct sum of $\sigma^{\chi^l}$ such that $\chi^l \in T(\R)$, and $\chi^l(a)<0$ 
(and $\chi^l(a)>0$, respectively). From Remark \ref{remark_bracket}, we can see that $\mathfrak{g}^s_a(\R)$ and 
$\mathfrak{g}^u_a(\R)$ are both Lie subalgebras of $\mathfrak{n}(\R)$. Let $G^s_a$ 
and $G^u_a$ denote the corresponding Lie subgroups.  Then it is easy to see that $\mathfrak{g}^s_a(\R)$ and $\mathfrak{g}^u_a(\R)$ correspond to the stable and unstable foliations of $M$ with respect to the action
$\rho_l(a)$,  and moreover, the stable (and unstable) leaf passing through any point $\overline{z} \in \s(M)$ is $\overline{z} G^s_a$ (and $\overline{z} G^u_a$ respectively). 
\par Because $\phi$ conjugates $\rho$ to $\rho_l$, we have that 
$$\begin{array}{ccc}
\phi(\mathcal{W}^s_a(\overline{z})) = \phi(\overline{z})G^s_a, & \text{ and } & \phi(\mathcal{W}^u_a(\overline{z})) = \phi(\overline{z})G^u_a.
\end{array}$$
In particular, we have that $\dim E^s_a = \dim \mathfrak{g}^s_a$ and 
$\dim E^u_a = \dim \mathfrak{g}^u_a$. Therefore, for every real {\bf Lyapunov exponent} $\chi^l$ for $\rho_l$,  $\chi^l(a) \neq 0$, 
and 
$$\mathfrak{n}(\R) = \mathfrak{g}^s_a(\R) \oplus \mathfrak{g}^u_a(\R).$$

Let $\mathcal{C}$ be a real {\bf Weyl chamber} adjacent to $\mathcal{C}_0$. Let $\mathrm{ker} \chi^l$ denote the real {\bf Weyl chamber wall} separating $\mathcal{C}_0$ and $\mathcal{C}$. Then by Remark \ref{remark_bracket}, $\sigma := \sigma^{[\chi^l]}$ is a Lie subalgebra of $\mathfrak{n}(\R)$. Let $V := V^{[\chi^l]} = \exp (\sigma^{[\chi^l]})$ denote the corresponding Lie subgroup.
\par Without loss of generality, we may assume that $\chi^l(a) <0$, then $V \subset G^s_a$. Let us define the strong stable subspace
of $\rho_l(a)$ by 
$$\mathfrak{g}^{ss}_a := \bigoplus_{\begin{array}{l} \sigma^{[\chi_1^l] } \neq \sigma \\ \sigma^{[\chi^l_1]} \subset \mathfrak{g}^s_a \end{array}} \sigma^{[\chi^l_1]}.$$
\begin{lemma}[see~{\cite[Lemma 3.1]{hertz_wang2014}}]
\label{lemma_hertz_wang_31}
$\quad$
\par 1. $\mathfrak{g}^{ss}_a$ is a Lie subalgebra.
\par 2. $\sigma \oplus \mathfrak{g}^u$ is a Lie subalgebra.
\par 3. $[\sigma , \mathfrak{g}^{ss}_a] \subset \mathfrak{g}^{ss}_a$ . 
\end{lemma}

Let $G^{ss}_a : = \exp \mathfrak{g}^{ss}_a$ be the Lie subgroup corresponding to $\mathfrak{g}^{ss}_a$. On the level of Lie algebra, we have the following decomposition:
 $$\mathfrak{n}(\R) = \mathfrak{g}^{u}_a \oplus \sigma \oplus \mathfrak{g}^{ss}_a.$$
 On the level of Lie group, the following lemma is proved in \cite{hertz_wang2014}:
 \begin{lemma}[see~{\cite[Corollary 3.3]{hertz_wang2014}}]
 \label{lemma_hertz_wang_33}
 $\quad$
 \par 1. The multiplication map 
 $$\begin{array}{rcl}
  G^{u}_a \times V \times G^{ss}_a & \rightarrow & G \\
   (g_u, g_V , g_{ss}) & \mapsto & g_u g_V g_{ss}
 \end{array}$$
 is a $\smooth$ diffeomorphism.
 \par 2. The multiplication map 
 $$\begin{array}{rcl}
 V\times G^{ss}_a & \rightarrow & G^s_a \\
   (g_V,g_{ss}) & \mapsto & g_V g_{ss}
 \end{array}$$
 is a $\smooth$ diffeomorphism.
 \end{lemma}
 \begin{remark}
In \cite{hertz_wang2014}, the order of the multiplication is $G^{ss}_a \times V \times G^u_a \rightarrow N(\R)$. The same proof works for our order here. We make this change in this paper because here $M = N(\Z) \setminus N(\R)$ while in \cite{hertz_wang2014} 
$M = N(\R)/ N(\Z)$. Also, later in this section, we will make similar changes due to this reason.
\end{remark}
From the above lemma, any $g \in N(\R)$ can be uniquely written as $g_{u} g_V g_{ss}$. Define $g_s := g_{V} g_{ss}$. We call $g_{ss}$ ($g_V$, $g_u$ and $g_s$ respectively) the projection of $g$ onto $G^{ss}_a$ ($V$, $G^u_a$ and $G^s_a$ respectively).

 
 Rodriguez Hertz and Wang \cite{hertz_wang2014} proved several results on the $G^{ss}_a \times V \times G^u_a $ coordinate of $N(\R)$. We sum up them in the following proposition. We refer to \cite{hertz_wang2014} for proofs.
\begin{proposition}
\label{proposition_group_coordinate}
$\quad$ \par
\begin{itemize}
\item For any given $g \in N(\R)$, the restriction of the map $h \mapsto (gh)_u$ to $G^u_a$ is a $\smooth$ diffeomorphism from $G^u_a$ to itself (cf. \cite[Corollary 3.5]{hertz_wang2014}).
\item For any $g_1, g_2 \in N(\R)$, $(g_1 g_2)_V = (g_1 (g_2)_u)_V ( g_2)_V$. In particular, if $g_2 \in G^s_a$, then $(g_1 g_2)_V = (g_1)_V (g_2)_V$ (cf. \cite[Corollary 3.6]{hertz_wang2014}).
\end{itemize}
\end{proposition}

\par Let us recall an argument from \cite{kalinin_fisher_spatzier2013} and \cite{hertz_wang2014} concerning the choice of $a \in \Z^k$ and estimate of {\bf Lyapunov exponents}.
\par Recall that $a \in \mathcal{C}_0$ where $\mathcal{C}_0$ denotes a {\bf Weyl chamber} of the action $\rho_l$, $\mathcal{C}$ denotes a {\bf Weyl chamber} adjacent to $\mathcal{C}_0$, and $[\chi^l]$ denotes the {\bf coarse Lyapunov exponent} such that $\mathrm{ker} \chi$ separates $\mathcal{C}_0$ and $\mathcal{C}$. 
\par Recall that $\Sigma \cong \Z ^2$ denotes the subgroup of $\Z^k$ given in Corollary \ref{cor:exponential-mixing}. For any $\xi >0$, we may choose $b \in \Sigma\cap \mathcal{C}$
such that $|\chi^l_1(b)| < \xi \|b\|$ for all 
$\chi^l_1 \in [\chi^l] \text{ or } [-\chi^l]$. Then the restriction of $\rho_l^{-1}(b)$ on $V$ is contracting, and for any $v \in \sigma$ and $n \in \N$, $\|D\rho_l(nb) v\| = O(e^{ n\xi\|b\| })$. For $\xi >0$ small enough, we have that for any 
$[\chi^l_2] \neq [\chi^l] \text{ or } [-\chi^l]$, $\chi^l_2 (a)$ and $\chi^l_2(b)$ have the same sign. Therefore, if 
$\sigma^{[\chi^l_2]} \subset \mathfrak{g}^{ss}_a$, then $\chi^l_2(b) <0$. Combining this with the fact that $|\chi^l_1(b)| < \xi \|b\|$ for all 
$\chi^l_1 \in [\chi^l] \text{ or } [-\chi^l]$, we conclude that for any $v \in \mathfrak{g}^s_a$ and any $n \in \N$, $\|D\rho_l(nb)v\| =  O(e^{ n\xi\|b\| })$. Since the foliation $\mathcal{W}^s_a$ corresponds to $G^s_a$ via $\phi$, and $\phi$ is $\theta$-H\"{o}lder, 
we conclude that for any $w \in \mathcal{E}^s_a$, and any $n \in \N$, $\|D\rho(nb) w\| = O(e^{n \xi' \|b\|})$ where $\xi' = \xi/\theta$. 
\par Summing up the argument above, we have the following proposition:
\begin{proposition}[see~{\cite[Lemma 3.8]{hertz_wang2014}}]
\label{prop_estimate_lyapunov}
For any $\xi >0$, there exists $b \in \Sigma\setminus\{\mathbf{0}\}$ such that the restriction of $\rho_l^{-1}(b)$ on $V$ is contracting, and for any 
$w \in \mathcal{E}^s_a$ and any $n \in \N$, 
$$\|D\rho(n b) w\| = O(e^{n \xi \|b\|}),$$
in other words, $\log \|D\rho(n b) |_{\mathcal{E}^s_a}\| \leq n\xi \|b\|$.
\end{proposition}

\subsection{The cohomological equation \uppercase\expandafter{\romannumeral2}}
\label{cohomological_equation_2}
Let $h: \s(M) \rightarrow N(\R)$ be as defined in \S \ref{cohomologicalequation1}. Note that the value of $h(\overline{z})$ only depends on the projection $z$ of $\overline{z}$ on $M$. Therefore we may regard $h$ as a $\smooth$ map from $M$ to $N(\R)$.  Let $h_u$ and $h_V$ denote the projection of $h$ on $G^u_a$ and $V$ respectively in the $G^{ss}_a \times V \times G^u_a$ coordinate. We want to get the cohomological equation of $h_V$.
\par Let us fix a constant $\xi >0$. By Proposition \ref{prop_estimate_lyapunov}, there exists $b \in \Sigma \setminus \{\mathbf{0}\}$ such that $\rho_l^{-1}(b)$ contracts $V$ and $\|D\rho(n b) w\| = O(e^{n \xi \|b\|})$ for all $w \in \mathcal{E}^s_a$ and 
$n \in \N$.
\par Applying Lemma \ref{lemma_cohomological} to $b$, we have that 
$$h(\overline{z}) = Q'_b(\overline{z}) \rho_l^{-1}(b) h(\rho(b)\overline{z}),$$
for some $\smooth$ map $Q'_b : \s(M) \rightarrow N(\R)$. By Proposition \ref{proposition_group_coordinate}, $(g_1 g_2)_V = (g_1 (g_2)_u)_V (g_2)_V$. 
Projecting both sides of the equation above to $V$, we have 
\begin{equation}
\label{cohomological_equation_V}
h_V(\overline{z}) = (Q'_b(\overline{z}) \rho_l^{-1}(b) h_u(\rho(b)\overline{z}))_V  \rho_l^{-1}(b) h_V(\rho(b)\overline{z}).
\end{equation}

\begin{notation}
We borrow the following notation from \cite{hertz_wang2014}.
\par For $\overline{x} \in \s(M) (\text{ or } N(\R))$ and $\epsilon >0$, let $B_{\epsilon}(\overline{x})$ denote the ball inside $M(\overline{x}) (\text{ or } N(\R))$ centered at $\overline{x}$ with radius $\epsilon$. For a Lie subgroup 
$F$ of $N(\R)$, $ \mathfrak{f} \in F$ and $\epsilon >0$, we denote by $B^F_{\epsilon}(\mathfrak{f})$ the ball inside $F$ centered at $\mathfrak{f}$ with radius $\epsilon$.
\par Let us fix a constant $\delta >0$ such that for $\overline{x} , \overline{y} \in \s(M)$ belonging to the same manifold slice and $\dist (\overline{x},\overline{y}) \leq \delta$, there exists a unique element 
$p(\overline{x},\overline{y}) \in B_{\delta} (e)$ with $\overline{y} = \overline{x} p(\overline{x},\overline{y})$. It is easy to see that the map
$$p: \{(\overline{x},\overline{y}) \in \s(M) \times \s(M) : \overline{x} \text{ and } \overline{y} \text{ belong to the same manifold slice, and } \dist (\overline{x}, \overline{y}) \leq \delta\} \rightarrow N(\R)$$
is $\smooth$. 
\par For $\overline{x} , \overline{y} \in \s(M)$ belonging to the same manifold slice and $\dist(\overline{x}, \overline{y}) \leq \delta$, we write $\phi(\overline{y}) = \phi(\overline{x}) H_{\overline{x}}(\overline{y})$ where $H_{\overline{x}}(\overline{y}) : = h^{-1}(x) p(x,y) h(y)$. It is easy to see that the map 
$$(\overline{x},\overline{y}) \mapsto H_{\overline{x}}(\overline{y})$$
is $\theta$-\holder in the pair $(\overline{x},\overline{y})$.
\end{notation}

We first prove the following proposition:
\begin{proposition}[see~{\cite[Corollary 3.14]{hertz_wang2014}}]
\label{proposition_unstable_smooth}
$$h_u  \in C^{\infty, \theta}_{\mathcal{W}^s_a}.$$
\end{proposition}
\begin{proof} The proof we present here follows  the proof of \cite[Lemma 3.13, Corollary 3.14]{hertz_wang2014} with minor modification.
\par We write $p^{-1}(\overline{x},\overline{y}) h(\overline{x}) = (p^{-1}(\overline{x},\overline{y}) h(\overline{x}))_u (p^{-1}(\overline{x},\overline{y}) h(\overline{x}))_s$ and $h(\overline{y}) = h_u(\overline{y}) h_s(\overline{y})$. Then 
$$H_{\overline{x}}(\overline{y}) =  (p^{-1}(\overline{x},\overline{y}) h(\overline{x}))_s^{-1} (p^{-1}(\overline{x},\overline{y}) h(\overline{x}))_u^{-1}  h_u(\overline{y}) h_s(\overline{y}).$$
$H_{ \overline{x} }(\overline{y}) \in G^s_a$ if and only if $(p^{-1}(\overline{x},\overline{y}) h(\overline{x}))_u^{-1}  h_u(\overline{y}) = e$, i.e., $h_u(\overline{y}) = (p^{-1}(\overline{x},\overline{y}) h(\overline{x}))_u$. Since $G^s_a$ corresponds to the foliation $\mathcal{W}^s_a$ via the conjugacy $\phi$. Therefore, near $\overline{x}$ the leaf $\mathcal{W}_a^s(\overline{x})$ is defined by 
\begin{equation}\label{equation_stable} h_u(\overline{y}) =  (p^{-1}(\overline{x},\overline{y}) h(\overline{x}))_u.\end{equation}
Since on $\mathcal{W}_a^s(s)$, $p(\overline{x},\overline{y})$ is $\smooth$ and $h(\overline{x})$ is constant, we have that $h_u(\overline{y})$ is $\smooth$ along $\mathcal{W}^s_a(\overline{x})$.
\par It remains to check that the partial derivatives along $\mathcal{W}^s_a$ vary \holder continuously. We fix a neighborhood $\Omega$ of $\overline{x}_0 \in \s(M)$ in the manifold slice $M(\overline{x}_0)$ such that $ \phi(\Omega)$ is of form $\phi(\overline{x}_0) B^{G^u_a}_{\epsilon}(e) B^{G^s_a}_{\epsilon}(e)$. Then every $\overline{y} \in \Omega$ can be projected to some $\overline{x} = \overline{x}(\overline{y}) \in \Omega \cap \mathcal{W}^u_a(\overline{x}_0)$ along $\mathcal{W}^s_a$. Since $\mathcal{W}^s_a$ is a \holder foliation, we have that the map $\overline{y} \mapsto \overline{x}(\overline{y})$ is \holder in $\overline{y}$. Thus the map $\overline{y} \mapsto h(\overline{x}(\overline{y}))$ is also \holder. By (\ref{equation_stable}), partial derivatives of $h_u(\overline{y})$ along $\mathcal{W}^s_a$ $\smooth$ depend on $\overline{x}$ and $h(\overline{x})$, and thus are \holder continuous in $\overline{y}$.
\par This completes the proof. 
 \end{proof}
 
 Now let us get back to the cohomological equation (\ref{cohomological_equation_V}). Let $\tilde{h}_V := \log h_V$ and 
 $$\Psi := \log (Q'_b(\overline{z}) \rho_l^{-1}(b) h_u(\rho(b)\overline{z}))_V.$$ By Proposition \ref{proposition_unstable_smooth}, $\Psi \in C^{\infty, \theta}_{\mathcal{W}^s_a}$. Then (\ref{cohomological_equation_V}) can be rewritten as 
 $$\exp \tilde{h}_V = \exp \Psi \exp (\rho_l^{-1}(b) \tilde{h}_V \circ \rho(b)).$$
 By Baker-Campbell-Hausdorff formula, we have that
 \begin{equation}
 \label{equation_bch}
 \begin{array}{rcl}
 \tilde{h}_V  & = & \rho_l^{-1}(b) \tilde{h}_V \circ \rho(b) + \Psi + \frac{1}{2} [\Psi , \rho_l^{-1}(b) \tilde{h}_V \circ \rho(b)] \\
  & &   - \frac{1}{12}[ \rho_l^{-1}(b) \tilde{h}_V \circ \rho(b),  [\Psi , \rho_l^{-1}(b) \tilde{h}_V \circ \rho(b)] ]\\
  & &+ \frac{1}{12} [ \Psi, [\Psi , \rho_l^{-1}(b) \tilde{h}_V \circ \rho(b)]] +\cdots 
 \end{array}
 \end{equation}
where there are only finitely many terms on the right hand side since both sides of the equation belong to the Lie algebra $\sigma$ of $V$ which is nilpotent. Consider the derived series of $\sigma$:
$$\sigma = \sigma_0 \supset \sigma_1 \supset \cdots \supset \sigma_l = \{0\}.$$ For $i=1, 2, \dots, l$, let 
$$\pi_i : \sigma \rightarrow \sigma_i\setminus \sigma$$
denote the canonical projection. Let $\tilde{h}_i := \pi_i \circ \tilde{h}_V$ and $\Psi_i := \pi_i \circ \Psi$.
\par Projecting the equation to $\sigma_1\setminus \sigma$, we will get the following linearized equation:
\begin{equation}
\label{cohomological_equation_linear}
\tilde{h}_1 = \rho_l^{-1}(b) \tilde{h}_1\circ \rho(b) + \Psi_1.
\end{equation}
We first prove the following lemma on linearized cohomological equations:
\begin{proposition}[see~{\cite[Proposition 3.15]{hertz_wang2014}}]
\label{prop_linear_equation}
Let $L$ be a vector space and $B: L \rightarrow L$ be a linear isomorphism such that $\|B^{-i}\|$ is uniformly bounded for all $i \geq 0$. For $\xi >0$, let $b \in \Sigma\setminus\{\mathbf{0}\}$ be the element given by Proposition \ref{prop_estimate_lyapunov}.  Suppose $\psi: \s(M) \rightarrow L$ is in $C^{\infty, \theta}_{\mathcal{W}^s_a}$ and $f: M \rightarrow L$ is \holder continuous and solves the linear equation:
\begin{equation}
\label{equation_1}
f = B^{-1} f \circ \rho(b) + \psi,
\end{equation}
then there exists $\xi_0 >0$ such that for $0< \xi \leq \xi_0$,  $f \in C^{\infty, \theta}_{\mathcal{W}^s_a}$.
\end{proposition}

\begin{proof}
We first claim that in order to show the lemma, it suffices to show the lemma assuming that the integral $\int_{\s(M)} f \dd \tilde{\mu} =0$. In fact, let 
$\overline{f}:= \int_{\s(M)} f \dd \tilde{\mu}$ and let $f_1 := f - \overline{f}$. Since $\overline{f}$ is a constant function, to show $f \in C^{\infty, \theta}_{\mathcal{W}^s_a}$, it suffices to show that $f_1 \in C^{\infty, \theta}_{\mathcal{W}^s_a}$. We have that $\int_{\s(M)}f_1\dd \tilde{\mu} =0$ and $f_1$ satisfies the following equation:
$$f_1 = B^{-1} f_1 \circ \rho(b) +\psi',$$
where $\psi' = \psi + B^{-1} \overline{f} \circ \rho(b) - \overline{f} \in C^{\infty, \theta}_{\mathcal{W}^s_a}$. This proves the claim. Therefore we may assume that 
$\int_{\s(M)} f \dd \tilde{\mu} =0$.
\par By iterating (\ref{equation_1}), we have that 
$$f = \sum_{j=0}^i B^{-j} \psi \circ \rho(jb)  + B^{-i} f \circ \rho(ib).$$
We claim that 
$$f = \sum_{i=0}^{\infty} B^{-j} \psi \circ \rho(jb) $$
in the sense of distributions. To show this, it suffices to show that for any $g \in C^{\infty}(M)$, 
$$\lim_{i \rightarrow \infty} \int_{\s(M)} B^{-i} \psi (\rho(ib)\overline{x}) g (\overline{x}) \dd \tilde{\mu}(\overline{x}) =0.$$
In fact, since $\tilde{\mu}$ is $\rho$-invariant and $f$ has zero average with respect to $\tilde{\mu}$,  we have 
$$\int_{\s(M)} B^{-i} f \circ \rho(ib) \dd \tilde{\mu} = B^{-i} \int_{\s(M)} f \dd \tilde{\mu} =0.$$
Since $\int_{\s(M)} f \dd \tilde{\mu} = \int_{\s(M)} B^{-1} f \circ \rho(b) \dd \tilde{\mu} + \int_{\s(M)} \psi \dd \tilde{\mu}$, we 
conclude that $\int_{\s(M)} \psi \dd \tilde{\mu} =0$.
Therefore, by Corollary \ref{cor:exponential-mixing}, there exist constants $C>0$ and $\eta'>0$ such that for all $g \in C^{\infty}(M)$,
$$\left |\int_{\s(M)} B^{-i} 
\psi (\rho(ib)\overline{x}) g(\overline{x}) \dd \tilde{\mu} \right| \leq C \|B^{-i}\| \|\psi\|_{\theta} \|g\|_{\theta} e^{-\eta' i \|b\|}.$$
Since $\|B^{-i}\|$ is uniformly bounded, we have that $|\int_{\s(M)} B^{-i} \psi (\rho(ib)\overline{x}) g(\overline{x}) \dd \tilde{\mu} | \rightarrow 0$ as $i \rightarrow +\infty$, which proves the claim.
\par By \cite[Theorem 1.1]{rauch_taylor2005} and its variations proved by Fisher-Kalinin-Spatzier \cite[Theorem 8.3.1]{kalinin_fisher_spatzier2013} and Rodriguez Hertz-Wang \cite[Theorem A.1]{hertz_wang2014}, to show $f \in C^{\infty, \theta}_{\mathcal{W}^s_a}$, it suffices to show $f \in C^{\infty, \theta, \ast}_{\mathcal{W}_a^s}$, i.e., for all $ k \in \N$, 
$\partial^k_{\mathcal{W}^s_a} f \in (C^{\theta}(M))^{\ast}$.
\par Given $\phi \in C^{\theta}(M)$, we have
$$\langle \partial^k_{\mathcal{W}^s_a} f , \phi \rangle = \sum_{i=0}^{\infty} \langle \partial^k_{\mathcal{W}^s_a} (B^{-i} \psi \circ \rho(ib)), \phi \rangle.$$
Since every $\partial^k_{\mathcal{W}^s_a} (B^{-i} \psi \circ \rho(ib)) \in C^{\infty, \theta}_{\mathcal{W}^s_a}$ is a \holder continuous function, the term 
$$ \langle \partial^k_{\mathcal{W}^s_a} (B^{-i} \psi \circ \rho(ib)), \phi \rangle = \int_{\s(M)}  \partial^k_{\mathcal{W}^s_a} (B^{-i} \psi \circ \rho(ib)) \phi \dd \tilde{\mu}.$$
\par Fix a compactly supported positive $\smooth$ bump function $\delta$ on $\mathfrak{n}(\R)$ supported on a neighborhood around $\mathbf{0}$. For small $\epsilon >0$, define on $N(\R)$ a function
$$\delta_{\epsilon} (x) = c_{\epsilon} \delta(\frac{\log x}{\epsilon})$$
where $c_{\epsilon} >0$ is chosen such that $\int_{N(\R)} \delta_{\epsilon} (g) \dd g =1$. Let $\phi_{\epsilon} := \phi \ast \delta_{\epsilon}$. By standard facts on convolutions, we have the following hold:
\begin{enumerate}
\item $\phi_{\epsilon}$ is $\smooth$.
\item $\|\phi - \phi_{\epsilon}\|_{\infty} \leq a_0 \epsilon^{\theta} \|\phi\|_{\theta}$ for a constant $a_0 >0$.
\item There exists a constant $c_k >0$ such that $\|\phi_{\epsilon}\|_{C^k} \leq c_k \epsilon^{-\dim M -k} \|\phi\|_{\infty}$.
\end{enumerate}
By Corollary \ref{cor:exponential-mixing} applied to $\psi$ and $\partial^k_{\mathcal{W}^s_a} \phi_{\epsilon}$, we have
\begin{equation}
\begin{array}{rcl}
|\langle \partial^k_{\mathcal{W}^s_a} (\psi \circ \rho(ib)), \phi_{\epsilon} \rangle| & = & |\langle \psi \circ \rho(ib) , \partial^k_{\mathcal{W}^s_a} \phi_{\epsilon} \rangle| \\
& \leq & a_1 \|\psi\|_{\theta} \|\partial^k_{\mathcal{W}^s_a} \phi_{\epsilon}\|_{\theta} e^{-\eta' i \|b\|} \\
& \leq & a_1  \|\psi\|_{\theta} \|\phi_{\epsilon}\|_{C^{k+1}} e^{-\eta' i \|b\|} \\
& \leq & a_1 c_{k}  \|\psi\|_{\theta} \|\phi\|_{\infty} \epsilon^{-\dim M -k -1} e^{-\eta' i \|b\|} \\
& \leq & C_1 \epsilon^{-\dim M -k -1} e^{-\eta' i \|b\|} \|\phi\|_{\theta},
\end{array}
\end{equation} 
where $a_1 >0$ and $\eta' >0$ are constants from Corollary \ref{cor:exponential-mixing}, and $C_1 := a_1 c_k \|\psi\|_{\theta} $.
\par We also need to estimate $|\langle \partial^k_{\mathcal{W}^s_a} (\psi \circ \rho(i a)), \phi - \phi_{\epsilon} \rangle|$:
\begin{equation}
\begin{array}{cl}
& |\langle \partial^k_{\mathcal{W}^s_a} (\psi \circ \rho(ib)) , \phi - \phi_{\epsilon} \rangle| \\
\leq & \|\partial^k_{\mathcal{W}^s_a} (\psi \circ \rho(ib)) \|_{\infty} \| \phi - \phi_{\epsilon}\|_{\infty} \\
\leq & a_0 \|\partial^k_{\mathcal{W}^s_a}\psi\|_{\infty} \|\partial^k_{\mathcal{W}^s_a} \rho(ib)\|_{\infty} \epsilon^{\theta} \|\phi\|_{\theta}.
\end{array}
\end{equation}
By \cite[Lemma 3.6]{kalinin_fisher_spatzier2013}, 
$$\|\partial^k_{\mathcal{W}^s_a} \rho(ib)\|_{\infty}  = O(\|D\rho(b)|_{\mathcal{E}^s_a}\|^{ik} i^T \|\partial^k_{\mathcal{W}^s_a} \rho(b)\|^T_{\infty}),$$
where $T>0$ depends on $k$ and $\dim \mathcal{W}^s_a$.
By Proposition \ref{prop_estimate_lyapunov}, $\|D\rho(b)|_{\mathcal{E}^s_a} \|  = O ( e^{\xi \|b\|})$. Therefore, we have
\begin{equation}
\begin{array}{rcl}|\langle \partial^k_{\mathcal{W}^s_a} (\psi \circ \rho(ib)), \phi - \phi_{\epsilon} \rangle| & \leq & a_2 e^{\xi i k \|b\|} i^T \epsilon^{\theta} \|\phi\|_{\theta} \\ 
 & \leq & a_2 e^{2 \xi i k \|b\|} \epsilon^{\theta} \|\phi\|_{\theta} \end{array}\end{equation}
for a constant $a_2 >0$ depending on $k$, $b$, $\psi$ and $\dim \mathcal{W}^s_a$. 
\par Let 
$$\epsilon = \exp \left(-\frac{i \|b\|(\eta' + 2 k \xi)}{\dim M + k+1 + \theta}\right),$$
then $\epsilon^{-\dim M - k -1 } e^{- \eta' i \|b\|}$ and $\epsilon^{\theta} e^{2\xi i k \|b\|}$ are both equal to
$$\exp \left(  \frac{i \|b\| [2 \xi k (\dim M + k +1) - \theta \eta']}{\theta + \dim M + k +1} \right).$$
For $\xi$ small enough, we will have 
$$ \eta_1 :=  - \frac{ 2 \xi k (\dim M + k +1) - \theta \eta'}{\theta + \dim M + k +1} >0. $$
By the estimates above, we have that 
$$|\langle \partial^k_{\mathcal{W}^s_a} (\psi\circ \rho(ib)), \phi  \rangle | \leq a_3 e^{- \eta_1 i \|b\| } \|\phi\|_{\theta}$$
for a constant $a_3 >0$. Since by assumption there exists a constant $a_4 >0$ such that $\|B^{-i}\| \leq a_4$ for all $i \in \N$, we will have that 
\begin{equation}
\begin{array}{rcl}
 |\langle \partial^k_{\mathcal{W}^s_a} f, \phi \rangle| &\leq& \sum_{i=0}^{\infty} |\langle \partial^{k}_{\mathcal{W}^s_a} (B^{-i} \psi \circ \rho(ib)), \phi \rangle| \\
  & \leq &  \sum_{i=0}^{\infty} \|B^{-i}\| |\langle \partial^k_{\mathcal{W}^s_a}(\psi \circ \rho(ib)), \phi \rangle| \\
  & \leq & a_4 \sum_{i=0}^{\infty} |\langle \partial^k_{\mathcal{W}^s_a} (\psi \circ \rho(ib)), \phi \rangle| \\
  & \leq & a_4 a_3 \sum_{i=0}^{\infty} e^{-\eta_1 i \|b\|} \|\phi\|_{\theta} \\
  & \leq & a_5 \|\phi\|_{\theta},
\end{array}
\end{equation}
where $a_5 = a_3 a_4 \frac{1}{1 - e^{-\eta_1 \|b\|}}$. This proves that $\partial^k_{\mathcal{W}^s_a} f \in (C^{\theta}(M))^{\ast}$. By our previous discussion, this completes the proof.
\end{proof}

\begin{proposition}[see~{\cite[Lemma 3.20]{hertz_wang2014}}]
\label{prop_smooth_hV}
For $i=0,1,\dots, l$, $\tilde{h}_i \in C^{\infty, \theta}_{\mathcal{W}^s_a}$. In particular, $\tilde{h}_V = \tilde{h}_l \in  C^{\infty, \theta}_{\mathcal{W}^s_a}$.
\end{proposition}
\begin{proof}
Let $\xi_0 >0$ be the constant given by Proposition \ref{prop_linear_equation}. For $\xi \in (0, \xi_0]$, let $b \in \Sigma\setminus \{\mathbf{0}\}$ be the element given by Proposition \ref{prop_estimate_lyapunov}.
\par Let us prove this lemma by induction on $i$. For $i=0$, the statement is trivial since $\sigma_0 = \sigma$. For $i \geq 1$, assume the lemma holds for all $j < i$. We want to show that the lemma holds for $i$.
\par Projecting (\ref{equation_bch}), we have the following equation for $\tilde{h}_i$:
\begin{equation}
\label{equation_hi}
\tilde{h}_i = \rho_l^{-1}(b) \tilde{h}_i \circ \rho(b) +\Psi_i + \frac{1}{2}[\Psi_i , \rho_l^{-1}(b) \tilde{h}_i \circ \rho(b)] + \cdots .
\end{equation}
Fix a subspace of $\mathfrak{z} \subset \sigma_i\setminus \sigma$ such that $\mathfrak{z} \oplus (\sigma_i\setminus\sigma_{i-1}) = \sigma_i\setminus \sigma$. According to this decomposition we may write $\tilde{h}_i = \tilde{h}_{\mathfrak{z}} + \tilde{h}^{\perp}_{\mathfrak{z}}$. Note that the canonical projection from $\mathfrak{z}$ to $\sigma_{i-1}\setminus \sigma$ is a linear isomorphism. By the inductive hypothesis on $\tilde{h}_{i-1}$, we conclude that $\tilde{h}_{\mathfrak{z}} \in C^{\infty, \theta}_{\mathcal{W}_a^s}$. Therefore $\rho_l^{-1}(b) \tilde{h}_{\mathfrak{z}} \circ \rho(b) \in C^{\infty, \theta}_{\mathcal{W}^s_a}$.
\par By writing $\tilde{h}_i = \tilde{h}_{\mathfrak{z}} + \tilde{h}^{\perp}_{\mathfrak{z}}$, we may write each higher order term in (\ref{equation_hi}) as a Lie bracket monomial of $\rho_l^{-1}(b) \tilde{h}_{\mathfrak{z}} \circ \rho(b)$, $\rho_l^{-1}(b) \tilde{h}^{\perp}_{\mathfrak{z}} \circ \rho(b)$ and $\Psi_i$. Because $\rho_l^{-1}(b)\tilde{h}^{\perp}_{\mathfrak{z}} \circ \rho(b) \in \sigma_i\setminus \sigma_{i-1}$, every Lie bracket monomial containing $\rho_l^{-1}(b)\tilde{h}^{\perp}_{\mathfrak{z}} \circ \rho(b)$ vanishes. Note that $\rho_l^{-1}(b) \tilde{h}_{\mathfrak{z}} \circ \rho(b)$ and $\Psi_i$
are both in $C^{\infty, \theta}_{\mathcal{W}^s_a}$, we have that the sum of higher order terms is in $C^{\infty, \theta}_{\mathcal{W}^s_a}$. Therefore 
$$\tilde{h}_i = \rho_l^{-1}(b) \tilde{h}_i \circ \rho(b) + \tilde{\Psi}_i$$
where $\tilde{\Psi}_i = \Psi_i + [ \text{ higher order terms } ]$ is in $C^{\infty, \theta}_{\mathcal{W}_a^s}$.  
Let $B :=\rho_l(b)$. Since $B^{-1}$ is contracting on $\sigma$, $\|B^{-i}\|$ is uniformly bounded when restricted to $\sigma_i\setminus \sigma$.
By Proposition \ref{prop_linear_equation}, we conclude that $\tilde{h}_i \in C^{\infty, \theta}_{\mathcal{W}^s_a}$.
\end{proof}
\par For $\overline{x}\in \s(M)$, let $\mathcal{W}^{ss}_a(\overline{x})$ denote the topological submanifold of $M(\overline{x})$ passing through $\overline{x}$ defined by 
$$\mathcal{W}^{ss}_{a} (\overline{x}) := \phi^{-1} (\phi(\overline{x}) G^{ss}_a).$$ 
Obviously every $\mathcal{W}^{ss}_a(\overline{x})$ is contained in a $\mathcal{W}^s_a$ leaf. Recall that 
$$H_{\overline{x}}(\overline{y}) = h^{-1}(\overline{x}) p(\overline{x},\overline{y}) h(\overline{y}).$$
The following proved in \cite{hertz_wang2014} gives the local description of $\mathcal{W}^{ss}_a(\overline{x})$:
\begin{lemma}[see~{\cite[Lemma 4.1]{hertz_wang2014}}]
\label{lemma_local_strong_stable}
Inside $\mathcal{W}^s_a(\overline{x})$, $\mathcal{W}^{ss}_a(\overline{x})$ is locally defined by the equation
$$(H_{\overline{x}} (\overline{y}))_V = e.$$
\end{lemma} 

\par Note that by Proposition \ref{proposition_group_coordinate},
$$(H_{\overline{x}} (\overline{y}))_V = h_V(\overline{y}) ((h(\overline{x}) p(\overline{x},\overline{y})^{-1})^{-1}_{s})_V.$$ 
By Proposition \ref{prop_smooth_hV}, $h_V(\overline{y})$ is $\smooth$ when restricted to $\mathcal{W}^s_a(\overline{x})$. Combined with $((h(\overline{x}) p(x,y)^{-1})^{-1}_{s})_V$ is $\smooth$ in $\overline{y}$, this implies that $\overline{y} \mapsto H_{\overline{x}}(\overline{y})$ is $\smooth$ in small neighborhoods of $\overline{x}$ in $\mathcal{W}^s_a(\overline{x})$. Moreover, since partial derivatives of $(H_{\overline{x}}(\overline{y}))_V$ along $\mathcal{W}^s_a(\overline{x})$ are polynomial combinations of 
$\partial^k_{\mathcal{W}^s_a} h_V(\overline{y})$ and $\partial^k_{\mathcal{W}^s_a} ((h(\overline{x}) p(x,y)^{-1})^{-1}_{s})_V$, we conclude that all partial derivatives $\partial^k_{\mathcal{W}^s_a} |_{\overline{y} = \overline{x}} (H_{\overline{x}} (\overline{y}))_V$
are $\theta$-\holder continuous in $\overline{x}$.

\par Our aim is to show that $\mathcal{W}^{ss}_a$ defines a \holder foliation with $\smooth$ leaves. By \cite[Corollary 4.3]{hertz_wang2014}, to show the smoothness of every $\mathcal{W}^{ss}_a(\overline{x})$, it suffices to show that for any $\overline{x} \in \s(M)$, the map $(H_{\overline{x}} (\overline{y}))_V$ is regular in $\overline{y}$ at $\overline{y} = \overline{x}$. We will modify the argument by Rodriguez-Hertz and Wang to prove the result.
\par Let $A$ be the set of points $\overline{x} \in \s(M)$ where $(H_{\overline{x}}(\overline{y}))_V$ is singular at $\overline{x}$. We want to show that $A$ is empty.
\begin{lemma}[see~{\cite[Lemma 4.4]{hertz_wang2014}}]
$A$ is closed and invariant under the $\Z^k$-action $\rho$.
\end{lemma}
\begin{proof}
Since $D_{\mathcal{W}^s_a}|_{\overline{y} = \overline{x}} (H_{\overline{x}}(\overline{y}))_V $ depends continuously on $\overline{x}$, and since being singular is a closed condition, we conclude that $A$ is closed. 
\par Let us show that $A$ is $\rho$-invariant. Fix $a' \in \Z^k$. For $\overline{x} \in \s(M)$ and $\overline{y} \in \mathcal{W}^s_a(\overline{x})$, by the definition of $H_{\overline{x}}(\overline{y})$, we have 
$$ \begin{array}{rcl} \phi(\rho(a')\overline{x}) H_{\rho(a')\overline{x}} (\rho(a')\overline{y}) & = & \phi(\rho(a')\overline{y}) = \rho_l(a') \phi(\overline{y}) \\  & = &  \rho_l(a')(\phi(\overline{x}) H_{\overline{x}}(\overline{y})) \\
& = & \rho_l(a')\phi(\overline{x}) \rho_l(a')H_{\overline{x}}(\overline{y})   \\
& = & \phi(\rho(a') \overline{x}) \rho_l(a')H_{\overline{x}}(\overline{y}) . \end{array}$$
Since $a'$ is fixed and both $H_{\overline{x}}(\overline{y})$ and $H_{\rho(a')\overline{x}} (\rho(a')\overline{y}) $ are close to $e$,
we conclude that 
$$H_{\rho(a') \overline{x}} (\rho(a') \overline{y}) = \rho_l(a') H_{\overline{x}}(\overline{y}).$$ 
By projecting the above equation to $V$, we get that 
$$(H_{\rho(a') \overline{x}} (\rho(a') \overline{y}))_V = \rho_l(a') (H_{\overline{x}}(\overline{y}))_V.$$
Since $\mathcal{W}^s_a$ is $\rho$-invariant, we have that 
$$D_{\mathcal{W}^s_a}|_{\overline{y} = \rho(a')\overline{x}} (H_{\rho(a') \overline{x}}(\overline{y}) )_V = D\rho_l(a')|_{\sigma} (D_{\mathcal{W}^s_a} |_{\overline{y} = \overline{x}} (H_{\overline{x}} (\overline{y}))_V) (D_{\rho(a') \overline{x}} \rho(a')|_{\mathcal{E}^s_a})^{-1}.$$
Since $D\rho_l(a') |_{\sigma} $ and $D_{\rho(a') \overline{x}} \rho(a')|_{\mathcal{E}^s_a}$ are both regular, we conclude that $D_{\mathcal{W}^s_a} |_{\overline{y} = \rho(a')\overline{x}} (H_{\rho(a') \overline{x}}(\overline{y}))_V$ is singular if and only if 
$D_{\mathcal{W}^s_a} |_{\overline{y} = \overline{x}} (H_{\overline{x}} (\overline{y}))_V$ is so. In other words, $\overline{x} \in A$ if and only if $\rho(a') \overline{x} \in A$. Since $a' \in \Z^k$ is chosen arbitrarily, we conclude that $A$ is $\rho$-invariant.
\end{proof}

\begin{proposition}
\label{prop_A_empty}
$A = \emptyset$.
\end{proposition}
\begin{proof}
For contradiction, we assume that $A$ is not empty. Then $A$ supports an ergodic $\rho$-invariant probability measure $\overline{\mu}$. By Oseledets' multiplicative ergodic theorem adapted to $\Z^k$-actions (cf. \cite[Proposition 2.1]{kalinin_sadovskaya2006} and \cite[Proposition 4.5]{hertz_wang2014}), there are finitely many linear functionals $\overline{\chi} \in (\R^k)^{\ast}$, an $\rho$-invariant 
subset $A' \subset A$ with $\overline{\mu}(A') =1$ and a $\rho$-invariant measurable splitting
$$ \T_{\overline{z}} (M) = \bigoplus_{\overline{\chi}} E^{\overline{\chi}}_{\overline{\mu}} (\overline{z})$$
over $\overline{z} \in A'$ such that for all $a \in \Z^k$ and $v \in E^{\overline{\chi}}_{\overline{\mu}}$,
$$\lim_{k \rightarrow \infty} \frac{\log \|D \rho(k a) v \|}{k} = \overline{\chi}(a).$$
By Pesin's strong stable manifold theorem (cf. \cite{ruelle1979}, \cite[Theorem 3.2]{hertz2007} and \cite[Lemma 4.6]{hertz_wang2014}), one can modify $A'$ such that for all $\overline{z} \in A'$ and $a \in \Z^k$, there are unique manifolds $\mathcal{W}^s_{a, \overline{\mu}}(\overline{z})$ and $\mathcal{W}^u_{a, \overline{\mu}} (\overline{z})$ respectively tangent to the stable and unstable distributions 
$$\begin{array}{ccc}E^s_{a, \overline{\mu}} := \bigoplus_{\overline{\chi}(a) <0} E^{\overline{\chi}}_{\overline{\mu}}, & \text{ and } &  E^u_{a, \overline{\mu}} := \bigoplus_{\overline{\chi}(a) >0} E^{\overline{\chi}}_{\overline{\mu}}, \end{array}$$
moreover, near $\overline{z}$, $\mathcal{W}^s_{a, \overline{\mu}}(\overline{z})$ is given by the set of $\overline{y} \in M(\overline{z})$ (recall that $M(\overline{z})$ denotes a {\em manifold slice} passing through $\overline{z}$, see Definition \ref{def:manifold-slice}) satisfying that $\dist (\overline{y}, \overline{z}) <\epsilon$ for some $\epsilon >0$ depending on $\overline{z}$, and 
$$\limsup_{k \rightarrow \infty} \log \dist (\rho(ka)\overline{y} , \rho(ka)\overline{z}) \leq \max \{\overline{\chi}(a): \overline{\chi}(a)<0\}.$$
Since $\phi$ is a bi-\holder conjugacy between $\rho$ and $\rho_l$, it is easily seen that if $\overline{\chi}(a) \neq 0$ for all Lyapunov functional $\overline{\chi}$ for $\overline{\mu}$, then for $\square = s, u$, 
$$\dim E^{\square}_{a, \overline{\mu}} = \dim \mathfrak{g}^{\square}_{a},$$
and 
$$\phi(\mathcal{W}^{\square}_{a, \overline{\mu}} (\overline{z})) = \phi(\overline{z}) G^{\square}_a ,$$
see \cite[Proposition 3.1 and Corollary 3.3]{hertz2007} and \cite[Lemma 4.7]{hertz_wang2014} for details.
\par Let us define {\bf coarse Lyapunov distributions} of $\overline{\mu}$ as follows:
\begin{equation}
\label{equation_coarse_lyapunov_distribution}
E^{[\overline{\chi}]}_{\overline{\mu}} := \bigoplus_{\overline{\chi}' = c \overline{\chi}, c >0} E^{\overline{\chi}'}_{\overline{\mu}}.
\end{equation}
By \cite[Lemma 4.9]{hertz_wang2014}, the {\bf coarse Lyapunov subspaces} in Definition \ref{def_coarse_lyapunov} and the {\bf coarse Lyapunov distributions} defined above are in one-to-one correspondence to each other. A pair of corresponding {\bf coarse Lyapunov subspace} and {\bf coarse Lyapunov distribution} have the same dimension and proportional {\bf coarse Lyapunov exponents}. 
\par According to the decomposition 
$$\mathfrak{g}^s_a = \sigma \oplus \mathfrak{g}^{ss}_a,$$
we decompose $E^s_{a, \overline{\mu}}$ as 
$$E^s_{a, \overline{\mu}} = E^{V}_{a, \overline{\mu}} \oplus E^{ss}_{a, \overline{\mu}},$$
where $E^V_{a, \overline{\mu}} = E^{[\overline{\chi}]}_{a, \overline{\mu}}$ is the {\bf coarse Lyapunov distribution} corresponding to $\sigma = \sigma^{[\chi^l]}$, and $E^{ss}_{a, \overline{\mu}}$ is the direct sum of {\bf coarse Lyapunov distributions} $E^{[\overline{\chi}_1]}_{a, \overline{\mu}}$ corresponding to $\sigma^{[\chi^l_1]} \subset \mathfrak{g}^{ss}_a$. Moreover, there exists $\lambda >0$ such that for any $\xi >0$, we can choose $a \in \Sigma \cap \mathcal{C}_0 $ such that 
 \begin{equation}
 \begin{array}{ll}
 \chi^l_1(a) \in (- \xi \|a\|, 0) & \text{if } \sigma^{\chi^l_1} \subset \sigma ; \\
 \chi^l_1(a) < -\lambda \|a\| & \text{if } \sigma^{\chi^l_1} \subset \mathfrak{g}^{ss}_a ; \\
 \chi^l_1(a)>0 & \text{if } \sigma^{\chi^l_1} \subset \mathfrak{g}^u_a .
 \end{array}
 \end{equation}
 Then by the correspondence, the {\bf Lyapunov exponents} of $\rho(a)$ with respect to $E^u_{a, \overline{\mu}}$, $E^V_{a, \overline{\mu}}$ and $E^{ss}_{a, \overline{\mu}}$ are in the intervals $(0,\infty)$, $(-\kappa \eta \|a\|, 0)$ and $(- \infty, - \kappa^{-1} \lambda \|a\|)$, respectively, where $\kappa>1$ denotes a constant determined by the \holder index $\theta$ of the conjugacy $\phi$.
Since the stable and unstable foliations remain the same after $a \in \Sigma \cap \mathcal{C}_0$ is changed, we have that $\rho(a)$ is still uniformly hyperbolic. 
\par Let $N $, $N^u$, $N^s$, $N^V $ and $N^{ss}$ denote the dimension of $M$, $E^u_{a, \overline{\mu}}$, $E^s_{a, \overline{\mu}}$, $E^V_{a, \overline{\mu}}$ and $E^{ss}_{a, \overline{\mu}}$ respectively. Let $B_{\R^n}(r)$ denote the ball in $\R^n$ centered at $\mathbf{0}$ with radius $r$. By the work of Ledrappier and Young \cite{ledrappier_young1985},  at $\overline{\mu}$-almost every $\overline{z}$, the {\bf Lyapunov decomposition} at $\overline{z}$: 
$$\T_{\overline{z}}(M) = E^u_{a, \overline{\mu}} \oplus E^{V}_{a,\overline{\mu}} \oplus E^{ss}_{a, \overline{\mu}}$$
 can be locally foliated. To be specific, there exists a $\rho(a)$-invariant subset $A'' \subset A$ with $\overline{\mu}(A'') =1$ and a measurable function $l: A'' \rightarrow (1, \infty)$ such that for $\overline{z} \in A''$, there is an embedding 
 $$\Phi_{\overline{z}}: B_{\R^N}(l^{-1}(\overline{z})) \rightarrow M(\overline{z}), $$
 satisfying several nice properties:
 \begin{enumerate}
 \item $\Phi_{\overline{z}}(\mathbf{0}) = \overline{z}$, and $D |_{\mathbf{x} = \mathbf{0}} \Phi_{\overline{z}} $ sends the splitting 
 $\R^{N^{ss}} \oplus \R^{N^V} \oplus \R^{N^u} $ to $E^{ss}_{a, \overline{\mu}} \oplus E^{V}_{a, \overline{\mu}} \oplus E^{u}_{a, \overline{\mu}}$.
 \item Set $f_{\overline{z}} = \Phi^{-1}_{\rho(a)\overline{z}} \circ \rho(a) \circ \Phi_{\overline{z}}$ and 
 $f^{-1}_{\overline{z}} =\Phi^{-1}_{\rho(-a) \overline{z}} \circ \rho(-a) \circ \Phi_{\overline{z}}$, then for $\square = u , V,$ or $ss$, and all non-zero vector $v \in \R^{N^{\square}}$, 
 $$\log \frac{\|(D |_{\mathbf{x} = \mathbf{0}}f_{\overline{z}}) v\|}{\|v\|} \in (\lambda^{\square}_{-} - \epsilon , \lambda^{\square}_{+} + \epsilon),$$
 where $\lambda^{\square}_{-}$ and $\lambda^{\square}_{+}$ denote respectively the smallest and largest {\bf Lyapunov exponents} of $\rho(a)$ on $E^{\square}_{a, \overline{\mu}}$ and the constant $\epsilon >0$ can be chosen arbitrarily small by modifying $A''$. 
 \item For $\mathbf{x}, \mathbf{x}' \in B_{\R^N}(l(\overline{z})^{-1})$, $c < \frac{\|\mathbf{x} - \mathbf{x}'\|}{\dist (\Phi_{\overline{z}}(\mathbf{x}) , \Phi_{\overline{z}}(\mathbf{x}') )} < l(\overline{z})$.
 \end{enumerate} 
 The foliations $\mathcal{W}^{\square}_{a}$ (for $\square= u, s $) near $\overline{z}$ can be translated to the corresponding foliations on $\R^N$ by $\Phi_{\overline{z}}^{-1}$. By \cite[Lemma 8.2.3 \& 8.2.5 ]{ledrappier_young1985}, there exists $\tau \in (0, 1/2)$, such that for all $\overline{y} \in \mathcal{W}^s_a(\overline{z}) \cap \Phi_{\overline{z}} (B_{\R^N} (\tau l^{-1}(\overline{z})))$, the image of the strong stable foliation $\mathcal{W}^{ss}_a(\overline{y})$ passing through $\overline{y}$ under $\Phi_{\overline{z}}^{-1}$ is the graph of a map $g_{\overline{z}, \overline{y}} : \R^{N^{ss}} \rightarrow \R^{N^V}$. This implies that $(H_{\overline{z}})_V \circ \Phi_{\overline{z}}$ is constant along the graph of $g_{\overline{z}, \overline{y}} $ for all $\overline{y} \in B_{\R^{N^s}} (l^{-1}(\overline{z}))$, cf. \cite[Lemma 4.12]{hertz_wang2014}. Moreover, by \cite[\S 8.3]{ledrappier_young1985}, for all $\overline{z} \in A''$, there exists a bi-Lipschitz homeomorphism 
 $$\pi_{\overline{z}}: \Phi_{\overline{z}}^{-1} (\mathcal{W}^s_a(\overline{z}) \cap \Phi_{\overline{z}}(B_{\R^N} (\tau l^{-1}(\overline{z}))) ) \rightarrow U^{N^s} \subset \R^{N^s}$$ such that for all $\overline{y} \in \mathcal{W}^s_a(\overline{z}) \cap \Phi_{\overline{z}}(B_{\R^N} (\tau l^{-1}(\overline{z})))$, $\pi_{\overline{z}}$ maps the graph of $g_{\overline{z}, \overline{y}}$ to a piece of a hyperplane parallel to $\R^{N^{ss}}$. Put 
 $P_{\overline{z}} : = \Phi_{\overline{z}} \circ \pi^{-1}_{\overline{z}}$, then $(H_{\overline{z}})_V \circ P_{\overline{z}}$ is constant along hyperplanes parallel to $\R^{N^{ss}}$.
 \par Now we are ready to give the contradiction, cf. \cite[Lemma 4.15 \& 4.16]{hertz_wang2014}.
 \par On the one hand, we claim that for every $\overline{z} \in A''$, there exists a decreasing sequence of bounded open neighborhoods 
 $B_{k, \overline{z}} \subset \mathcal{W}^{s}_{a}(\overline{z})$ of $\overline{z}$ such that 
 $$\lim_{k \rightarrow \infty} \frac{\Vol_{\phi(\overline{z})G^s_a } (\phi(B_{k, \overline{z}}))}{\Vol_{\mathcal{W}^s_a(\overline{z})} (B_{k,\overline{z}})} =0,$$
 where $\Vol_{\phi(\overline{z})G^s_a }$ and $\Vol_{\mathcal{W}^s_a(\overline{z})}$ denote the volume forms of the induced Riemannian metrics on 
 $\phi(\overline{z})G^s_a $ and $\mathcal{W}^s_a(\overline{z})$ respectively.
 \par In fact, we may choose $\delta_0 >0$ and a neighborhood $B_{\overline{z}} \subset \mathcal{W}^s_a(\overline{z})$ such that 
 $$B_{\R^{N^{ss}}} (\delta_0) \times B_{\R^{N^V}}(\delta_0) \subset P_{\overline{z}}^{-1}(B_{\overline{z}}).$$
 Fix a decreasing sequence $\{\delta_k >0 : k \in \N\}$ approaching $0$ as $k \rightarrow \infty$. For each $k \in \N$, define 
 $$B_{k, \overline{z}} := P_{\overline{z}} \left( B_{\R^{N^{ss}}} (\delta_0) \times B_{\R^{N^V}} (\delta_k) \right).$$
 Since $P_{\overline{z}}$ is bi-Lipschitz, to show the claim, it suffices to show that
 $$\lim_{k \rightarrow \infty} \frac{ \Vol_{\phi(\overline{z}) G^s_a} (\phi \circ P_{\overline{z}} (B_{\R^{N^{ss}}} (\delta_0) \times B_{\R^{N^V}} (\delta_k)))}{\Vol_{\R^{N^s}} (B_{\R^{N^{ss}}} (\delta_0) \times B_{\R^{N^V}} (\delta_k))} = 0.$$
 The denominator is of order $O(\delta_k^{N^V})$. Let us analyze the numerator. Our aim is to show that the numerator is of order $o(\delta^{N^V}_k)$ (cf. Notation \ref{notation}).  Since $G^s$ is decomposed as $G^{ss} \cdot V$ and since $V$ normalizes $G^{ss}$, we have that $\dd \Vol_{G^s} = \dd \Vol_{G^{ss}} \cdot \dd \Vol_V$. It is easy to see that the $G^{ss}$-projection of $\phi \circ P_{\overline{z}} (B_{\R^{N^{ss}}} (\delta_0) \times B_{\R^{N^V}} (\delta_k))$ is uniformly bounded, so to show that 
 $$\Vol_{\phi(\overline{z}) G^s_a} (\phi \circ P_{\overline{z}} (B_{\R^{N^{ss}}} (\delta_0) \times B_{\R^{N^V}} (\delta_k))) = o(\delta_k^{N^V}),$$
 it suffices to show that 
 $$\Vol_V \left(  (H_{\overline{z}})_V \circ P_{\overline{z}} (B_{\R^{N^{ss}}}(\delta_0) \times B_{\R^{N^V}} (\delta_k) ) \right) = o(\delta_k^{N^V}).$$
 We have seen that $(H_{\overline{z}})_V \circ P_{\overline{z}}$ only depends on the second coordinate, so 
  $$(H_{\overline{z}})_V \circ P_{\overline{z}} (B_{\R^{N^{ss}}}(\delta_0) \times B_{\R^{N^V}} (\delta_k) ) = (H_{\overline{z}})_V \circ P_{\overline{z}} (B_{\R^{N^{ss}}}(\delta_k) \times B_{\R^{N^V}} (\delta_k) ).$$
  Note that $P_{\overline{z}}$ is Lipschitz, we have that 
  $$P_{\overline{z}} (B_{\R^{N^{ss}}}(\delta_k) \times B_{\R^{N^V}} (\delta_k) ) \subset B_{\mathcal{W}^s_a(\overline{z})}(\overline{z}, C \delta_k)$$ for a constant $C = C(\overline{z})$, here $B_{\mathcal{W}^s_a(\overline{z})} (\overline{z}, r)$ denotes the ball in $\mathcal{W}^s_a(\overline{z})$ centered at $\overline{z}$ of radius $r$. Then to prove the claim, it is enough to show that 
  $$\Vol_V \left( (H_{\overline{z}})_V (B_{\mathcal{W}^s_a (\overline{z})} (\overline{z}, \delta)) \right) = o(\delta^{N^V}) \text{ as } \delta \rightarrow 0.$$
  This is true since by our hypothesis, for $\overline{z} \in A'' \subset A$,  $D_{\mathcal{W}^s_a} |_{\overline{y} = \overline{z}} (H_{\overline{z}} (\overline{y}) )_V: \mathcal{E}^s_a(\overline{z}) \rightarrow  \sigma $ has rank less than $N^V = \dim \sigma$.
  
  \par On the other hand, we claim that for every $\overline{z} \in \s(M)$, there exists a positive continuous function $J_{\overline{z}}$ such that 
  $$\phi_{\ast} \dd \Vol_{\mathcal{W}^s_a (\overline{z})} = J_{\overline{z}} \dd \Vol_{\phi(\overline{z}) G^s_a} .$$
Note that this claim contradicts the previous one. Therefore, to complete the proof, it suffices to prove this claim.
\par Recall that $\mu$ denotes the Haar measure on the solenoid $\s(M)$ and $\tilde{\mu} := \phi^{-1}_{\ast} \mu$. Let $\tilde{\mu}^s_{\overline{z}}$ denote the condiitonal measure of $\tilde{\mu}$ along the stable leaf $\mathcal{W}^s_a(\overline{z})$. Then the Radon-Nykodim derivative of $\dd \tilde{\mu}^s_{\overline{z}}$ with respect to $\dd \Vol_{\mathcal{W}^s_a(\overline{z})}$ can be calculated as follows:
\begin{equation}
\label{radon_nykodim}
\frac{\dd \tilde{\mu}^s_{\overline{z}}}{ \dd \Vol_{\mathcal{W}^s_a (\overline{z})}} (\overline{y}) = r_{\overline{z}}(\overline{y}) := \prod_{k \geq 0} \frac{J^s\rho(a) (\rho(ka)(\overline{y}))}{J^s \rho(a) (\rho(ka)(\overline{z}))}, \text{ for } \overline{y} \in \mathcal{W}^s_a(\overline{z}),\end{equation}
 where $J^s \rho(a)$ denotes the Jacobian of $\rho(a)$ along the stable bundle. Since $J^s \rho(a)$ is \holder continuous, and $\rho(ka)$ contracts the stable leaf $\mathcal{W}^s_a(\overline{z})$ exponentially with $k$, we have that the infinite product (\ref{radon_nykodim}) is uniformly convergent. Therefore $r_{\overline{z}}(\overline{y})$ is uniformly bounded for $\overline{y}$ in a neighborhood $B \subset \mathcal{W}^s_a(\overline{z})$ of $\overline{z}$. The same argument shows that $r_{\overline{y}}(\overline{z})$ is also uniformly bounded. Note that $r_{\overline{y}}(\overline{z}) = r_{\overline{z}}^{-1}(\overline{y})$, we conclude that $r_{\overline{z}}(\overline{y})$ is also bounded away from zero. Note that for almost every $\overline{z}$, 
 $$ \dd \tilde{\mu}^s_{\overline{z}} = r_{\overline{z}} \dd \Vol_{\mathcal{W}^s_a(\overline{z})},$$
 and
 $$ \dd \phi_{\ast} \tilde{\mu}^s_{\overline{z}} = \dd \Vol_{\phi(\overline{z}) G^s_a} .$$
 Define 
 $$J_{\overline{z}}(\overline{y}) :=  \frac{1}{r_{\overline{z}} (\phi^{-1} (\overline{y}))},$$
 then it is easy to see that $J_{\overline{z}}$ is a continuous positive function, and we have that for almost every $\overline{z}$,
 $$\phi_{\ast} \dd \Vol_{\mathcal{W}^s_a(\overline{z})} = \phi_{\ast} (r_{\overline{z}}^{-1} \dd \tilde{\mu}^s_{\overline{z}}) = J_{\overline{z}} \phi_{\ast} \dd \tilde{\mu}^s_{\overline{z}} = J_{\overline{z}} \dd \Vol_{\phi(\overline{z}) G^s_a}.$$
 Because both sides of the above equation are continuous and agree on a full measure subset, we have that the equation holds for every 
 $\overline{z} \in \s(M)$. This proves the claim, and hence get the contradiction.
 \par This completes the proof.     \end{proof}

It immediately follows that:
\begin{corollary}
$\mathcal{W}^{ss}_a$ defines a \holder foliation consisting of smooth leaves. 
\end{corollary}
\begin{proof}
By Proposition \ref{prop_A_empty}, for every $\overline{z} \in \s(M)$, the function $(H_{\overline{z}})_V$ (restricted to $\mathcal{W}^s_a(\overline{z})$) is smooth and regular at $\overline{z}$. Since $\mathcal{W}^{ss}_a(\overline{z})$ is locally defined by $(H_{\overline{z}})_V = e$, we conclude that $\mathcal{W}^{ss}_{a}(\overline{z})$ is smooth. 
\par This completes the proof of Proposition \ref{prop_A_empty}.
\end{proof}
\begin{remark}
We call $\mathcal{W}^{ss}_a$ the strongly stable foliation of $\rho(a)$.
\end{remark}

This result combined with a result of Ma\~n\'e (see \cite{mane1977}) implies the following:
\begin{proposition}
\label{prop_uniform_hyperbolic}
For any {\bf Weyl chamber} $\mathcal{C}$ adjacent to $\mathcal{C}_0$, and for any $a' \in \mathcal{C}$, $\rho(a')$ is uniformly hyperbolic.
\end{proposition}
\begin{proof}
See the proof of \cite[Proposition 4.17]{hertz_wang2014}.
\end{proof}
\begin{remark}
It follows from the proposition that for any {\bf Weyl chamber} $\mathcal{C}$ and any $a' \in \mathcal{C}$, $\rho(a')$ is uniformly hyperbolic.
\end{remark}
Combined with the discussion at the beginning of the section, Propostion \ref{prop_uniform_hyperbolic} implies the following:
\begin{theorem}
\label{thm:lyapunov-foliation}
For every {\bf coarse Lyapunov exponent} $[\chi]$ of $\rho$, the corresponding {\bf coarse Lyapunov distribution} $E^{[\chi]}$ admits a \holder foliation consisting of $\smooth$ leaves.
\end{theorem}

\section{Regularity of the conjugacy}
\label{regularity}
In this section we will prove Theorem \ref{goal_thm} when $\dim M \geq 5$. As we discussed in \S \ref{preliminaries}, we write $\phi(\overline{z}) = \overline{z} h(\overline{z})$. Then it suffices to show that $h$ is $\smooth$. 
\par We follow the process described in \S \ref{outline_of_the_proof}. Recall that we follow Notation \ref{notation_lyapunov_exponent} and Notation \ref{notation_coarse_lyapunov_subspace} to denote {\bf coarse Lyapunov exponents} and {\bf coarse Lyapunov subgroups}. By Theorem \ref{thm:lyapunov-foliation}, every {\bf coarse Lyapunov exponent} $[\chi]$ of $\rho$ admits a \holder foliation with $\smooth$ leaves. Let us call it the {\bf coarse Lyapunov foliation} associated with $[\chi]$.
\subsection{Case of tori}
\par In this subsection, we assume that $M$ is a torus. By our discussion in \S \ref{outline_of_the_proof}, it is enough to show that 
$$h_V = \sum_{i=0}^{\infty} \rho_l(a)^{-i} \Phi \circ \rho(a)^i$$
is $\smooth$ for any {\bf coarse Lyapunov subgroup} $V$ associated with a {\bf coarse Lyapunov exponent} $[\chi^l]$.
\par By \cite[Corollary 8.4]{kalinin_fisher_spatzier2013}, to show $h_V$ is $\smooth$, it suffices to show that 
$h_V \in C^{\infty, \theta, \ast}_{\mathcal{V}'}(M)$ for every {\bf coarse Lyapunov foliation} $\mathcal{V}'$.
\begin{proposition}
\label{main_prop_torus}
For any {\bf coarse Lyapunov foliation} $\mathcal{V}'$, $h_V \in C^{\infty, \theta, \ast}_{\mathcal{V}'}$, i.e., for every $k \in \N$, 
$\partial^k_{\mathcal{V}'} h_V \in (C^{\theta} (M))^{\ast}$.
\end{proposition}
\begin{proof}
Suppose $\mathcal{V}'$ is associated with the {\bf coarse Lyapunov exponent} $[\chi']$. For any fixed small constant $\xi >0$,  choose $a \in \Sigma$ such that $\chi^l(a) >0$ and $|\chi'(a)| < \xi \|a\|$ for all $\chi' \in [\chi']$. Then we have that 
$$h_V =   \sum_{i=0}^{\infty} \rho_l(a)^{-i} \Phi \circ \rho(a)^i ,$$
where $\Phi = (Q'_a)_V$. Then repeating the argument in the proof of Proposition \ref{prop_linear_equation} to $\mathcal{V}'$, one concludes that
$h_V \in C^{\infty, \theta, \ast}_{\mathcal{V}'}(M)$.
\end{proof}
\begin{proof}[Proof of Theorem \ref{goal_thm} for the case of tori]
By \cite[Corollary 8.4]{kalinin_fisher_spatzier2013}, Proposition \ref{main_prop_torus} implies that $h_V \in C^{\infty}(M)$. This proves that $h$ is $\smooth$ since $h$ is the sum of $h_V$'s where $V$ runs over all {\bf coarse Lyapunov subgroups} of $\rho_l$.
\end{proof}

\subsection{General case}
Now we deal with the general case. As described in \S \ref{outline_of_the_proof}, we consider the derived series of $N$:
$$N = N_0 \supset N_1 \supset \cdots \supset N_{r-1} \supset N_r = \{0\}.$$
Note that $\rho_l$ preserves $N_i$ for each $i=0,1,\dots, r$.
\par We will prove Theorem \ref{goal_thm} by showing the following stronger statement.
\begin{proposition}
\label{prop_goal}
For $i=0,1,\dots, r$, let $h: M \rightarrow N_i(\R)$ (regarded as a map defined on $\s(M)$) be a $\theta$-\holder map. Suppose for all $a \in \Z^k$, there exists a $\smooth$ map $Q_a : \s(M) \rightarrow N_i(\R)$, such that
 $$h(\overline{z}) = Q_a(\overline{z}) \rho_l^{-1}(a) h\circ \rho(a)(\overline{z}).$$
 Then $h$ is $\smooth$.
\end{proposition}
\begin{proof}[Proof of Theorem \ref{goal_thm} from Proposition \ref{prop_goal}]
We  conclude the proof by applying the proposition with $i=0$.
\end{proof}
\begin{proof}[Proof of Proposition \ref{prop_goal}]
We will prove the statement by induction on $i$.
\par When $i=r$, the statement is trivial. Suppose the statement holds for $i+1$. We need to prove the statement for 
$i$.
\par 
Let $\mathfrak{n}_i$ denote the Lie algebra of $N_i$. Then $\mathfrak{n}_i(\R)$ admits the following splitting
 $$\mathfrak{n}_i(\R) = \mathfrak{n}'_i \oplus  \mathfrak{n}_{i+1} (\R) ,$$
 where $\mathfrak{n}_{i+1}$ is the Lie algebra of $N_{i+1}$ and $\mathfrak{n}'_i$ is a subspace of $\mathfrak{n}_i(\R)$. Then we  write $h =  h_0 h_1$
 where $h_1 \in N_{i+1}(\R)$ and $h_0 = \exp \mathfrak{n}'_i$. To show $h$ is $\smooth$, it suffices to show that both $h_0$ and $h_1$ are $\smooth$.
 \par We first prove that $h_0$ is $\smooth$. Let $G_i := N_{i+1}\setminus N_i$ and $\pi_i: N_i(\R) \rightarrow G_i(\R) = \R^{l_i} $ denote the projection from $N_i(\R)$ to $G_i(\R)$. Let $\bar{h} = \pi_i \circ h$. Then $h_0$ is $\smooth$ if and only if $\bar{h}$ is $\smooth$ since $\mathfrak{n}'_i$ can be identified with the Lie algebra of $G_i(\R)$. Let $\bar{\rho}_l$ denote the induced action of $\rho_l$ on $\s(G_i(\Z) \setminus G_i(\R))$. Then it is easy to see that $\bar{h}$ satisfies 
 $$\bar{h} = \bar{\rho}_l^{-1} (a) \bar{h} \circ \rho(a) + \bar{Q}_a,$$
 where $\bar{Q}_a = \pi_i(Q_a) $. Then from our proof for the case of tori, we conclude that $\bar{h} $ is $\smooth$, and thus 
 $h_0$ is $\smooth$.
 \par Let us prove that $h_1$ is $\smooth$. In fact, since $h = h_0 h_1$, we have that for any $a \in \Z^k$,
 $$ h_0 (\overline{z})  h_1(\overline{z}) =  Q_a (\overline{z}) \rho^{-1}_l (a) (h_0 \circ \rho(a) (\overline{z})) \cdot \rho^{-1}_l(a)(h_1\circ \rho(a) (\overline{z})).$$
 Therefore 
 $$h_1(\overline{z}) = Q'_{a}(\overline{z}) \rho_l^{-1}(a)(h_1 \circ \rho(a)(\overline{z}) ),$$
 where 
 $$Q'_a (\overline{z}) =  (h_0(\overline{z}))^{-1} Q_a(\overline{z}) \rho^{-1}_l (a) (h_0 \circ \rho(a) (\overline{z})).$$
 Since $\rho_l$ preserves $N_{i+1}$, we have that $h_1(\overline{z}) $ and $ \rho_l^{-1}(a)(h_1 \circ \rho(a)(\overline{z}) )$ are both in 
 $N_{i+1}(\R)$. Therefore $Q'_a(\overline{z})$ is also in $N_{i+1}(\R)$. Moreover $Q_a$ is $\smooth$ since $Q_a$ and $h_0$ are $\smooth$. By
 our inductive hypothesis, we conclude that $h_1$ is $\smooth$.
 \par This completes the proof.

\end{proof}
\section{The low dimensional cases}  
\label{lowdimension}
\par In this section we consider the case $\dim M \leq 4$. Note that in this case the fundamental group of $M$ must be abelian, i.e., $M$ is homeomorphic to a torus.
\begin{proof}[Proof of Theorem \ref{goal_thm} for $\dim M \leq 4$ ]
\par For $\dim M \leq 3$, $M$ does not have any exotic differential structures (cf. \cite{rado1925} and \cite{moise1952}). Therefore, the argument for $\dim M \geq 5$ applies to this case.
\par For $\dim M =4$, we have that the homeomorphism
$$\phi \times \phi: M \times M \rightarrow \mathcal{M} \times \mathcal{M} $$
conjugates the action $\rho \times \rho$ to $\rho_l \times \rho_l$. It is easy to verify that $\rho \times \rho$ has no rank-one factor. By Theorem \ref{goal_thm}, we conclude that $\phi \times \phi$ is $\smooth$, since $\dim (M \times M) \geq 5$. This happens only if $\phi$ itself is $\smooth$.
\par This completes the proof.
\end{proof}
\medskip

\bibliography{expanding.bib}{}

\begin{thebibliography}{FKS13}

\bibitem[BG06]{bombieri_gubler2006}
Enrico Bombieri and Walter Gubler.
\newblock {\em Heights in {D}iophantine geometry}, volume~4 of {\em New
  Mathematical Monographs}.
\newblock Cambridge University Press, Cambridge, 2006.

\bibitem[CG90]{corwin-greenleaf}
Lawrence~J. Corwin and Frederick~P. Greenleaf.
\newblock {\em Representations of nilpotent {L}ie groups and their
  applications. {P}art {I}}, volume~18 of {\em Cambridge Studies in Advanced
  Mathematics}.
\newblock Cambridge University Press, Cambridge, 1990.
\newblock Basic theory and examples.

\bibitem[Dek12]{dekimpe2012}
Karel Dekimpe.
\newblock What an infra-nilmanifold endomorphism really should be{$\ldots$}.
\newblock {\em Topol. Methods Nonlinear Anal.}, 40(1):111--136, 2012.

\bibitem[FJ78]{farrelljones1978}
F.~T. Farrell and L.~E. Jones.
\newblock Examples of expanding endomorphisms on exotic tori.
\newblock {\em Invent. Math.}, 45(2):175--179, 1978.

\bibitem[FKS11]{kalinin_fisher_spatzier2011}
David Fisher, Boris Kalinin, and Ralf Spatzier.
\newblock Totally nonsymplectic {A}nosov actions on tori and nilmanifolds.
\newblock {\em Geom. Topol.}, 15(1):191--216, 2011.

\bibitem[FKS13]{kalinin_fisher_spatzier2013}
David Fisher, Boris Kalinin, and Ralf Spatzier.
\newblock Global rigidity of higher rank {A}nosov actions on tori and
  nilmanifolds.
\newblock {\em J. Amer. Math. Soc.}, 26(1):167--198, 2013.
\newblock With an appendix by James F. Davis.

\bibitem[Gro81]{gromov1981}
Mikhael Gromov.
\newblock Groups of polynomial growth and expanding maps.
\newblock {\em Inst. Hautes {\'E}tudes Sci. Publ. Math.}, (53):53--73, 1981.

\bibitem[GS14]{gorodnik_spatzier2014}
Alexander Gorodnik and Ralf Spatzier.
\newblock Exponential mixing of nilmanifold automorphisms.
\newblock {\em Journal d'Analyse Math{\'e}matique}, 123(1):355--396, 2014.

\bibitem[GS15]{gorodnik_spatzier_acta}
Alexander Gorodnik and Ralf Spatzier.
\newblock Mixing properties of commuting nilmanifold automorphisms.
\newblock {\em Acta Math.}, 215(1):127--159, 2015.

\bibitem[GT12]{green_tao2012}
Ben Green and Terence Tao.
\newblock The quantitative behaviour of polynomial orbits on nilmanifolds.
\newblock {\em Ann. of Math. (2)}, 175(2):465--540, 2012.

\bibitem[GT14]{GreenTao-erratum}
Ben Green and Terence Tao.
\newblock On the quantitative distribution of polynomial nilsequences---erratum
  [mr2877065].
\newblock {\em Ann. of Math. (2)}, 179(3):1175--1183, 2014.

\bibitem[HW14]{hertz_wang2014}
Federico~Rodriguez Hertz and Zhiren Wang.
\newblock Global rigidity of higher rank abelian anosov algebraic actions.
\newblock {\em Inventiones mathematicae}, 198(1):165--209, jan 2014.

\bibitem[KK01]{kalinin_katok2001}
Boris Kalinin and Anatole Katok.
\newblock Invariant measures for actions of higher rank abelian groups.
\newblock In {\em Proceedings of Symposia in Pure Mathematics}, volume~69,
  pages 593--638. Providence, RI; American Mathematical Society; 1998, 2001.

\bibitem[KS96]{katok_spatzier1996}
Anatole Katok and Ralf~J Spatzier.
\newblock Invariant measures for higher-rank hyperbolic abelian actions.
\newblock {\em Ergodic Theory and Dynamical Systems}, 16(04):751--778, 1996.

\bibitem[KS06]{kalinin_sadovskaya2006}
Boris Kalinin and Victoria Sadovskaya.
\newblock Global rigidity for totally nonsymplectic anosov $\z^k$ actions.
\newblock {\em Geometry \& Topology}, 10(2):929--954, 2006.

\bibitem[Lei05]{leibman2005}
A~Leibman.
\newblock Pointwise convergence of ergodic averages for polynomial sequences of
  translations on a nilmanifold.
\newblock {\em Ergodic Theory and Dynamical Systems}, 25(01):201--213, 2005.

\bibitem[LY85]{ledrappier_young1985}
Fran{\c{c}}ois Ledrappier and L-S Young.
\newblock The metric entropy of diffeomorphisms: part ii: relations between
  entropy, exponents and dimension.
\newblock {\em Annals of Mathematics}, pages 540--574, 1985.

\bibitem[Ma{\~n}77]{mane1977}
Ricardo Ma{\~n}{\'e}.
\newblock Quasi-anosov diffeomorphisms and hyperbolic manifolds.
\newblock {\em Transactions of the American Mathematical Society},
  229:351--370, 1977.

\bibitem[Moi52]{moise1952}
Edwin~E. Moise.
\newblock Affine structures in {$3$}-manifolds. {V}. {T}he triangulation
  theorem and {H}auptvermutung.
\newblock {\em Ann. of Math. (2)}, 56:96--114, 1952.

\bibitem[MQ01]{margulis_qian2001}
Gregory~A Margulis and Nantian Qian.
\newblock Rigidity of weakly hyperbolic actions of higher real rank semisimple
  lie groups and their lattices.
\newblock {\em Ergodic Theory and Dynamical Systems}, 21(01):121--164, 2001.

\bibitem[Par69]{parry1969}
William Parry.
\newblock Ergodic properties of affine transformations and flows on
  nilmanifolds.
\newblock {\em American Journal of Mathematics}, pages 757--771, 1969.

\bibitem[Rad25]{rado1925}
Tibor Rad{\'o}.
\newblock {\"U}ber den begriff der riemannschen fl{\"a}che.
\newblock {\em Acta Litt. Sci. Szeged}, 2:101--121, 1925.

\bibitem[Rag72]{raghunathan1972}
Madabusi~Santanam Raghunathan.
\newblock {\em Discrete subgroups of Lie groups}, volume~29.
\newblock Springer New York, 1972.

\bibitem[RH07]{hertz2007}
Federico Rodriguez~Hertz.
\newblock Global rigidity of certain abelian actions by toral automorphisms.
\newblock {\em J. Mod. Dyn.}, 1(3):425--442, 2007.

\bibitem[RT05]{rauch_taylor2005}
Jeffrey Rauch and Michael Taylor.
\newblock Regularity of functions smooth along foliations, and elliptic
  regularity.
\newblock {\em Journal of Functional Analysis}, 225(1):74--93, 2005.

\bibitem[Rue79]{ruelle1979}
David Ruelle.
\newblock Ergodic theory of differentiable dynamical systems.
\newblock {\em Publications Math{\'e}matiques de l'Institut des Hautes
  {\'E}tudes Scientifiques}, 50(1):27--58, 1979.

\bibitem[Shu69]{Shub1969}
Michael Shub.
\newblock Endomorphisms of compact differentiable manifolds.
\newblock {\em American Journal of Mathematics}, pages 175--199, 1969.

\bibitem[Shu70]{shub1970}
Michael Shub.
\newblock Expanding maps.
\newblock In {\em Global Analysis (Proc. Sympos. Pure Math., Vol. XIV,
  Berkeley, Calif., 1968)}, pages 273--276, 1970.

\bibitem[Spa16]{brinprize2015}
Ralf Spatzier.
\newblock On the work of {R}odriguez {H}ertz on rigidity in dynamics.
\newblock {\em J. Mod. Dyn.}, 10:191--207, 2016.

\bibitem[Sta99]{starkov1999}
AN~Starkov.
\newblock The first cohomology group, mixing, and minimal sets of the
  commutative group of algebraic actions on a torus.
\newblock {\em Journal of Mathematical Sciences}, 95(5):2576--2582, 1999.

\bibitem[Vee86]{Veech}
William~A. Veech.
\newblock Periodic points and invariant pseudomeasures for toral endomorphisms.
\newblock {\em Ergodic Theory Dynam. Systems}, 6(3):449--473, 1986.

\bibitem[Wal70]{walters1970}
Peter Walters.
\newblock Conjugacy properties of affine transformations of nilmanifolds.
\newblock {\em Theory of Computing Systems}, 4(4):327--333, 1970.

\bibitem[Wil74]{williams1974}
Robert~F Williams.
\newblock Expanding attractors.
\newblock {\em Publications Math{\'e}matiques de l'IH{\'E}S}, 43:169--203,
  1974.

\end{thebibliography}
\bibliographystyle{alpha}

\end{document}